\documentclass[aos]{imsart}

\RequirePackage{amsthm,amsmath,amsfonts,amssymb}
\RequirePackage[numbers]{natbib}

\startlocaldefs
\usepackage{xcolor}
\usepackage{longtable}
\usepackage{graphicx}

\newcommand{\do@not@check@textbox@alter}{}
\usepackage[
left=2.5cm,
right=2.5cm,
top=2.5cm,
bottom=2.5cm
]{geometry}

\newcommand{\pl}{\parallel}

\newtheorem{definition}{Definition}

\newtheorem{example}{Example}
\newtheorem{theorem}{Theorem}
\newtheorem{corollary}{Corollary}
\newtheorem{lemma}{Lemma}

\newtheorem{assumption}{Assumption}
\newtheorem{proposition}{Proposition}
\newcommand{\nl}{\newline}
\newcommand{\openr}{\hbox{${\rm I\kern-.2em R}$}}
\newcommand{\openn}{\hbox{${\rm I\kern-.2em N}$}}

\newcommand{\Var}{\operatorname{Var}}

\usepackage{booktabs} % for ate simulations

\endlocaldefs

\begin{document}
	
	\begin{frontmatter}
		\title{Highly Adaptive Empirical Risk Minimization \\ with Principal Components}
		%\title{A sample article title with some additional note\thanksref{T1}}
		\runtitle{MEIXIDE, WANG, SCHULER and VAN DER LAAN}
		%\thankstext{T1}{A sample of additional note to the title.}
		
		\begin{aug}
			%%%%%%%%%%%%%%%%%%%%%%%%%%%%%%%%%%%%%%%%%%%%%%%
			%% Only one address is permitted per author. %%
			%% Only division, organization and e-mail is %%
			%% included in the address.                  %%
			%% Additional information such as            %%
			%% identifying the corresponding author must %%
			%% be included in in the Acknowledgments     %%
			%% section if necessary.                     %%
			\author[B]{\fnms{CARLOS}~\snm{GARCÍA MEIXIDE}},
			\author[A]{\fnms{MINGXUN}~\snm{WANG}},\phantom{sss}
			\author[A]{\fnms{ALEJANDRO}~\snm{SCHULER}}
			\and
			\author[A]{\fnms{MARK J. }~\snm{VAN DER LAAN}}
			%%%%%%%%%%%%%%%%%%%%%%%%%%%%%%%%%%%%%%%%%%%%%%
			%% Addresses                                %%
			%%%%%%%%%%%%%%%%%%%%%%%%%%%%%%%%%%%%%%%%%%%%%%
			\address[A]{University of California, Berkeley}
			
			\address[B]{ICMAT, National Research Council of Spain}
		\end{aug}
		
		\begin{abstract}
			{The Highly Adaptive Lasso (HAL) delivers unprecedented guarantees in nonparametric minimum loss estimation under minimal smoothness assumptions, such as dimension-free minimax optimal rates. However, the practical use of HAL has been severely limited by its exponentially growing computationally prohibitive indicator basis expansion in moderate to high dimensions. Existing screening strategies drastically reduce this dimension but lack any theoretical justification. We introduce the Principal Component Highly Adaptive (PC-HA) family of estimators, which for the first time provide a principled and theoretically valid dimension reduction. We establish formal results on the score equations solved by these PC-HA estimators, allowing to transfer plug-in efficiency and pointwise asymptotic normality results from HAL to these PC-HA estimators, under comparable complexity control.} 

		\end{abstract}
		
	\end{frontmatter}
	%%%%%%%%%%%%%%%%%%%%%%%%%%%%%%%%%%%%%%%%%%%%%%
	%% Please use \tableofcontents for articles %%
	%% with 50 pages and more                   %%
	%%%%%%%%%%%%%%%%%%%%%%%%%%%%%%%%%%%%%%%%%%%%%%
	%\tableofcontents
	
	%%%%%%%%%%%%%%%%%%%%%%%%%%%%%%%%%%%%%%%%%%%%%%
	%%%% Main text entry area:

	\section{Introduction}\label{section1}
	
	The Highly Adaptive Lasso (HAL) estimator of a multivariate target function  is defined by minimizing  an empirical risk over  high dimensional linear models in spline basis functions under an $L_1$-norm constraint on the coefficients \cite{VanDerLaan:2015:GTMLE,benkeser2016highly,van2017generally}. 
	This $L_1$-norm equals the $k$-th order sectional variation norm of the function \cite{ gill1995inefficient,van2023higher} ({see Supplement \ref{app:secv} for a rigorous definition}), which represents a bound on the sectional variation norm of the $k$-th order derivatives of the  function. These higher order HAL estimators  incorporate smoothness by using a higher order spline basis \cite{Neuhaus:1971:WCSMP}. A variety of  theoretical results have been established about HAL such as plug-in efficiency, pointwise asymptotic normality,  and dimension free rates of convergence  up till $\log n$-factors, under minimal conditions \citep{bibaut2019fast,fang2021multivariate,Gao:2013:BEHDD}. Constraining the $L_2$-norm yields the Highly Adaptive Ridge estimator \cite{schuler2024highlyadaptiveridge} (HAR) that has the same rate of convergence as HAL, but selects much less sparse working models.
	
	The initial design matrices proposed for HAL can have as many as $ {N=n(2^d-1)}$ columns for the zero order splines---and even higher $N$ for the $k$-th order HAL---even though the final fit has a much smaller set of non-zero coefficients. Due to memory constraints, one is typically forced to use screeners that drastically reduce the dimension. However, current choices of such screeners lack theoretical grounding so that the resulting quasi-HAL might not have the key theoretical properties of HAL. 
	
	We address this problem through a principal component-based screener. Let $H$ be the $n\times N$ design matrix implied by the $N$ basis functions evaluated at $n$ values $x^n$ that might correspond with the $n$ observations. For example, in the standard regression context these might be the $n$ covariate vector observations.
	{In this article, we develop a principal component highly adaptive (PC-HA) estimator family that first determines the $n$ leading eigenvectors of the $N\times N$ inner-product matrix $H^{T}H$ 
		(or equivalently, the right singular vectors corresponding to the largest squared singular values of the $n \times N$ design matrix $H$). 
		We then use the $n$-dimensional linear model spanned by the corresponding PC-basis functions as the starting point for estimation, that are themselves linear combinations of the original $N$ basis functions.} This represents now an $n$-dimensional submodel of the $N$ dimensional working model typical of HAL.
	We can use techniques from multivariate analysis to compute these $n$ PCs from the  eigenvalue decomposition of the $n\times n$-analogue $HH^T$.
	We now consider estimators that minimize the empirical risk over this $n$-dimensional linear model under a bound on a norm of its coefficients, where this  bound is selected with cross-validation. In particular, we consider the Euclidean $L_2$-norm, $L_1$-norm and the $L_1$-norm on the $N$ dimensional coefficient vector implied by the $n$ dimensional vector of coefficients. These correspond with fitting a standard ridge regression; standard lasso regression; and a generalized lasso regression \cite{Tibshirani:2011:SPGL}, respectively, applied to this $n$ dimensional linear model. We refer to these estimators as PC-HAR, PC-HAL, and PC-HAGL, respectively.
	We provide general conditions under which the resulting PC-HA converges at the same rate as HAL, and we present a particular gradient descent algorithm for computing PC-HAGL (as an alternative of the generalized Lasso software). 
	
	\subsection{Statistical estimation problem}
	Suppose we observe $n$ i.i.d. copies of $O\sim P_0\in  {\cal M}$ for a specified statistical model ${\cal M}$. Let $D^{(0)}([0,1]^d)$ be the space of càdlàg functions only assuming
	that the sectional variation norm is bounded. Let $\Psi:{\cal M}\rightarrow D^{(0)}([0,1]^d)$ be a functional parameter that maps a data distribution into a function that is an element of càdlàg function space on a unit cube with finite zero order sectional variation norm. 
	Let $L(\psi)(O)$ be a loss function so that $\psi_0 \equiv \Psi(P_0)=\arg\min_{\psi\in D^{(0)}([0,1]^d)}P_0L(\psi)$. 
	For example, we might be in the least squares regression setting  with $O=(X,Y)$, $\Psi(P_0)(X)=E_0(Y\mid X)$, and
	the squared error loss $L(\psi)(O)=(Y-\psi(X))^2$, or density estimation setting with $\Psi(P_0)(Y\mid X)=p(Y\mid X)$ with log-likelihood loss $L(\psi)(O)=-\log \psi(Y\mid X)$.
	Let $d_0(\psi,\psi_0)=P_0L(\psi)-P_0L(\psi_0)$ be the loss-based dissimilarity which typically behaves as a square of an $L^2$-norm of $\psi-\psi_0$.
	
	\subsection{Sectional variation norm, smoothness classes and their spline representation}
	Let $D^{(k)}([0,1]^d)\subset D^{(0)}([0,1]^d)$ be the $k$-th order smoothness class defined in \cite{van2023higher}, $k=0,\ldots$ in which functions have a finite $k$-th order sectional variation norm defined as a sum of variation norms of measures generated by sections of $k$-th order Radon-Nikodym derivatives of the function w.r.t. Lebesgue measure. $D^{(k)}_M([0,1]^d)$ represents the subset in which it is assumed that this $k$-th order sectional variation norm is bounded by $M$.  In particular, $D^{(0)}_M([0,1]^d)$ is the class of real valued càdlàg functions with sectional variation norm bounded by $M$, where for a function $f$, we define the subset-specific sections 
	$f_s(x(s))=f(0(-s),x(s))$ that is only a function of $x(s)\in  {[0,1]^{\lvert s\rvert}}$ corresponding with a subset $s\subset\{1,\ldots,d\}$; the variation norm of $f_s$ is $ {\lVert f_s\rVert_v=\int \lvert df_s(u)\rvert}$; and the (zero-order) sectional variation norm is defined as $ {\lVert f\rVert_v^*=\lvert f(0)\rvert +\sum_{s\subset \{1,\ldots,d\}}\lVert f_s\rVert_v}$, the sum of the variation norms of the sections of $f$. The first order sectional variation norm is essentially defined as a sum of sectional variation norms of Radon-Nikodym derivatives $df_s/d\mu(s)$ w.r.t. Lebesgue measure across all subsets $s$, and so on (see Supplement \ref{app:secv}). {A representation theorem of $D^{(k)}([0,1]^d)$ as a closure of the linear span of linear combinations of tensor products of $k$-th order spline basis functions has been established (Theorem 1 in \cite{van2023higher}) for $k=0,1,\ldots$ ; where the $L_1$-norm of the coefficients equals the $k$-th order sectional variation norm. The theorem establishes that any function in the class $D^{(k)}([0,1]^d)$ can be exactly represented as an infinite additive model decomposition involving linear combinations of $\leq k$-th order tensor product spline basis functions.} 
	%\label{theoremkthorderspline}

	{Crucially, the inherent smoothness of the function, quantified by its $k$-th order sectional variation norm,  is the $L_1$-norm of the coefficients in this spline representation. For $k=0$, the latter involves} a sum over subsets $s_0\subset\{1,\ldots,d\}$ of an $s_0$-specific generalized cumulative distribution function (a difference of cumulative distribution functions that are bounded from above by some finite constant \cite{Hildebrandt1963}), where each one is approximated by linear combination of zero order splines. The first order spline representation  involves a sum over $(s_0,s_1)$, $s_0\subset\{1,\ldots,d\}$, $s_1\subset s_0$ of $(s_0,s_1)$-specific first order primitives of a generalized cumulative distribution function, where the latter is approximated by linear combination of first order splines. The general $k$-th order spline representation of $D^{(k)}([0,1]^d)$ represents each function as a sum over nested subsets $\bar{s}(k)=(s_0,\ldots,s_{k})$-specific $k$-th order primitive of a generalized CDF, where each one can be approximated by linear combinations of $k$-th order splines. As shown in \cite{van2023higher}, putting restrictions on $\bar{s}(k)$ results in all kinds of additive submodels of $D^{(k)}([0,1]^d)$ with interpretations in terms of higher order derivatives being zero on the 0-edges of the cube $[0,1]^d$.
	
	\subsection{Approximation results and covering numbers for smoothness classes}
	It has been shown (Theorem 2 in \cite{van2023higher} with the base $L_2$-approximation rate of a $d$-variate càdlàg function using $J$ zero-order splines being $r(d, J) \sim J^{-1}(\log J)^{2(d-1)}$) that the class $D^{(k)}_M([0,1]^d)$ (and its many submodels) can be approximated by a {finite dimensional} linear span of $O(J)$ such $k$-th order spline basis functions with approximation error bounded by $O^+(1/J^{k+1})$, where we use the notation $O^+(r(n))$ for a term that is bounded by a constant times $r(n)$ up till a $\log n$-factor. {This result demonstrates that increasing the assumed smoothness ($k$) allows the use of smaller finite-dimensional models ($J$) to maintain a fixed level of bias, which is essential for statistical efficiency.}

	This approximation is obtained by approximating each of the $\bar{s}(k)$-specific $k$-th order primitive of a CDF in the representation of a function in $D^{(k)}([0,1]^d)$ by a linear span of $J$ $k$-th order splines. 
	For $k\geq 1$, this approximation holds  w.r.t. uniform norm while for $k=0$ this approximation holds w.r.t. $L^2$-norm (the latter result is a consequence of the covering number for multivariate cumulative distribution functions in the literature \cite{fang2021multivariate}). These approximation results for $D^{(k)}_M([0,1]^d)$ imply bounds on covering numbers of this class as presented in Lemma 29 in \cite{van2023higher}. The covering number for $D^{(k)}_M([0,1]^d)$ is of the same order as the covering number of the class of $d$-variate $k$-th order primitives $\mu^k(F)$ of a cumulative distribution $F$ defined recursively as  $\mu^0(F)=\int_{(0,x]}dF(u)$; $\mu^{m+1}(F)=\int_{(0,x]}\mu^m(x)dx$, $m=0,\ldots,k-1$. These bounds imply bounds on the supremum norm entropy integrals $J_{\infty}(\delta,D^{(k)}_M([0,1]^d))\sim \delta^{(2k+1)/(2k+2)}$ up till $\log \delta$-factors, $k=1,\ldots$, where $J_{\infty}(\delta,{\cal F})=\int_0^{\delta}\sqrt{\log N(\epsilon,{\cal F}, {\lVert\cdot\rVert_{\infty}})} d\epsilon$. A general rate of convergence proof based on the bounds on the entropy integral establishes that an MLE over $D^{(k)}_M([0,1]^d)$ converges in loss-based dissimilarity to the true $\psi_0$ at rate $O^+(n^{-2k^*/(2k^*+1)})$, $k^*=k+1$, $k=0,
	\ldots$ as we show in {Theorem \ref{rate}}. 
	
	%\label{lemmasupnormcoveringnumber}
	
	\subsection{Finite dimensional linear starting working models}
	Let $D^{(k)}({\cal R}_N)$ be a finite dimensional submodel of $D^{(k)}([0,1]^d)$ spanned by $\{\phi_j: 1 \leq j \leq N\}$, a finite set of  $N$ $k$-th  order spline basis functions, chosen so that $\psi_{N,0}=\arg\max_{\psi\in D^{(k)}({\cal R}_N)}P_0L(\psi)$ satisfies
	$d_0(\psi,\psi_{N,0})=O_P(n^{-2k^*/(2k^*+1)})$, $k^*=k+1$. The HAL literature has proposed such a basis with (maximal)  $N=n (2^d-1)$ for the zero order HAL: specifically, in the regression context {$D^{(k)}({\cal R}_N)$ has been defined} as the linear span of the zero-order splines $\phi_{X_i(s)}^0(x)=I(x\geq X_i(s))$ across the knot-points $\{X_i(s): i=1,\ldots,n,s\subset\{1,\ldots,d\}$). This corresponds with using $n$ zero order splines to approximate the $s$-specific CDF across all $s\subset\{1,\ldots,d\}$ in the zero-order spline representation of $D^{(0)}([0,1]^d)$. Similar sets of even larger size (approximately $\sim n2^{kd}$  were also proposed for the $k$-th order HAL, $k \geq 1$. In other words, each $\bar{s}(k)$-specific $k$-th order primitive of the CDF is approximated by $n$ $k$-th order splines implied by $n$ zero-order splines for the CDF.
	The total number of elements $\bar{s}(k)$ is of the order $2^{kd}$, even though it is smaller in size due to the sets being nested.
	
	{The approximation error of these initial working  models of size $\sim n$ has been proven to be $d_0(\psi_{N,0},\psi_0)=O_P(n^{-1})$ under a absolute continuity condition on the true $\psi_0$ with respect to $P_0$ (see Lemma 23 proof, Appendix B2 in \cite{van2023higher}). This rate is found as a direct consequence of using the empirical measure's support as the knotpoint set, rather than relying on the general $k$-th order spline uniform approximation Theorem 2 in \cite{van2023higher}.}Therefore, these  initial starting model could be chosen much smaller since this particular choice generally results in a  $O_P(n^{-1/2})$ uniform approximation of $D^{(k)}([0,1]^d)$, which is more precise than the needed $n^{-k^*/(2k^*+1)}$.
	In fact, our results show that there exists  initial working models (outcome blind) that are of size $\sim n^{1/(2k^*+1)} 2^{kd}$, replacing the $n$ splines for each $\bar{s}(k)$-specific section by only $n^{1/(2k^*+1)}$ splines. 
	Such an initial set is completely determined by finding $J$ zero order splines to approximate a càdlàg cumulative distribution function at $L^2$-rate $O^+(1/J)$, where this result was shown by \cite{fang2021multivariate}.
	However, even though such sets of size $n^{1/(2k^*+1)}$ exist, the proof in \cite{fang2021multivariate} is not constructive so that it takes some effort to determine sensible sets of that size. HAL avoids having to do the careful in choosing an initial working model, but just relies on an initial working model that is rich enough. 
	Such working models allows one to just focus on this finite dimensional analogue of $D^{(k)}([0,1]^d)$. One obtains corresponding submodels by restricting the choices $\bar{s}(k)$.

	\subsection{Theoretical properties of HAL}
	
	Standard HAL fits $\psi_0$ with \\ $\psi_n=\psi_{n,C_n}$ with  $\psi_{n,C}=\arg\min_{\psi\in D^{(k)}_C({\cal R}_N)}P_n L(\psi)$ for each $C$ and selects $C$ with a cross-validation selector $C_n$ or an undersmooth selector. 
	If we define $\psi_{n,\beta}=\sum_{j\in {\cal R}_N}\beta(j)\phi_j$, we have $\psi_{n,C}=\psi_{n,\beta_n(C)}$ with
	$\beta_n(C)=\arg\min_{\beta, {\lVert \beta\rVert_1\leq C}}P_n L(\psi_{N,\beta})$ and $ {\lVert \beta\rVert_1=\sum_{j\in {\cal R}_N}\lvert \beta(j)\rvert}$. Due to the covering number of $D^{(k)}_M([0,1]^d)$ presented in \cite{van2023higher},  a  general proof for rate of convergence of MLEs over function classes establishes that
	$d_0(\psi_n,\psi_0)=O^+_P(n^{-2k^*/(2k^*+1)})$ in {Theorem 2} as long as the starting model $D^{(k)}({\cal R}_N)$ has at least this same approximation error of $D^{(k)}([0,1]^d)$ and $\psi_0\in D^{(k)}([0,1]^d)$. 
	
	Beyond the above mentioned rate of convergence in loss-based dissimilarity, many other theoretical properties have been established for HAL {based on its MLE property}, such as plug-in asymptotic normality and efficiency for pathwise differentiable features of the target function $\psi_0$; point-wise asymptotic normality of HAL as an estimator of $\psi_0$; simultaneous confidence bands for $\psi_0$. These results all rely on HAL solving a large collection of score equations that asymptotically spans a space of scores, essentially making HAL behave as a parametric MLE  for a data adaptively chosen parametric working model that has negligible approximation error. 
	
	\subsection{Screeners for HAL and discrete super-learning}
	Beyond computational reasons for using a screener, from a statistical  finite sample performance perspective it is better to start with a smaller starting model, as long as its bias is negligible relative to the standard error.  Therefore, one generally wants to use a discrete super learner with a set of submodels of $D^{(k)}({\cal R}_N)$, thereby creating an adaptive HAL, where one can also vary $k$ so that we do not rely on knowing the true smoothness of $\psi_0$. Instead of just indexing HAL by  subspaces, one could also couple HAL with a screener such a machine learning algorithm whose result gets mapped into a set of spline basis functions for initializing HAL. For example, we could first run MARS or gradient boosting and use its fits to obtain a sensible initial working model for HAL. Such algorithms are already based on the outcome so that cross-validation has to take into account the coupled algorithm that first uses a screener and then runs a HAL.

	However, for such data adaptive dimension reductions it is theoretically challenging to argue that the resulting initial working models  yield a good enough approximation of a HAL to which the above rate of convergence applies. 
	The oracle property of the cross-validation selector guarantees an asymptotically optimal choice among the candidate HALs. Nonetheless, if the library of the discrete super-learner only includes heavy dimension reductions before running HAL, then it is hard to argue theoretically that the discrete HAL-based super-learner satisfies the same theoretical properties as HAL over function spaces that contain the true target function such as the  zero-order HAL based on the $ {n(2^d-1)}$-dimensional working model.
	
	Any current implementation of HAL either finds an initial starting model of relatively low dimension (avoiding the $2^{kd}$-multiplier) or uses a computationally fast screener,  or it has to deal with a serious computational challenge of implementing HAL with  very large initial working models. The latter is particularly problematic as $d$ increases
	
	\subsection{{Historical context and literature review}}
	
	Nonparametric estimation has a long and distinguished history in statistics. Early breakthroughs include the celebrated kernel estimators of
	\cite{Davis2011,parzen1962estimation}, which established
	nonparametric density and regression estimation as a core statistical discipline,
	leading to major theoretical refinements such as optimal minimax rates under
	smoothness constraints \citep{stone1982optimal}. Spline-based procedures soon
	followed, bringing structured shape constraints, enhanced approximation
	properties, and efficient computation
	\citep{wahba1990spline}. Methods such as Random Forests \cite{breiman2001random} and Bayesian Additive Regression Trees (BART) \cite{chipman2010bart} have demonstrated strong predictive performance by utilizing ensembles of regression trees. Deep neural networks with ReLU activation functions can achieve optimal minimax convergence rates for nonparametric regression tasks when the target function possesses a compositional structure \cite{schmidt2020nonparametric}. 
	
	Parallel developments in empirical process theory and sieve estimation
	\citep{grenander1981abstract, birge} enabled M-estimators over rich
	function classes to achieve statistical efficiency. A critical milestone was the
	advent of penalization and sparsity for infinite-dimensional models: seminal
	contributions such as the Lasso \citep{tibshirani1996lasso}, trend-filtering and
	generalized Lasso \citep{mammen1997locally, Tibshirani:2011:SPGL}, and
	boosting \citep{buhlmann2003boosting} demonstrated that adaptivity to unknown
	structure could be achieved through convex regularization. These ideas have been
	central in the growth of high-dimensional inference, especially within the
	Le Cam–van der Vaart asymptotic framework
	\citep{bickel2009simultaneous,van2000asymptotic}. 
	
	Beyond spline-based methods, adaptive basis construction emerged as a powerful
	principle for high-dimensional nonparametrics. Classical decision trees and
	multivariate adaptive regression splines (MARS)
	\citep{breiman1984cart,friedman1991multivariate} introduced data-driven
	partitioning rules that automatically adapt to local structure, laying early
	foundations for the sparse, piecewise modeling strategy later adopted by HAL.
	Parallel lines of work showed that sparsity in rich dictionaries could enable
	near-optimal approximation: wavelet shrinkage
	\citep{donoho1994ideal,donoho1995noising} and neural network approximation
	results \citep{barron1993universal} formalized the idea that adaptivity to
	unknown smoothness can be achieved through selective basis activation. These
	ideas crystallized within the asymptotic decision-theoretic program
	initiated by Le~Cam \citep{lecam1973convergence}, while
	powerful empirical process techniques and concentration inequalities
	\citep{massart2000some} enabled rigorous statistical guarantees for
	penalization-based estimators in infinite-dimensional models. Critical
	complexity results for bounded-variation classes
	\citep{birman1967piecewise,vandegeer2000empirical} clarify why variation-norm
	regularization controls function class entropy. Functional estimation through variation norms marked a conceptual turning point:
	estimators controlling total variation — including the Grenander estimator in
	monotone models \citep{grenander1956theory} — reveal that strong shape
	constraints can yield dimension-free convergence guarantees. 
	
	Significant progress has been made in high-dimensional inference \cite{guo2022}, correcting for regularization bias and hidden confounding in large-scale regression settings. However, such methods are predominantly predicated on linear working models, inherently limiting their capacity. Recent theoretical advances have provided convergence guarantees for spectral methods in linear regression, specifically using a \textit{Feature Space Decomposition} to analyze the population excess risk. This framework relates to our PC-HAR estimator, as it effectively corresponds to a spectral method where the filter function selects the leading eigenspace of the covariance structure. Matching upper and lower bounds for such estimators have been  {established} in \cite{lecue}, justifying the truncation of the tail spectrum—the "noise absorption part"—when the signal is well-aligned with the leading components. 
	
	\section{A novel  {fast-to-compute} PC-screener that preserves HAL theory}
	In this article we are particularly interested in the challenging but common  case that the number of basis functions  $N$ is much larger than sample size $n$ as generally occurs when $d$ is large (e.g. $d\geq 10$).
	In particular, we aim to tackle an approximation of the  HAL-implementation with the full set of  $ {N=n(2^{d}-1)}$ basis functions for the zero order HAL, and, similarly for the $k$-th order HAL.
	In this article, we propose a fast to compute dimension reduction based on finding $n$ eigenvectors $E_N(\cdot,m)$, $m=1,\ldots,n$,  implied by the  $n\times N$-design matrix $H=(\phi_j(x^n):j\in {\cal R}_N)$ for a given set $x^n=(x_i, i=1,\ldots n)$ implied by the observations $O_i$, $i=1,\ldots,n$: i.e. $H(i,j)=\phi_j(x_i)$. Here $H:\openr^N\rightarrow\openr^n$, where we could select
	a specific inner product $\langle\cdot,\cdot\rangle_n$ in the image space , while we always use  the standard inner product $\langle\cdot,\cdot\rangle_N$ in the domain: the definition of the adjoint $H^{\top}$ of $H$ depends on this choice of inner product. A standard choice is $\langle \phi_1,\phi_2\rangle_n=1/n \sum_i \phi_1\phi_2(x_i)$, but we can also make it resemble the empirical covariance of scores $S_{Q_n}(\phi_1)$ and $S_{Q_n}(\phi_2)$ with 
	$S_{Q_n}(\phi)=d/dQ_nL(Q_n)(\phi)$ the score at a good estimator $Q_n$. These $\{E_N(\cdot,m):m=1,\ldots,n\}$ are defined as  eigenvectors of the eigenvalue decomposition of $N\times N$ matrix $H^{\top}H$ with non-zero eigenvalues.
	These eigenvectors generate $n$ basis functions $\tilde{\phi}_m=\sum_{j\in {\cal R}_N}E_N^n(j,m)\phi_j$, a linear combination of the $N$ $k$-th order spline basis functions $\{\phi_j:j\in {\cal R}_N\}$,  $m=1,\ldots,n$. These basis functions $\tilde{\phi}_m$ are orthogonal  w.r.t the inner product. $\langle \tilde{\phi}_{m_1},\tilde{\phi}_{m_2}\rangle_n$: for example, for the standard inner product, we have $ 1/n \sum_i \tilde{\phi}_{m_1}(x_i)\tilde{\phi}_{m_2}(x_i)=0$ for $m_1\not =m_2$. In addition, we also have that
	$E_N^n(\cdot,m)$, $m=1,\ldots,n$, are orthonormal w.r.t. standard inner product $\langle\cdot,\cdot\rangle_N$.

	\subsection{The PC-working model}
	These $n$ basis functions $\tilde{\phi}_m$ span an  $n$-dimensional working model $D^{(k)}({\cal E}_n)\subset D^{(k)}({\cal R}_N)$: each element is parametrized as $\theta_{n,\alpha}=\sum_{m=1}^n  {\alpha(m)}\tilde{\phi}_m$. 
	
	We then compute a regularized MLE over this working model using a constraint on its coefficient vector $\alpha$ that can be chosen with cross-validation. In particular, we can apply standard main term LASSO (i.e {.}, $ {\lVert\alpha\rVert_1}$),  a ridge regression (i.e., $ {\lVert \alpha\rVert_2}$), or a generalized LASSO constraining the $L_1$-norm $ {\lVert \beta(\alpha)\rVert_1}$ of the $N$-dimensional  coefficient vector $\beta(\alpha)$ for the underlying spline basis functions implied by the $n$-dimensional coefficient vector $\alpha$.  If the user is satisfied  {with} only need {ing} the values $\tilde{\phi}_m(x^n)$, $m=1,\ldots,n$, then 
	this can be achieved without having to store the $n\times N$-matrix, thereby truly allowing massive $N$. If one wants to know the actual coefficients of the eigenvectors, then one needs to store this matrix so that $N$ needs to be restricted to control memory  {requirements}.  We refer to these estimators as principal component Highly Adaptive estimators, abbreviated as  PC-HA-estimators: PC-HAL; PC-HAR; PC-HAGL respectively.
	
	We will show that this dimension reduction still allows for the desired rate of convergence, and presumably  gain in finite samples by carrying out an outcome blind (i.e., non-informative) effective {unsupervised} dimension reduction before fitting the target function.

	\section{PC Highly Adaptive estimators: PC-HAGL, PC-HAL and PC-HAR}\label{Section2}
	
	{We start this section discussing} a so called  \textit{sufficiency assumption}  on $x^n$: 
	\begin{assumption}[Sufficiency assumption on $x^n$]\label{sufficiencyassumption}
		Let
		$x_i\in [0,1]^d$, $i=1,\ldots,m_n$, be $m_n$ values chosen so that $L(\psi_{N,\beta})(O_i)$, $i=1,\ldots,n$, only depends on $\psi_{N,\beta}$ through 
		a design matrix $H_{m_n\times N}(i,j)=\phi_j(x_i)$, $i=1,\ldots,m_n$, $j\in {\cal R}_N$. Let $x^n=(x_i: i=1,\ldots,m_n)$ be the $m_n$-dimensional vector of these values. 
	\end{assumption}
	In various of our applications $m_n=n$ with one point $x_i$ for each observation $O_i$,  but in some cases we want to define repeated points for each $O_i$ in our sample.
	In Section \ref{section3} we will weaken this sufficiency requirement on $x^n$ substantially and present a constructive approach for deciding on this set $\{x_1,\ldots,x_{m_n}\}$. However, to satisfy this weaker condition  it is good strategy to  select $x^n$ to make it as sufficient as possible.  Therefore consideration of the sufficiency assumption helps us understand the PC dimension reduction as a sufficient statistic and that a rate of convergence could be derived from that principle alone as in our next theorem. Our weaker assumption will only require that $x^n$ is an i.i.d. sample from a probability distribution $P_X$ chosen so that $d_0(\psi,\psi_0)$ is equivalent with $\int(\psi-\psi_0)^2 dP_X$. 
	
	{Note that the superscript $n$ in $x^{n}$, and the subscript in $D^{(k)}(\mathcal{E}_n)$, do not indicate size $n$. Both objects have size $m_n$ (typically equal to $n$). The index $n$ is used solely to emphasize their dependence on the sample, not to imply that their dimension necessarily coincides with the sample size.}
	
	It is  {worthwhile} to note that we can weaken the sufficiency assumption on $x^n$ more directly by only requiring it for an approximate loss function $L_n(\psi)$ that strongly approximates the true desired loss function $L(\psi)$, as defined in the following assumption.
	
	\begin{assumption}[Approximate sufficiency and regularity of $L_n$]\label{ass:approx_sufficiency}
		Let $L_n$ be an approximation of the true loss function $L$. Define the approximate minimizer $\psi_{0n} = \arg\min_{\psi} P_0 L_n(\psi)$ and the approximate dissimilarity $d_{0n}(\psi, \psi_{0n}) = P_0 L_n(\psi) - P_0 L_n(\psi_{0n})$.
		
		We assume the following conditions hold:
		
		\begin{enumerate}
			\item {Approximation error:} The approximate minimizer converges to the true minimizer $\psi_0$ at a rate faster than the critical minimax rate. Specifically, for some $\delta > 0$: $d_0(\psi_{0n}, \psi_0) = o_P\left(n^{-\frac{2k^*}{2k^*+1} - \delta}\right),$
			where $\psi_0 \in D^{(k)}_M([0,1]^d)$ is the minimizer of the true risk $P_0 L(\psi)$.
			
			\item {Loss regularity:} The approximate loss satisfies uniform boundedness and a quadratic margin condition (Bernstein condition) with respect to the approximate dissimilarity:
			\[
			\sup_{\psi \in D^{(k)}_M, o} |L_n(\psi)(o)| = O_P(1),
			\quad
			P_0 \left( L_n(\psi) - L_n(\psi_{0n}) \right)^2 = O_P(d_{0n}(\psi, \psi_{0n})).
			\]
			
			\item {Sufficiency of design points:} There exists a set of evaluation points $x^n = (x_1, \ldots, x_{m_n}) \in [0,1]^{d \times m_n}$ such that for any spline candidate $\psi_{N,\beta} = \sum_{j \in \mathcal{R}_N} \beta(j)\phi_j$, the vector of empirical losses $(L_n(\psi_{N,\beta})(O_i))_{i=1}^n$ depends on the function $\psi_{N,\beta}$ solely through the design matrix values $H_{ij} = \phi_j(x_i)$.
		\end{enumerate}
	\end{assumption}
	
	% \begin{assumption}[Approximate sufficiency assumption on $x^n$]\label{approxsufficiencyassumption}
		% Let $L_n(\psi)$ be an approximation of the loss function $L(\psi)$. Let $d_{0n}(\psi,\psi_{0n})=P_0L_n(\psi)-P_0L(\psi_{0n})$, where
		%  $\psi_{0n}(P_0)=\arg\min_{\psi}P_0L_n(\psi)$. Recall  $\psi_0=\arg\min_{\psi}P_0L(\psi)$ and assume $d_0(\psi_{0n},\psi_0)=o_P(n^{-2k^*/(2k^*+1)-\delta})$ for some $\delta>0$, where $\psi_0\in D^{(k)}_M([0,1]^d)$.
		%  In addition, assume that the loss function $L_n$ satisfies the typical boundedness and quadratic dissimilarity assumption uniformly in $n$:
		% $\sup_{\psi\in D^{(k)}_M([0,1]^d),o}\mid L_n(\psi)(o)\mid = O_P(1)$; $P_0(L_n(\psi)-L_n(\psi_{0n}))^2=O_P(d_{0n}(\psi,\psi_{0n}))$.
		% Let
		% $x_i\in [0,1]^d$, $i=1,\ldots,m_n$, be $m_n$ values chosen so that $L_n(\psi_{N,\beta})(O_i)$, $i=1,\ldots,n$, only depends on $\psi_{N,\beta}$ through 
		% a design matrix $H_{m_n\times N}(i,j)=\phi_j(x_i)$, $i=1,\ldots,m_n$, $j\in {\cal R}_N$. Let $x^n=(x_i: i=1,\ldots,m_n)$ be the $m_n$-dimensional vector of these values. 
		% \end{assumption}

	The most straightforward case in which the sufficiency assumption {in its exact version (\ref{sufficiencyassumption})} trivially holds is that $L(\psi)(O_i)=L(O_i,\psi(X_i))$ is a function of $\psi(X_i)$ with $X_i$ determined by $O_i$. In this case one simply sets $x^n=(X_i: i=1,\ldots,n)$. 
	
	\begin{example}The squared error loss $(Y-\psi_{N,\beta}(X))^2$ only {depends on $\psi_{N,\beta}$ through $\phi_j(x_i)$, with $x_i$ the covariate vector full sample}. If the sufficiency assumption holds, we also often have that the empirical score 
		$S_{N,\beta}(O_i)=d/d\beta L(\psi_{N,\beta})(O_i)$, $i=1,\ldots,n$,  is also determined by this same design matrix $H_{n\times N}$. For example, for this squared error loss $L(\psi_{N,\beta})(O_i)=(Y_i-\psi_{N,\beta}(X_i))^2$  the score of $\beta(j)$ equals $\phi_j(X_i)(Y_i-\psi(X_i))$ so that also the scores are only depending on $\psi_{N,\beta}$ through $\phi_j(X_i)$, $j\in {\cal R}_N$, $i=1,\ldots,n$. 
		{Similarly, the log-likelihood loss for binary outcome regression $L(\psi)(X,Y)=-Y\log 1/(1+\exp(-\psi))(X) -(1-Y)\log(\exp(-\psi)/(1+\exp(-\psi))(X)$ also exactly satisfies the sufficiency assumption for $x^n=(X_i: i=1,\ldots,n)$.}\end{example}
	
	\begin{example}[Log-likelihood loss with exponential link] Let's consider \\ $L(\psi_{N,\beta})(o)=-\log\{ \exp(\psi_{N,\beta})/C_N(\beta)\}$ with $C_N(\beta)=\int_{[0,1]^d}\exp(\psi_{N,\beta}(u)) du$ the normalizing constant. This is a density estimation problem {where an exponential family model is assumed}. We can set $x_i=o_i$, $i=1,\ldots,n$ but 
		in this case we have a slight deviation from the sufficiency assumption (1) since the normalizing constant also depends on the basis functions at intermediate values. Our later theorem  shows that this choice $x_i=o_i$ is indeed correct: here we use that the KL-divergence is equivalent to the square of $L^2(P_0)$-norm provided that all densities are bounded away from $\delta> 0$ for some $\delta$, so that indeed $X_i=O_i$ applies with $P_X=P_0$. 
		Alternatively we replace the normalizing constant  by an empirical mean of $n$ draws $u_i$, $i=1,\ldots,n$, from a uniform $[0,1]^d$ distribution so that
		$C_{N,n}(\beta)=n^{-1}\sum_{i=1}^n \exp(\psi_{N,\beta}(u_i))$. One can then show that $\sup_{\beta, {\lVert \beta\rVert_1<M}}
		\mid (C_{N,n}-C_N)(\beta)\mid =O_P(n^{-1/2})$ by {arguing that $\{\exp(\psi_{N,\beta}(\cdot)) :   {\lVert \beta\rVert_1<M} \}$ is a Donsker class}. We can then apply the approximate sufficiency assumption (2) with  the $2n$ points $x^n=((o_i: i=1,\ldots,n),(u_i: i=1,\ldots,n))$.
		As a related example, consider also conditional density estimation where $O_i=(Z_i,Y_i)$ and we use log-likelihood loss $L(\psi_{N,\beta})(O)=-\log \{ \exp(\psi_{N,\beta}(Z,Y))/C_N(\beta)(Z)\}$ with exponential link, and $C_N(\beta)(Z)=\int_y \exp(\psi_{N,\beta}(Z,y)) dy$. 
		The analogue arguments  applies to this setting as well. 
	\end{example}

	\begin{example}
		Suppose $O=(X,T)$ and that the failure time $T$ is discrete with support $\{t_j: j=1,\ldots,m\}$. 
		Let our target function be given by $\Psi(P_0)(t,x)=\lambda_0(t\mid x)$, where $\lambda_0(t\mid x)=p_0(t\mid x)/P_0(T\geq t\mid X=x)$ is the conditional (discrete) hazard of the failure time. 
		Suppose we model $\lambda_0(t\mid x)$ with $\mbox{Logit}\lambda_{\beta}(t\mid x)=\sum_j\beta(j)\phi_j(t,x)$ as logistic regression, where $\psi_{N,\beta}=\sum_{j\in {\cal R}_N}\beta(j)\phi_j$. Then the likelihood is,
		\[
		p_{\beta}(t\mid x)=\prod_{t_j<t}(1-\lambda_{\beta}(t_j\mid x))\lambda_{\beta}(t_j\mid x).\]
		Let $L(\psi_{N,\beta})(T,X)=-\log p_{\beta}(T\mid X)$. The log-likelihood $L(p_{\beta})(T_i,X_i)$, $i=1,\ldots,n$, depends on $\psi_{N,\beta}$ through $\{\psi_{N,\beta}(t_j,X_i): i=1,\ldots,n, j=1,\ldots,m\}$. Therefore the sufficiency assumption holds for $x^n=\{(t_j,x_i): i=1,\ldots,n,j=1,\ldots,m\}$.
	\end{example}
	
	\begin{example} Suppose now that $T$ is continuous and we wish to estimate  the continuous conditional hazard. In that case, we would model $\lambda(t\mid x)=\exp(\psi_{N,\beta}(t,x))$. In this case, the true log-likelihood would involve a product integral over $s\leq t$ of $1-\lambda(s\mid x) ds$. We could approximate this log-likelihood by replacing this continuous product integral by a finite product $\prod_{t_j\leq t}(1-\lambda(t_j\mid x) (t_{j+1}-t_j))$ for a fine set of grid-points $t_j$, $j=1,\ldots,m$. Then, the resulting approximate log-likelihood loss satisfies the sufficiency assumption for $x^n=\{(t_j,x_i):i=1,\ldots,n,j=1,\ldots,m\}$. Moreover, an HAL-MLE for this approximate log-likelihood would be asymptotically equivalent with the actual HAL-MLE if we choose the number of grid-points fine  enough and growing with sample size. This clarifies that our proof relying on the sufficiency assumption can be applied to this approximate log-likelihood loss. This demonstrates that indeed the sufficiency assumption drives very much how we determine this set $x^n$, even in cases where the sufficiency assumption only holds by approximation. 
	\end{example}
	
	\subsection{The PC dimension reduction working model $D^{(k)}({\cal E}_n)$}
	%Compute $n$ eigenvectors of $H^{\top}H$ with non-zero eigenvalue:
	
	Consider the design matrix $H: \openr^N\rightarrow\openr^{{m_n}}$ and endow domain and image with inner products
	$\langle v_1,v_2\rangle_N\equiv \sum_j v_1(j)v_2(j)$ and $\langle \phi_1,\phi_2\rangle_n$ respectively, where the latter could be tailored for our purpose.
	A standard choice is $\langle \phi_1,\phi_2\rangle_n=1/m_n\sum_i \phi_1(x_i)\phi_2(x_i) w_n(x_i)$ for  some weight function $w_n$ where, for example, $w_n=1$. We also consider the choice $\langle \phi_1,\phi_2\rangle_n\equiv 1/m_n \sum_i S_{\theta_n}(\phi_1)S_{\theta_n}(\phi_2)$, where $S_{\theta_n}(\phi)=\frac{d}{d\theta_n}L(\theta_n)(\phi)$ is the score of the loss function at a good estimator $\theta_n$ of $Q_0$. For example, in the latter case $\theta_n$ could be a PC-HA for the standard inner product (and one could imagine iterating, but that is not needed). This latter choice arranges that the  scores of the coefficients in front of eigenvector basis functions are orthogonal in $L^2(P_n)$.
	
	We consider the case that $m_n$ is much smaller than $N$ (and typically $m_n=n$).
	Let $H^{\top}: \openr^{m_n}\rightarrow\openr^N$ be the adjoint of $H$ w.r.t. the chosen inner products. 
	We can then define the self-adjoint linear mapping $I_N=H^{\top}H:\openr^N\rightarrow\openr^N$. 
	For the standard inner product choice we have that $I_N$ is the $N\times N$ matrix defined by $I_N(j_1,j_2)=1/m_n \sum_{i=1}^{m_n} \phi_{j_1}(x_i)\phi_{j_2}(x_i):j_1,j_2)$.
	The eigenvalue decomposition allows one to write $I_N=H^{\top}H=E_N \Lambda_N E_N^{\top}$.
	This matrix $I_N$, just as $H$, has rank $m_n$. Let $E_m$, $m=1,\ldots,m_n$, be the $m_n$ $N\times 1$-eigenvectors with non-zero eigenvalue in the eigenvalue decomposition of $I_N$. Let $E_N^n$ be the $N\times m_n$ matrix whose columns are the first $m_n$ columns extracted from $E_N$. Let $\tilde{\phi}_m=\sum_{j\in {\cal R}_N}E_N^n(j,m)\phi_j$, $m=1,\ldots,m_n$, represent the linear combination of the original spline basis functions implied by the $m$-th eigenvector. We are creating an $m_n$ dimensional working model based on $m_n$ eigenvector specific basis functions. For notational convenience, we  denote $E_m=E^n_N(\cdot,m)$ for the $m$-th eigenvector so that $\tilde{\phi}_m(x^n)=H(E_m)$.
	
	Computational tricks allow us to obtain fast and memory friendly computations of the eigenvector basis functions at $x^n$ from the eigenvalue decomposition of $HH^{T}$. 
	Consider the singular value decomposition $H=UDE$. We have that $H^{\top}H= ED^2E^{\top}$ so we obtain the eigenvalue decomposition from the singular value decomposition.  Then the design matrix for PC-HA is given by $UD$. As $E$ is orthogonal, we have $HE = UD$. Then $UD(\alpha)= HE(\alpha)$, so that the new $n$ basis functions $\tilde{\phi}_m$ are given by columns of $HE$ or equivalently of $UD$. Also, it follows that the sectional variation norm of a candidate function in the minimization problem is the $L_1$-norm of the vector $E\alpha$. To ensure computational feasibility when $N >> n$, we utilize the correspondence between the singular value decomposition (SVD) of $H$ and the eigenvalue decomposition of the kernel matrix $K=HH^T$. Let $H=UDE^T$ be the SVD of $H$. The non-zero eigenvalues of $K$ are the squared singular values $D^2$. Consequently, the PC-basis vectors are given by the columns of $UD = HE$. In the special case $k=0$, it has been shown \cite{schuler2024highlyadaptiveridge} that $HH^T$ is directly computable without storing $H$.

	Let $D^{(k)}({\cal E}_n)=\{\sum_m {\alpha(m)}\tilde{\phi}_m:\alpha \in \mathbb{R}^{m_n}\}$ be the $m_n$-dimensional submodel of $D^{(k)}({\cal R}_N)$ spanned by these eigenvector functions $\tilde{\phi}_m$, $m=1,\ldots,m_n$. Therefore, $D^{(k)}({\cal E}_n)$ represents our $m_n$-dimensional desired dimension reduction of the original $N$ basis functions.

	\begin{definition}[Coefficient map from PC basis to original basis] For a given $\alpha \in \mathbb{R}^{m_n}$, we define the $j$-th component of  $\beta(\alpha)$ as
		\[
		\beta(\alpha)(j)=\sum_{m=1}^{m_n}\alpha(m)E_N^n(j,m).\]
	\end{definition}
	
	Some important properties of  this new basis $\{\tilde{\phi}_m: 1 \leq m \leq m_n \}$ are the following.

	\begin{lemma}
		\label{lem:orthogonality}
		Let $\{\tilde{\phi}_m : 1 \le m \le m_n\}$ and $\{E_m : 1 \le m \le m_n\}$ be defined as above.
		
		\medskip
		{(i) Orthogonality with respect to $\langle\cdot,\cdot\rangle_n$.}
		For $m_1 \neq m_2$,
		$$\langle \tilde{\phi}_{m_1}, \tilde{\phi}_{m_2} \rangle_n = 0,
		\qquad
		\langle E_{m_1}, E_{m_2} \rangle_N = 0.$$
		Moreover,
		$$\langle E_m, E_m \rangle_N = 1
		\quad \text{and} \quad
		\langle \tilde{\phi}_m, \tilde{\phi}_m \rangle_n = \lambda_m,$$
		where $\lambda_m$ is the eigenvalue associated with $E_m$, satisfying
		$I_N(E_m) = \lambda_m E_m$.
		
		\medskip
		\noindent
		{(ii) Orthonormality with respect to the coefficient inner product.}
		For $f_1, f_2 \in D^{(k)}({{\cal E}_{ {n}}})$, define
		$\langle f_1, f_2 \rangle_N := \langle \beta(f_1), \beta(f_2) \rangle_N.$
		Then, for $m_1 \neq m_2$,
		$$\langle \tilde{\phi}_{m_1}, \tilde{\phi}_{m_2} \rangle_N = 0,$$
		and
		$$\langle \tilde{\phi}_m, \tilde{\phi}_m \rangle_N = 1
		\quad \text{for all } m.$$
	\end{lemma}

	The result follows from the adjoint properties of the design operator. Observe that: $\langle \tilde{\phi}_{m_1},\tilde{\phi}_{m_2}\rangle_n=\langle H(E_{m_1}),H(E_{m_2})\rangle_n
	=\langle E_{m_1},H^T H (E_{m_2})\rangle_N
	=\langle E_{m_1},\lambda_{m_2}E_{m_2}\rangle_N
	=\lambda_{m_2}\langle E_{m_1},E_{m_2}\rangle_N$. Since eigenvectors corresponding to distinct eigenvalues are orthogonal, the term vanishes for $m_1\not =m_2$ and equals $\lambda_{m}$ for $m_1=m_2=m$. 
	
	%$\Box$
	
	The important feature of this dimension reduction is that if we view elements of $D^{(k)}({\cal R}_N)$  as $m_n$-dimensional vectors evaluated at $(x_i: i=1,\ldots,m_n)$, assuming the sufficiency assumption on $x^n$, we have that 
	$\{\psi(x^n): \psi\in D^{(k)}({\cal E}_n)\}=\{\psi(x^n): \psi\in D^{(k)}({\cal R}_N)\}$; i.e. the $m_n$-dimensional PC-working model $D^{(k)}({\cal E}_n)$ is sufficient under sufficiency assumption (1). As a consequence, we have that
	the set $\{L(\theta_{n,\alpha})(O^n):\alpha\}$ equals $\{L(\psi_{N,\beta})(O^n):\beta\}$, where $ {\theta_{n,\alpha}=\sum_{m=1}^n \alpha(m)\tilde{\phi}_m}$ and $\psi_{N,\beta}=\sum_{j\in {\cal R}_N}\beta(j)\phi_j$. Therefore, we can conclude that the minimum empirical risk over $D^{(k)}({\cal E}_n)$ equals the minimum empirical risk over $D^{(k)}({\cal R}_N)$.  From that perspective, we might consider $(P_n L(\theta_{n,\alpha}):\alpha)$ as a sufficient statistic for estimation of $\psi_{n,\beta_0}$ with $\beta_{0,N}=\arg\min_{\beta}P_0 L(\psi_{N,\beta})$. Note that $\theta_{n,\alpha}=\psi_{N,\beta(\alpha)}.$
	%In fact, the latter is a fair statement since we will show below that in fact PC-HAL and PC-HAR defined as MLEs over $\{\theta_{n,\alpha}:\alpha\}$ with appropriate $L_1$ and $L_2$-constraint on the $\alpha$ achieve same empirical risk as HAL and HAR.
	
	{In what follows, we set $m_n=n$ for notational convenience}. We define the estimators by imposing constraints on various norms of the coefficient vector that parameterizes the space $D^{(k)}({\cal E}_n)$. 
	For a given norm $ {\lVert\cdot\rVert}$, we can define a subset $ {D^{(k)}_{\lVert\cdot\rVert,C}({\cal E}_n)}=\{\theta_{n,\alpha}\in D^{(k)}({\cal E}_n):
	{\lVert \alpha\rVert<C}\}$. We will consider different choices of norms:
	\begin{eqnarray*}
		{\lVert\alpha\rVert_1}&\equiv&\sum_{m=1}^n  {\lvert} \alpha(m) {\rvert} \\
		{\lVert\alpha\rVert_2}&\equiv&\left( \sum_{m=1}^n \alpha(m)^2\right) ^{1/2}\\
		{\lVert \alpha\rVert_{3}}&\equiv&\sum_{j=1}^N {\lvert} \beta(\alpha)(j) {\rvert}  .
	\end{eqnarray*}
	Recall that $ {\lVert \alpha\rVert_3}$ constraints the $k$-th order sectional variation norm, just as done in in $k$-th order HAL. The other norms will indirectly constrain this $k$-th order sectional variation norm, as discussed below.
	Let's denote the corresponding $ {D^{(k)}_{\lVert\cdot\rVert,C}({\cal E}_n)}$ with  $D^{(k)}_{1,C}({\cal E}_n)$; $D^{(k)}_{2,C}({\cal E}_n)$ and $D^{(k)}_{3,C}({\cal E}_n)$, respectively.
	Recall $D^{(k)}_C({\cal R}_N)=\{\psi_{N,\beta}: {\lVert \beta\rVert_1<C}\}$ is the HAL-model constraining the $k$-th order sectional variation norm to be bounded by $C$, and $D^{(k)}_{2,C}({\cal R}_N)=\{\psi_{N,\beta}: {\lVert \beta\rVert_2<C}\}$, the space that defines HAR.

	We have the following simple facts concerning embedding $D^{(k)}_{\pl\cdot\pl,C}({\cal E}_n)$ in $D^{(k)}_M({\cal R}_N)$ for some $M=M_C<\infty$  so that the complexity  of $D^{(k)}_{\pl\cdot\pl,C}({\cal E}_n)$ is controlled by constraining the $k$-th order sectional variation norm. Recall that $D^{(k)}_M({\cal R}_N)$ exhibits covering number bounds leading to dimension free rates of convergence (Lemma 29 in \cite{van2023higher}). We also know that if we can embed it in $D^{(k)}_{2,C/n^{1/2}}({\cal R}_N)$, then we will also be controlling this $k$-th order sectional variation norm.
	
	% \begin{lemma}
		% Recall that the size of ${\cal R}_N$ is $O(n)$.
		% \newline
		% {\bf $\pl \alpha\pl_2$:} We have
		%  $\pl \beta(\alpha)\pl_2=\pl \alpha\pl_2$ and thus $D^{(k)}_{2,C}({\cal E}_n)\subset D^{(k)}_{2,C}({\cal R}_N)$.
		%  By Cauchy-Schwarz inequality,  
		% $D^{(k)}_{2,C/N^{1/2}}({\cal R}_N)\subset D^{(k)}_C({\cal R}_N)$, so that also
		% $D^{(k)}_{2,C/N^{1/2}}({\cal E}_n)\subset D^{(k)}_C({\cal R}_N)$.
		%  \newline
		%  {\bf $\pl \alpha\pl_3$:} We have 
		% $\pl \alpha\pl_3=\pl \beta(\alpha)\pl_1$, and therefore $D^{(k)}_{3,C}({\cal E}_n)\subset D^{(k)}_C({\cal R}_N)$; 
		% \newline
		% {\bf $\pl \alpha\pl_1$:}
		% We have $\sup_{\pl\alpha\pl_1<C,\max_m\mid \alpha(m)\mid <C/n}\pl\beta(\alpha)\pl_1 =O(C)$.
		% Therefore, if we define $D^{(k)}_{1,C,C/n}({\cal E}_n)=\{\theta_{n,\alpha}:\pl \alpha\pl_1<C,\max_m \mid \alpha(m)\mid \leq C/n\}$, then we have $D^{(k)}_{1,C,C/n}({\cal E}_n)\subset D^{(k)}_{M(C)}({\cal R}_N)$ for some $M(C)=O(C)$.
		% \end{lemma}
	
	\begin{lemma}[Relationships between PC models and the HAL space]\label{lem:sub}
		Let $N$ denote the cardinality of the original spline basis set $\mathcal{R}_N$, where $N = O(n \cdot 2^{kd})$. Let $\beta(\alpha)$ denote the mapping from the $n$-dimensional PC-coefficient vector to the $N$-dimensional spline coefficient vector.
		
		\begin{enumerate}
			\item {PC-HAR:} 
			Since the columns of ${E}_N^n$ are orthonormal, the Euclidean norm is preserved: $\|\beta(\alpha)\|_2 = \|\alpha\|_2$. Consequently, the PC-HAR model is a subset of the full HAR model:
			\[ D^{(k)}_{2, C}(\mathcal{E}_n) \subset D^{(k)}_{2, C}(\mathcal{R}_N). \]
			Furthermore, by the Cauchy-Schwarz inequality, $\|\beta\|_1 \leq \sqrt{N} \|\beta\|_2$, which implies:
			\[ D^{(k)}_{2, C/\sqrt{N}}(\mathcal{E}_n) \subset D^{(k)}_{1, C}(\mathcal{R}_N). \]
			
			\item {PC-HAGL:} 
			By defining the constraint directly on the induced $L_1$-norm of the original coefficients, i.e., $\|\beta(\alpha)\|_1 \leq C$, we  {ensure} that the PC-HAGL model remains a valid submodel of the HAL space:
			\[ D^{(k)}_{3, C}(\mathcal{E}_n) \subset D^{(k)}_{1, C}(\mathcal{R}_N). \]
			
			\item {PC-HAL:} 
			Suppose we constrain the PC-coefficients directly such that they belong to the subset $\tau_n(C) \equiv {\{\alpha : \|\alpha\|_1 < C, \|\alpha\|_\infty < C/n\}}$ {.} Under these conditions, the induced sectional variation norm  {is bounded as}:
			\[ \sup_{\alpha \in  {\tau_n(C)}} \|\beta(\alpha)\|_1 = O(C). \]
			Define $D^{(k)}_{1,C,C/n}({\cal E}_n)=\{\theta_{n,\alpha}: \alpha \in  {\tau_n(C)} \}$. There exists a constant $M(C) = O(C)$ such that the restricted PC-HAL model  {satisfies} $D^{(k)}_{1, C, C/n}(\mathcal{E}_n) \subset D^{(k)}_{1, M(C)}(\mathcal{R}_N)$.
		\end{enumerate}
	\end{lemma}
	
	{The previous result poses a sparsity vs. smoothness tradeoff.  While explicitly constraining the induced variation norm $\|\beta(\alpha)\|_1 \leq M$ (as in PC-HAGL) provides the most direct control over the HAL complexity, the standard Euclidean norms on $\alpha$ offer alternative regularization mechanisms with distinct theoretical implications.
	}

	\subsection{{Definitions of PC-HA estimators}}
	
	We now define the three specific estimators based on the choice of norm constraint imposed on the coefficient vector $\alpha$ within the PC-reduced working model $D^{(k)}({\cal E}_n)$. In all definitions below, let $\theta_{n,\alpha} = \sum_{m=1}^{m_n} \alpha(m)\tilde{\phi}_m$ denote a candidate function in the reduced space, and let $C_n$ denote a regularization parameter selected via cross-validation.
	
	\begin{description}
		\item[1. PC-HAGL (Generalized Lasso):]  directly  {constrains} the implied sectional variation norm of the original basis coefficients $\beta(\alpha)$. The estimator is defined as:
		\[
		\hat{\Psi}_{\text{PC-HAGL}}(P_n) = \theta_{n,\hat{\alpha}_{GL}}, \quad \text{where} \quad \hat{\alpha}_{GL} = \mathop{\mathrm{arg\,min}}_{\alpha : \|\beta(\alpha)\|_1 \le C_n} P_n L(\theta_{n,\alpha}) = \mathop{\mathrm{arg\,min}}_{\theta_{n,\alpha}\in D^{(k)}_{3,C_n}({\cal E}_n)}P_n L(\theta_{n,\alpha}). 
		\]
		Since $\|\beta(\alpha)\|_1$ represents the $L_1$-norm of a linear transformation of $\alpha$, this optimization corresponds to a \textit{generalized Lasso} problem \cite{Tibshirani:2011:SPGL}. By construction, the resulting estimator satisfies $\theta_{n,\hat{\alpha}_{GL}} \in D^{(k)}_{C_n}({\cal R}_N)$, preserving the exact sectional variation norm constraint of the original HAL.
		
		\item[2. PC-HAL (Lasso):] 
		imposes a standard $L_1$-constraint directly on the PC-coefficients $\alpha$. The estimator is defined as:
		\[
		\hat{\Psi}_{\text{PC-HAL}}(P_n) = \theta_{n,\hat{\alpha}_{L}}, \quad \text{where} \quad \hat{\alpha}_{L} = \mathop{\mathrm{arg\,min}}_{\alpha : \|\alpha\|_1 \le C_n} P_n L(\theta_{n,\alpha}) = \mathop{\mathrm{arg\,min}}_{\theta_{n,\alpha}\in D^{(k)}_{1,C_n}({\cal E}_n)}P_n L(\theta_{n,\alpha}).
		\]
		This formulation allows for computation using the closed form in \cite{wang2026}. However, as noted in the Lemma above, ensuring that $\hat{\Psi}_{\text{PC-HAL}}$ resides in a bounded variation class $D^{(k)}_M({\cal R}_N)$ requires the additional condition that the coefficients are uniformly small, specifically $\max_m |\hat{\alpha}_{L}(m)| = O(n^{-1})$. While not explicitly enforced by the standard Lasso, this condition is generally plausible if the true target function $\psi_0$ satisfies typical smoothness assumptions.
		
		\item[3. PC-HAR (Ridge):]  imposes an $L_2$-constraint on $\alpha$. Due to the orthonormality of the PC-basis coefficients in the embedding space (Lemma \ref{lem:orthogonality}), we have $\|\beta(\alpha)\|_2 = \|\alpha\|_2$. The estimator is defined as:
		\[
		\hat{\Psi}_{\text{PC-HAR}}(P_n) = \theta_{n,\hat{\alpha}_{R}}, \quad \text{where} \quad \hat{\alpha}_{R} = \mathop{\mathrm{arg\,min}}_{\alpha : \|\alpha\|_2 \le C_n} P_n L(\theta_{n,\alpha}) = \mathop{\mathrm{arg\,min}}_{\theta_{n,\alpha}\in D^{(k)}_{2,C_n}({\cal E}_n)}P_n L(\theta_{n,\alpha}) {.}
		\]
		This estimator is implemented via standard ridge regression on the design matrix $UD$. To ensure the resulting function has bounded sectional variation (i.e., $\|\beta(\hat{\alpha}_R)\|_1 = O(1)$), the tuning parameter must satisfy $C_n = O(n^{-1/2})$, relying on the relationship $\|\beta\|_1 \leq \sqrt{N}\|\beta\|_2$ and $N \asymp n$. This estimator inherits the closed form of standard ridge, and does not need to invert any matrix due to Lemma \ref{lem:orthogonality} \cite{wang2026}. 
	\end{description}
	
	\subsection{Computational and statistical comparisons of PC-HAGL, PC-HAL {,} PC-HAR {.}}
	The constraint $ {\lVert \beta(\alpha)\rVert_1<C}$ is somewhat non-standard, while, on the other hand, $ {\lVert \beta(\alpha)\rVert_2=\lVert \alpha\rVert_2}$. This makes PC-HAR trivial to implement with standard ridge regression. By restricting $C=O(1/N^{1/2})$ one still controls the $k$-th order sectional variation norm.  The constraint $ {\lVert \beta(\alpha)\rVert_1}$ corresponds with a standard $L_1$-constraint on the linear transformation
	$E_N^n(\alpha)$: i.e. $ {\lVert \beta(\alpha)\rVert_1=\lVert E_N^n(\alpha)\rVert_1}$. In various cases, this type of generalized Lasso implementation might not be standard and easily available: in a later section we will propose a particular algorithm for computing $\hat \alpha_{GL}$.  Either way,   the easy PC-HAL controlling $ {\lVert \alpha\rVert_1}$ or PC-HAR controlling $ {\lVert \alpha\rVert_2}$ are  computationally attractive \cite{wang2026}.
	
	All three estimators appear to control $ {\lVert \beta(\alpha)\rVert_1}$, which is a measure of true complexity of the space, in particular, controlling the covering number. Therefore, all three estimators will generally achieve the desired rates of convergence. In what sense will they end up being different in controlling the $L_1$-norm of $\beta(\alpha)$? Let $\psi_{N,0}=\Psi_{N,\beta_{0,N}}$ be the loss-based projection of $\psi_0$ onto $D^{(k)}({\cal R}_N)$. This implies a best approximation $\theta_{n,\alpha_{0,n}}$ with $\alpha_{0,n}$ defined by
	$\theta_{n,\alpha_{0,n}}(x^n)=\psi_{N,0}(x^n)$. 
	Let $C_0= {\lVert \beta(\alpha_{0,n})\rVert_1}$. Then, $D^{(k)}_{3,C_0}({\cal E}_n)$ already includes $\theta_{\alpha_{0,n}}$, while controlling the $L_1$-norm (i.e {.}, $k$-th order sectional variation norm) at level $C_0$.
	Let $C_{2,0}\equiv  {\lVert \alpha_{0,n}\rVert_2}$. Then $D^{(k)}_{2,C_{2,0}}({\cal E}_n)$ includes $\theta_{\alpha_{0,n}}$, but $\sup_{ {\lVert \alpha\rVert\leq C_{2,0}}} {\lVert \beta(\alpha)\rVert_1}$ might be significantly larger than $C_0$. As a consequence the cross-validated PC-HAR might end up minimizing over a significantly larger set of functions than cross-validated PC-HAGL. This suggests that in finite samples the PC-HAGL might be better than PC-HAR. Let $C_{1,0}= {\lVert \alpha_{0,n}\rVert_1}$. Then the set $D^{(k)}_{1,C_{1,0}}({\cal E}_n)$ includes $\theta_{n,\alpha_{0,n}}$, but we might have that $\sup_{ {\lVert\alpha\rVert_1<C_{1,0}}} {\lVert \beta(\alpha)\rVert_1}$ is significantly larger than $C_0$. 
	
	%Dont think this is true: An interesting feature of PC-HAGL is that the choice of inner product $\langle\cdot,\cdot\rangle_n$, which clearly affects the resulting $\{\tilde{\phi}_m:m\}$, does not affect the definition of PC-HAGL. This follows since $\{\beta:\pl \beta(\alpha)\pl_1\leq C\}$ is invariant under reparemetrizations of $D^{(k)}({\cal E}_n)$.

	In summary, the $m_n$ principal components with non-zero eigenvalues of the $N\times N$ matrix $\Sigma_n=H^{\top}H=P_n {\bf \phi}{\bf \phi}^{\top}=(1/n\sum_i \phi_{j_1}(x_i)\phi_{j_2}(x_i) {,\ 1 \leq j_1,j_2 \leq N})$  {imply} an $m_n$-dimensional working model 
	$D^{(k)}({\cal E}_n)=\{\theta_{n,\alpha}:\alpha\}$. Under the sufficiency assumption, this working model allows us to minimize the empirical risk as much as the full HAL-model $D^{(k)}({\cal R}_N)$. Three different PC-Highly Adaptive estimators (PC-HA) are now defined by specifying a norm $ {\lVert \cdot\rVert}$ on $\alpha$; computing the constrained MLE over $ {D^{(k)}_{\lVert\cdot\rVert,C}({\cal E}_n)}$ for each $C$, thus constraining $ {\lVert \alpha\rVert\leq C}$; choosing $C$ with cross-validation. Examples of such PC-HA  {estimators} are using $ {\lVert \alpha\rVert_3\equiv \lVert \beta(\alpha)\rVert_1}$; $ {\lVert \alpha\rVert_2\equiv \lVert \beta(\alpha)\rVert_2}$; $ {\lVert \alpha\rVert_1}$. All of these will control $ {\lVert \beta(\alpha)\rVert_1}=O(1)$, but $ {\lVert\cdot\rVert_3}$ does directly control $ {\lVert \beta(\alpha)\rVert_1}$, while other norms control it indirectly, possibly coming at a cost of its cross-validated version having to optimize the empirical risk over a set of functions with higher $L_1$-norm than needed to capture the true function. In the next Section \ref{section3} we establish that these PC-HAs will converge at  {the} same rate as HAL or HAR, under our assumption that $P_n L(\psi)$ only depends on $\psi(x^n)$, and the assumption that the norm chosen effectively  {constrains} $ {\lVert \beta(\alpha)\rVert_1}$ as well, which is guaranteed by $ {\lVert \cdot\rVert_3}$ in particular.
	In the subsequent Section  we will generalize the results to  more general $x^n$ choices drastically weakening the sufficiency assumption on $x^n$.
	
	\section{Convergence rates}\label{section3}
	{This section presents two theoretical results which, under different sets of assumptions, establish the same convergence rate. The first result relies on Approximate Sufficiency Assumption 2. The second assumes that $x^n$ is an i.i.d. sample from a probability distribution $P_{X,0}$ such that the loss-based dissimilarity is equivalent to the squared $L_2(P_{X,0})$ norm.}
	\begin{assumption}\label{ass:oracle}
		Let $\psi_{N,0} = \arg\min_{\psi \in D^{(k)}({\cal R}_N)} P_0 L_n(\psi)$ be the best approximation in the finite-dimensional spline model. We assume:
		\begin{enumerate}
			\item {Boundedness:} There exists $M < \infty$ such that $\psi_{N,0} \in D^{(k)}_M({\cal R}_N)$.
			\item The oracle approximation error satisfies
			\[ d_0(\psi_{N,0}, \psi_0) = O_P^+(n^{-2k^*/(2k^*+1)}), \]
			where $\psi_0$ is the true minimizer of $P_0 L(\psi)$.
			\item {Loss regularity:} The class $\{L_n(\psi) : \psi \in D^{(k)}_M({\cal R}_N)\}$ satisfies the same entropy bounds as $D^{(k)}_M({\cal R}_N)$, and the approximate loss satisfies the local quadratic condition:
			\[ \sup_{\psi \in D^{(k)}({\cal R}_N) : d_0(\psi, \psi_{N,0}) \le 1} P_0 \{L_n(\psi) - L_n(\psi_{N,0})\}^2 = O(1). \]
		\end{enumerate}
	\end{assumption}
	
	\begin{assumption}[ {SV norm control under PC representation}]\label{ass:pc_adequacy}
		Let $\psi_{N,\beta_n}$ be the HAL-MLE on the full spline model $D^{(k)}_M({\cal R}_N)$. Let $\alpha_n^*$ be a solution to the equation $\theta_{n,\alpha}(x^n) = \psi_{N,\beta_n}(x^n)$, ensuring equivalent empirical risk: $P_n L_n(\theta_{n,\alpha_n^*}) = P_n L_n(\psi_{N,\beta_n})$.
		Let $r(n)$ be a rate function such that $\sup_{\|\alpha\| < r(n)} \|\beta(\alpha)\|_1 = O(1)$. We assume that the PC-representation of the HAL-MLE is bounded by this rate:
		\[ \|\alpha_n^*\| = O_P(r(n)). \]
	\end{assumption}
	
	{
		The HAL-MLE satisfies $\|\beta_n\|_1=O_P(M)$. The equation $\theta_{n,\alpha}(x^n)=\psi_{N,\beta_n}(x^n)$ is a linear system in $\alpha$; under full rank of the design, there is at least one solution. For example, the coefficient vector $\alpha=E^{\top}\beta_n$ (the projection of $\beta_n$ onto the span of the PC eigenvectors $\{E_m\}$ in $\mathbb{R}^N$) reproduces the same function on the design, so $\theta_{n,\alpha}(x^n)=\psi_{N,\beta_n}(x^n)$. For this choice, $\|\beta(\alpha)\|_1\leq\|\beta(\alpha)\|_2\sqrt{m_n}\leq\|\beta_n\|_2\sqrt{m_n}=O_P(1)$ under the usual scaling $\|\beta_n\|_2=O_P(n^{-1/2})$, hence $\|\alpha\|=O_P(r(n))$ for the norm in the assumption. Thus there exists a solution with the required rate; the assumption is that the chosen $\alpha_n^*$ (e.g., a minimal-norm solution) is of that order.}
	
	We start by establishing a convergence rate under the Approximate Sufficiency Assumption \ref{ass:approx_sufficiency}. The proof can be found in the Supplement.  
	
	\begin{theorem}[Convergence rate of PC-HA estimators]\label{theoremone}
		Suppose Assumptions \ref{ass:approx_sufficiency}, \ref{ass:oracle}, and \ref{ass:pc_adequacy} hold. Consider the PC-HA estimator defined by minimizing the empirical risk over the submodel $D^{(k)}_{\|\cdot\|, C_n}({\cal E}_n)$ with a norm constraint $\|\cdot\|$. Then,
		
		\begin{enumerate}
			\renewcommand{\theenumi}{\roman{enumi}}
			
			\item There exists a deterministic sequence $C_n = O(r(n))$ such that the estimator $\theta_{n,\alpha_n}$ with $\alpha_n = \arg\min_{\|\alpha\| < C_n} P_n L_n(\theta_{n,\alpha})$ converges to the oracle approximation $\psi_{N,0}$ at the optimal nonparametric rate. Furthermore, if the constraint parameter is selected via $V$-fold cross-validation ($C_{n,cv}$), the resulting estimator $\theta_{n,\alpha_{n,cv}}$ achieves the same adaptive rate.
			
			\item Specifically, the loss-based dissimilarity satisfies:
			\begin{equation}
				d_0(\theta_{n,\alpha_{n,cv}}, \psi_{N,0}) = O_P^+\left( n^{-\frac{2k^*}{2k^*+1}} \right)
			\end{equation}
			where $k^* = k+1$ and $k$ is the order of the spline basis. 
		\end{enumerate}
	\end{theorem}

	Applying this theorem to the three norms on $D^{(k)}({\cal E}_n)$ yields the following corollary.
	\begin{corollary}
		Consider the setting of the above theorem for $ {\lVert \alpha\rVert_3=\lVert \beta(\alpha)\rVert_1}$ with $r(n)=1$; $ {\lVert \alpha\rVert_2}$ with $r(n)=1/n^{1/2}$; $ {\lVert \alpha\rVert_1}$ with $r(n)=1$.
		Let $\alpha_{n,j}=\arg\min_{ {\lVert \alpha\rVert_j}<C_{n,j}}P_n L(\theta_{n,\alpha})$ with $C_{n,j}$ the $V$-fold cross-validation selector, $j=1,2,3$. In particular, recall the definition of
		$\alpha_n^*$ of Theorem \ref{theoremone} chosen so that $P_n L(\theta_{n,\alpha_n^*})=P_n L(\psi_{N,\beta_n})$ where $\psi_{N,\beta_n}$ is the HAL-MLE over a set contained in $D^{(k)}_M([0,1]^d)$.

		If $ {\lVert \beta(\alpha_n^*)\rVert_1}=O(1)$, then $d_0(\theta_{n,\alpha_{n,3}},\psi_{N,0})=O^+_P(n^{-2k^*/(2k^*+1)})$.
		
		If $ {\lVert \alpha_n^*\rVert_2}=O_P(n^{-1/2})$, then $d_0(\theta_{n,\alpha_{n,2}},\psi_{N,0})=O^+_P(n^{-2k^*/(2k^*+1)})$.
		
		If $ {\lVert \alpha_n^*\rVert_1}=O(1)$ and $\max_{j\in {\cal E}_n}\mid \alpha_n^*(j)\mid =O(1/n)$, then $d_0(\theta_{n,\alpha_{n,1}},\psi_{N,0})=O^+_P(n^{-2k^*/(2k^*+1)})$.

		If $d_0(\psi_{N,0},\psi_0)=O_P(n^{-2k^*/(2k^*+1)})$, then the above rates apply to $d_0(\theta_{n,\alpha_{n,j}},\psi_0)$, $j=1,2,3$.
	\end{corollary}

Below we analyze rate of convergence without relying on the assumption that $P_n L(\psi)$ is determined by $\psi(x^n)$, but instead we will assume that $x^n$ represents an i.i.d. sample from a probability distribution $P_{X,0}$ for which the loss-based dissimilarity is equivalent with the square of the $L^2(P_{X,0})$. In this manner, we establish results in cases in which the empirical risk $P_n L(\psi)$  is not exactly determined by $\psi(x^n)$. One would first study the loss based dissimilarity to determine such a $P_{X,0}$ distribution and then one could take a sample from this $P_{X,0}$ or determine $x_i$ from the observations $o_i$ that would satisfy this.

Recall the definition $\psi_{n,0}=\theta_{n,\alpha_{0,n}}=\arg\min_{\psi\in D^{(k)}({\cal E}_n)}P_0L(\psi)$. To prove our desired result we need to show  $d_0(\theta_{n,\alpha_{n,0}},\psi_{N,0})=O^+(n^{-k^*/(2k^*+1)})$ and $\theta_{n,\alpha_{n,0}}\in D^{(k)}_M({\cal R}_N)$ for some $M<\infty$.
To achieve this we make the following two assumptions.

%Let's now propose a strategy for analyzing PC-HAR and PC-HAL.A typical loss based dissimilarity is equivalent with the square of an $L^2$-norm when restricting to functions for which the loss-function is universally and uniformly bounded so that the above assumption is not a strong assumption. This makes the following assumption a weak assumption.
\begin{assumption}\label{assumption1}
	$\psi_{N,0}\in D^{(k)}_M({\cal R}_N)$. Assume that $d_0(\psi,\psi_{N,0})=O^+(d_{2,0}(\psi,\psi_{N,0}))$ uniformly in $\psi\in D^{(k)}_M({\cal R}_N)$ for some $M<\infty$, where
	$d_{2,0}(\psi,\psi_{N,0})\equiv P_{X,0}(\psi-\psi_{N,0})^2$ for some probability distribution $P_{X,0}$.
\end{assumption}

%Regarding our choice of the points $\{x_i: i=1,\ldots,n\}$ we now only require that it equals an i.i.d. sample of $P_{X,0}$, thereby weakening the previous assumption that these $n$ points should be chosen so that the empirical risk $P_n L(\psi)$ only depends on $(\psi(x_i); i=1,\ldots,n)$ for all $\psi\in D^{(k)}({\cal R}_N)$. Nonetheless, inspecting the empirical risk would often suggest this set of points.
\begin{assumption}\label{assumption2}
	Assume that $x^n=\{x_i:1 \leq i \leq m_n\}$ is an i.i.d. sample of $m_n$ observations from the distribution $P_{X,0}$. 
\end{assumption}

{In the Supplement \ref{app:loss}, we verify that Assumption \ref{assumption1} and the i.i.d. sampling requirement(Assumption \ref{assumption2}) hold for squared loss, logistic regression, and density estimation with exponential families.}

For a particular norm $ {\lVert\cdot\rVert}$ on $\alpha$, we define a $C$-specific MLE $\psi_{n,C}=\arg\min_{\theta_{n,\alpha}\in D^{(k)}({\cal E}_n), {\lVert\alpha\rVert}\leq C}P_n L(\theta_{n,\alpha})$. 

\begin{assumption}\label{assumption4}
	Suppose that $ {\lVert \alpha_{0,n}\rVert}=O_P(r(n))$ for a rate $r(n)$ implying 
	$\sup_{ {\lVert \alpha\rVert}<r(n)} {\lVert \beta(\alpha)\rVert_1}=O(1)$. 
\end{assumption}

{Assumption \ref{assumption4} is automatically satisfied under PC-HAGL because $
	\|\alpha\|_2 = \|\beta(\alpha)\|_2 \le \|\beta(\alpha)\|_1 \le C$, but it needs to be explicitly imposed for PC-HAR and PC-HAL. These estimators constrain the coefficients of the principal components. However, we require a constraint on the sectional variation norm (the coefficients $\beta$ of the original splines). There is no automatic guarantee that keeping the PC-coefficients small keeps the spline coefficients small.} 

%Under this assumption, there exists a $C_n$ so that $\psi_{n,C_n}$ is an MLE over a subset of $D^{(k)}_{M_n}({\cal E}_n)$ that contains the true $\psi_{N,0}\in D^{(k)}({\cal E}_n)$ and that $M_n=O_P(1)$. We can then analyze this MLE of $\psi_{N,0}$ over this class as usual using  the covering number of $D^{(k)}_M([0,1]^d)$, giving the rate $O^+(n^{-k^*/(2k^*+1)})$.

%The proof of Theorem \ref{rate} can be found in the Supplement and exploits that the PC-HAGL estimator  $ \hat{\Psi}_{\text{PC-HAGL}}(P_n) =\theta_{n,\alpha_n}$ is an MLE of $\theta_{n,\alpha_{n,0}}$ over a class $D^{(k)}_{3,C_n}({\cal E}_n)\subset D^{(k)}_{C_n}({\cal R}_N)\subset D^{(k)}_M([0,1]^d)$. 

\begin{theorem}\label{rate}
	Consider the $ {\lVert\cdot\rVert}$-specific regularized MLE $\psi_n=\psi_{n,C_n}$ with $C_n$ chosen with cross-validation. Under Assumptions \ref{assumption1}, \ref{assumption2},  \ref{assumption4}  {we have for the three PC-HA estimators}
	$$d_0(\psi_n,\psi_{N,0})=O_P^+(n^{-2k^*/(2k^*+1)}).$$
\end{theorem}

{Note that  {Assumption \ref{assumption4}} holds directly for PC-HAGL by construction.} We can apply this theorem to PC-HAR with $r(n)\sim n^{-1/2}$ in Assumption \ref{assumption4}. 
One then needs to establish that $ {\lVert \alpha_{0,n}\rVert_2}=O(n^{-1/2})$ which then allows that we select a constraint $C_n$ on $ {\lVert \alpha\rVert_2}$ that controls $ {\lVert\beta(\alpha)\rVert_1}=O(1)$ and still does not exclude the true $\alpha_{0,n}$. In \cite{schuler2024highlyadaptiveridge} it was shown that $\psi_{N,0}=\psi_{N,\beta_{0,n}}$ satisfies $ {\lVert \beta_{0,N}\rVert_2}=O(n^{-1/2})$ under a condition on $\psi_0$. Therefore one would expect that its projection onto the subset $D^{(k)}({\cal E}_n)$ would also satisfy $ {\lVert \beta(\alpha_{0,n})\rVert_2}=O(n^{-1/2})$ so that Assumption 4 holds. In general, this theorem makes it plausible that a collection of penalized MLEs over $D^{(k)}({\cal E}_n)$ would converge at our desired rate, the same rate as HAL and HAR achieve.

\section{Algorithm for computing PC-HAGL}\label{section5}
Let $R_n(\psi)$ be the empirical risk (for example, the empirical MSE
$R_n(\psi)=P_n (Y-\psi(X))^2$). 
{Closed forms for the empirical risk minimizer have been established for PC-HAL and PC-HAR in \cite{wang2026}}. For PC-HAGL we have to optimize $R_n(\theta_{n,\alpha})$ over $\{\alpha:  {\lVert \beta(\alpha)\rVert_1}=C\}$ for a user supplied $C$ where $\theta_{n,\alpha}=\sum_m \alpha(m)\tilde{\phi}_m$.
Here we propose a universal steepest descent algorithm. 
For notational convenience, let $R_n(\alpha)=R_n(\theta_{n,\alpha})$.

We define a collection of paths parametrized by $\alpha$ that preserve $ {\lVert \beta(\alpha)\rVert_1}=C$. Consider $\alpha_{\delta}^h(m)=(1+\delta h(m))\alpha(m)$ through $\alpha$ indexed by a direction $h$. 
We want to restrict $h$ so that $ {\lVert \beta(\alpha_{\delta}^h)\rVert_1}=C$ for $\delta\approx 0$. 
We have
\begin{eqnarray*}
	\beta(\alpha_{\delta}^h)(j)&=& \sum_m \alpha_{\delta}^h(m) {E_N^n}(j,m)\\
	&=& \sum_m (1+\delta h(m))\alpha(m)  {E_N^n}(j,m)\\
	&=&\beta(\alpha)(j)+\delta \sum_m h(m) \alpha(m) {E_N^n}(j,m)\\
	&=&\left(1+\delta \frac{\sum_m h(m)\alpha(m)E(j,m)}{\beta(\alpha)(j)}\right) \beta(\alpha)(j).
\end{eqnarray*}
For small $\delta\approx 0$, we have that this preserves the $L_1$-norm of $\beta(\alpha)$ if the $N$ dimensional vector
$(\sum_m h(m)\alpha(m) {E_N^n}(j,m)/\beta(\alpha)(j):j\in {\cal R}_N)$ is orthogonal to $( {\lvert} \beta(\alpha) {\rvert} (j):j\in {\cal R}_N)$: in that case, this path has the form $(1+\delta h^*(j))\beta(j)$ {with $h^*(j)= \frac{\sum_m h(m)\alpha(m) {E_N^n}(j,m)}{\beta(\alpha)(j)}$}and we have $ {\lVert} (1+\delta h^*)\beta {\rVert_1}= {\lVert}\beta {\rVert_1}+\delta \sum_j h^*(j) {\lvert} \beta(j) {\rvert}= {\lVert} \beta {\rVert_1}$ if $h^*\perp  {\lvert} \beta {\rvert}$ and ${\delta}<1/\max_j  {\lvert} h^*(j) {\rvert}$). 
%{Because $\delta > 0$, the problematic cases are those such that $h^*(j)<0$, so that we require $1 + \delta h^* (j) = 1 - \delta |h^*(j)| > 0 $. We therefore need $\delta < \frac{1}{|h^* (j)|}$ for all $j$, so that asking for ${\delta}<1/\max_j \mid h^*(j)\mid$ is enough.}
Let $\ell(\beta)(j)= {\mathrm{sign}}(\beta(j))$ be the sign of $\beta(j)$ so that $\ell(\beta)\in  {\{-1,0,1\}^N}$.
Thus, we need that $h$ satisfies:
\begin{eqnarray*}
	0&=& \sum_j \sum_m h(m)\alpha(m)  {E_N^n}(j,m) \ell(\beta(\alpha))(j) \\
	&=&\sum_m \alpha(m) h(m)\left\{  \sum_j  {E_N^n}(j,m)\ell(\beta(\alpha))(j)\right\}\\
	&\equiv& \sum_m h(m) a(\alpha)(m),
\end{eqnarray*}
where we defined
\[
a(\alpha)(m)=\alpha(m) \sum_j  {E_N^n}(j,m) \ell(\beta(\alpha))(j).\]
Thus, we can conclude that $h$ needs to be an element of the following $n-1$-dimensional subspace of $ {\openr^n}$:
\[
{\cal H}(\alpha)\equiv \{h-\langle h,a(\alpha)\rangle/\langle a(\alpha),a(\alpha)\rangle a(\alpha): h\}.\]
We now compute the pathwise derivative $d/d\delta R_n(\alpha_{\delta}^h)$ at $\delta=0$ and write it as
an inner product:
\[
{\frac{d}{d\delta}R_n(\alpha_{\delta}^h)\big\vert_{\delta=0}}=\langle D^*(\alpha), h\rangle,\]
where we restrict $D^*(\alpha)\in {\cal H}(\alpha)$. For any empirical risk it is straightforward to express  this derivative as $\langle D(\alpha),h\rangle$ and then $D^*(\alpha)=D(\alpha)-\Pi(D(\alpha)| a(\alpha))$ is obtained by  subtracting its projection onto $a(\alpha)$. A local steepest descent path is then defined by $\alpha_{\delta}^{D^*(\alpha)}=(1+\delta D^*(\alpha))\alpha$.

A universal steepest descent algorithm starts with an initial estimator $\alpha_n^0$ with $ {\lVert \beta(\alpha^0)\rVert_1}=C$; update $\alpha_n^0$ into $\alpha_n^1$ by moving a small amount $\delta$ along this path $\alpha_{\delta}^{D^*(\alpha)}$ decreasing the $R_n(\alpha)$; and iterate this updating process until convergence. At convergence, the resulting $\alpha_n$ satisfies $\langle D^*(\alpha_n),D^*(\alpha_n)\rangle =0$. Therefore at the limit $\alpha_n$ we have that the pathwise derivative of $R_n(\alpha_{n,\delta}^h)$ equals zero at $\delta=0$ for all allowed directions $h\in {\cal H}(\alpha_n)$. This proves that $\alpha_n$ is a local minimum of $R_n(\alpha)$ over the subset $\{\alpha:  {\lVert \beta(\alpha)\rVert_1}=C\}$.

{We suggest to first compute PC-HAR (which does not even need matrix inversion) for a given regularization parameter $\lambda_n$ and obtain $\beta_n(\alpha_n)$. We use this as a starting point to run the above steepest descent algorithm to compute PC-HAGL for this single $C_n(\lambda_n) = {\lVert \beta_n\rVert_1} $. We determine the cross-validation selector for this whole two-step procedure.} 
%In order to avoid having to run the above algorithm for a collection of $C$-values and choosing it with cross-validation, we suggest to first compute PC-HAR and determine the cross-validation selector for PC-HAR, resulting in a $\beta_n$ representing the PC-HAR fit. This is straightforward and fast since PC-HAR is nothing else than running ridge with an $n$-dimensional linear model. Given $\beta_n$ we compute $C_n=\pl \beta_n\pl_1$ and then run the above steepest descent algorithm to compute PC-HAL for this single $C_n$.
%To refine this, we can then select a few values $C_j\leq C_n$ (e.g., $C_j=\delta_j C_n$ with $\delta_j\in \{0.5,0.6,0.7,0.8,0.9,1\}$) and use cross-validation to select the $L_1$-norm constraint. 

%\subsection{Generalized Lasso}
%Recall the $N\times n$-matrix $E_N^n=(E(j,m):j=1,\ldots,N, m=1,\ldots,n)$ whose $(j,m)$-th  matrix element is the $j$-th component of the $m$-th eigenvector, $E(j,m)$, among the $n$ eigenvectors with non-zero eigenvalues.
%$\beta(\alpha)=E_N^n(\alpha)$ is $N\times 1$-vector with $j$-th component $\beta(\alpha)(j)$, and it is obtained by applying the matrix $E_N^n$ to $\alpha$.  The generalized LASSO optimizes
%the empirical risk over a model $\beta^t X$  plus $\lambda$ times an $L_1$-norm of a matrix transformation of the coefficient vector. Therefore, we can use the generalized LASSO software to optimize
%  \[
%R_n(\alpha,\lambda)\equiv R_n(\theta_{n,\alpha})+\lambda \pl E_N^n(\alpha)\pl_1.\]
%If such software is available it might be preferable over the above universal steepest descent algorithm. 

\subsection{{Gradients for specific risks}}

We facilitate the optimization of the PC-HAGL estimator by deriving the explicit gradients for common loss functions. We begin by establishing a general relationship between the coordinate-wise partial derivatives of the empirical risk and the derivative along a multiplicative path, denoted as $D(\alpha)(m)$.

\begin{proposition}
	\label{prop:general_gradient}
	Let $R_n(\alpha)$ be an empirical risk function that is differentiable with respect to the coefficient vector $\alpha$. Consider the multiplicative perturbation path defined by $\alpha_\delta^h(m)=(1+\delta h(m))\alpha(m)$ for a direction $h$. The coordinate of the gradient representer $D(\alpha)(m)$, corresponding to the derivative along this path at $\delta=0$, is given by {:} $D(\alpha)(m) = \alpha(m) \frac{\partial R_n(\alpha)}{\partial\alpha(m)}.$
\end{proposition}

We focus on two specific empirical risk functions, corresponding to regression and classification. Let $\theta_{n,\alpha} = \sum_m \alpha(m)\tilde{\phi}_m$ denote an arbitrary element in the PC-HA class.

\begin{enumerate}
	\item {Empirical mean squared error (MSE):} Used for regression tasks: $ R_{n}(\alpha) = P_n (Y - \theta_{n,\alpha})^2 $
	
	\item {Logistic risk:} Used for binary classification where $Y \in \{-1,1\}$. Following \cite{mccullagh}, the negative log-likelihood risk is: $R_{n}(\alpha) = P_n \log \left(1+\exp \left(-Y \theta_{n,\alpha}\right)\right)$
\end{enumerate}

Applying Proposition \ref{prop:general_gradient} to these specific risks yields the following closed-form expressions for $D(\alpha)(m)$.

\begin{proposition}[Derivatives for MSE and logistic risk]\label{prop:grad}
	For the empirical risk functions defined above, the gradient coordinates $D(\alpha)(m)$ are given by:
	
	\begin{description}
		\item[MSE:] $D(\alpha)(m) = -2 \alpha(m) P_n \left[ \tilde{\phi}_m \left(Y-\sum_k \alpha(k)\tilde{\phi}_k\right) \right] $
		
		\item[Logistic:]$D(\alpha)(m) = \alpha(m) P_n \left[ \frac{-Y \tilde \phi_m}{1+\exp \left(Y \sum_k \alpha(k) \tilde \phi_k \right)} \right] $
	\end{description}
\end{proposition}
% For example, for the empirical MSE we obtain $D(\alpha)(m)=2 \alpha(m) P_n \tilde{\phi}_m(Y-\sum_m \alpha(m)\tilde{\phi}_m)$.
% {Let $Y \in \{-1,1\}$. The empirical risk is in this case} \citep{mccullagh}: 

% \begin{equation}\label{eq:logistic}
	% P_n \log \left(1+\exp \left(-Y \sum_m \alpha(m) \tilde \phi_m\right)\right)
	% \end{equation}

% \begin{proposition}
	% Let $R_n(\alpha)$ be differentiable in $\alpha$. For the path $\alpha_\delta^h(m)=(1+\delta h(m))\alpha(m)$, the coordinate $D(\alpha)(m)$ is $\alpha(m)\,\frac{\partial R_n(\alpha)}{\partial\alpha(m)}.$
	% \end{proposition}

% \begin{proposition}\label{prop:grad}
	% For the empirical risk (\ref{eq:logistic}), we have 
	
	% $$D(\alpha)(m) = \alpha(m) P_n \left(\frac{-Y \tilde \phi_m}{1+\exp \left(Y \sum_m \alpha(m) \tilde \phi_m \right)}\right) $$
	% \end{proposition}

\section{Score equations solved by PC-HAGL}\label{section6}
%For the analysis of PC-HAGL we can assume that $\tilde{\phi}_m$, $m=1,\ldots,n$, satisfies the following:1) the linear span of $(\tilde{\phi}_m(x^n):m=1,\ldots,n)$ is identical to linear span of $(\phi_j(x^n):j)$; 2) $\tilde{\phi}_m$ are orthogonal  with respect to  inner product $\langle f_1,f_2\rangle_{n}\equiv P_n S_{\theta_n}(f_1)S_{\theta_n}(f_2)$. This choice of inner product $\langle f_1,f_2\rangle_n$ does not affect $D^{(k)}({\cal E}_n)$ and moreover does also not affect the MLE $\theta_n$ PC-HAGL  over $\pl \beta(\alpha)\pl_1<C$. This is interesting fact about the GL, while for standard LASSO and HAR it will have a real effect on the resulting MLE. This speaks in favor of the PC-HA-GL. This orthonormality will be helpful for establishing that the PC-HAGL solves the desired set of score equations. 

%In the proof below we work with a $\lambda_n$. I think we can keep it that but we should still add that by choosing the above score covariance inner product we can argue $\lambda_n$ might be fully bounded or at least be slow. $\lambda_n$ is all about bounding the variance of the score $S_{Q_0}(\sum_m h(m)\tilde{\phi}_m)$ by $\sum_m h(m)^2$.  SO by having this orthonormality w.r..t empirical score covariance we get close to this.

The analysis of plug-in PC-HA estimators of pathwise differentiable target parameters relies on their ability to solve score equations so that the linear span of these score equations approximates the canonical gradient of the pathwise derivative of the target parameter. The plug-in efficiency proof for HAL could thus be generalized to all PC-HA estimators if one can show that they solve a class of score equations whose linear span can be used to approximate any canonical gradient of a target feature. The pointwise asymptotic normality of $\hat{\Psi}_{PC-HA}(P_n)(x)$ as an estimator of $\Psi(P_0)(x)$ relies on the linear span of score equations approximating the scores of a sparse working model with the desired approximation error. We show that PC-HAs solve score equations for a large set of directions: exactly for directions in the tangent subspace and with error $o_P\left(n^{-1 / 2}\right)$ for any $f^* \in D^{(k)}\left( {{\cal E}_n}\right)$ with $\left\|f^*\right\|_2 \leq J_n^{-1 / 2}$.

% analogous to that established for HAL in \cite{van2023higher}, would also rely on the linear span of the score equations solved by the estimator approximating the scores of a sparse $J$-dimensional parametric working model that achieves the desired approximation error $O^+(1/J^{k+1})$ of the target function $\psi_0$---not on the estimator itself being an MLE of that sparse model. Such a working model exists by virtue of Theorem 2 in \cite{van2023higher}, which states that for any function in the class $D^{(k)}_M([0,1]^d)$ there exists a finite discrete measure (corresponding to a set of $J$ knot points/splines) such that the approximation error is bounded by $O(J^{-(k+1)})$ (up to logarithmic factors).

% This said,  PC-HAGL (say $k$-th order) needs to solve enough score equations so that the linear span of its scores approximates the scores of this growing sparse   parametric working model where the size $J$ of the working model is trading off bias $O^+(1/J^{k+1})$ with standard error $O((J/n)^{1/2})$.

We assume that $R_n(\psi)=P_n L(\psi)$ for a loss function $L()$.
PC-HAGL $\theta_{n,\alpha_n}$ optimizes  $R_n(\theta_{n,\alpha})$ over $\{\alpha \in \mathbb{R}^n:  {\lVert \beta(\alpha)\rVert_1}=C_n\}$ for a given $C_n$,  where $\theta_{n,\alpha}=\sum_m \alpha(m)\tilde{\phi}_m$ and $\tilde{\phi}_m=\sum_{j} {E_N^n}(j,m)\phi_j$ corresponding with the $n$ eigenvectors $E_m$, $m=1,\ldots,n$, of $H^{\top}H$ and $H(i,j)=\phi_j(x_i)$. 
For notational convenience, let $R_n(\alpha)=R_n(\theta_{n,\alpha})$.

\begin{theorem}[Approximate score solution for PC-HA]
	\label{score:gl}
	
	Let $\theta_n = \theta_{n,\alpha_n}$ be a PC-HA estimator. Let $\mathcal{E}(J_n) = \{m : \alpha_n(m) \neq 0\}$ denote the set of active PC-basis indices.
	We define the geometry of the constraint boundary as follows:
	\begin{enumerate}
		\item {Constraint gradient:} Define the vector $a_{1n} \in \mathbb{R}^{| \mathcal{E}(J_n) |}$ coordinate-wise by
		\[ a_{1n}(m) = \sum_{j \in  {\mathcal{R}_N}}  {E_N^n}(j,m) \ell(\beta(\alpha_n)(j)), \]
		where $\ell(\cdot)$ denotes the subgradient of the $L_1$ norm.
		\item {Unsolved direction:} Let $f_{a_n}$ be the normalized function in the PC-space corresponding to the constraint gradient:
		\[ f_{a_n} = \frac{1}{\|a_{1n}\|_2} \sum_{m \in \mathcal{E}(J_n)} a_{1n}(m) \tilde{\phi}_m, \quad \text{satisfying } \|f_{a_n}\|_2 = 1. \]
		Define the scaled version $f_{a_n}^* = n^{-1/2} f_{a_n}$, which satisfies $\|\beta(f_{a_n}^*)\|_1 = O_P(1)$.
		\item {Tangent subspace:} Let $D_n(\mathcal{E}(J_n)) = \{ f \in D(\mathcal{E}(J_n)) : \langle f, f_{a_n} \rangle = 0 \}$ be the subspace of score directions orthogonal to the constraint gradient. Let $\tilde{f}_{a_n}$ denote the $L^2(P_0)$-projection of $f_{a_n}^*$ onto this subspace.
	\end{enumerate}
	
	Furthermore, we make the following assumptions: 
	
	\begin{enumerate}
		\item {Complexity:} The class of score functions $\mathcal{F}_{score} = \{ \frac{d}{d\theta} L(\theta)(f) : f \in D_M(\mathcal{E}(J_n)), \theta \in D_M(\mathcal{E}(J_n)) \}$ satisfies the same entropy bound conditions as the primary function class $D_M^{(k)}([0,1]^d)$.
		\item {Projection approximation:} The projection error of the unsolved direction is controlled by the convergence rate of the estimator. Specifically, we assume either:
		\[ \| \theta_n - \psi_0 \|_\infty \| f_{a_n}^* - \tilde{f}_{a_n} \|_{1, P_0} = o_P(n^{-1/2}), \]
		or, using the loss-based dissimilarity:
		\[ d_0^{1/2}(\theta_n, \psi_0) \| f_{a_n}^* - \tilde{f}_{a_n} \|_{P_0} = o_P(n^{-1/2}). \]
	\end{enumerate}

	Then, the estimator $\theta_n$ solves the empirical score equations exactly for all directions in the tangent subspace:
	\[ P_n \frac{d}{d\theta_n} L(\theta_n)(h) = 0 \quad \text{for all } h \in D_n(\mathcal{E}(J_n)). \]
	Furthermore, for any target function $f^* \in D_M(\mathcal{E}(J_n))$ with bounded Euclidean norm $\|f^*\|_2 \le J_n^{-1/2}$, the empirical score is asymptotically negligible:
	\[ \sup_{f^* \in D_M(\mathcal{E}_n), \|f^*\|_2 \le J_n^{-1/2}} \left| P_n \frac{d}{d\theta_n} L(\theta_n)(f^*) \right| = o_P(n^{-1/2}). \]
\end{theorem}

\subsection{Application of Theorem \ref{score:gl} to least squares PC-HAGL.}
Consider the MSE $P_n (Y-\sum_m \alpha_{n,\delta}^h(m)\tilde{\phi}_m)^2$ with 
$h\in {\cal H}(\alpha_n)$. The scores are given by
\[
0=P_n S_h(\alpha_n)=2 P_n \sum_{m\in {\cal E}(J_n)}(h-\Pi(h\mid a_1(\alpha_n)))(m)\tilde{\phi}_m\left(Y-\sum_m \alpha_n(m)\tilde{\phi}_m\right)\]
for all $h\in \openr^n$ with $h=hI(\alpha_n\not =0)$. So $\alpha_n$ solves the following score equations:
\begin{eqnarray*}
	0&=&P_n \left(\sum_m h(m)\tilde{\phi}_m-\sum_m h(m) \frac{a_1(\alpha_n)(m)}{ {\lVert a_1(\alpha_n)\rVert_2^2}} \sum_m a_1(\alpha_n)(m)\tilde{\phi}_m)  \right)\\ 
	&& \hspace*{6cm} \left(Y-\sum_m \alpha_n(m)\tilde{\phi}_m\right).\end{eqnarray*}

Consider a score equation $P_n \sum_m f^*(m)\tilde{\phi}_m (Y-\sum_m \alpha_n(m)\tilde{\phi}_m)$ where 
$ {\lVert \beta(f^*)\rVert_1}=O(1)$ and $f^*(m)=0 $ for $m\not\in {\cal E}(J_n)$.
We assume $ {\lVert \beta(f^*)\rVert_2}=O(n^{-1/2})$ which is equivalent with assuming $ {\lVert f^*\rVert_2}=O(n^{-1/2})$.
%\begin{theorem}Consider an $f^*\in D^{(k)}({\cal E}(J_n))$ so that $\pl f^*\pl_2=O({n^{-1/2})$, which implies (and corresponds with) $\pl \beta(f^*)\pl_1=O(1)$. We assume that the number of non-zero coefficients in $\beta(f^*)$ is $\sim n$, recall$\beta(f^*)(j)=\sum_{m\in {\cal E}_n}E(j,m)f^*(m)$. 
Define $f_{a_n}=\sum_m a(\alpha_n)(m)\tilde{\phi}_m/ {\lVert a(\alpha_n)\rVert_2}$.
As shown above, 
\begin{eqnarray*}
	\mid P_n \sum_m f^*(m)\tilde{\phi}_m(Y-\theta_{n,\alpha_n}(X)) \mid &\leq&
	\mid P_n f_{a_n}^*(Y-\theta_n)\mid 
\end{eqnarray*}
Application of the above theorem demonstrates the conditions under which  \[
\sup_{ {\lVert f^*\rVert_2}=O(n^{-1/2})}\mid P_n f^*(Y-\theta_n(X))\mid = o_P(n^{-1/2}).\]

\section{{Plug-in asymptotic normality and efficiency of PC-HA for pathwise differentiable features}}

In previous sections we established the rate of convergence of the PC-HA estimators w.r.t. loss-based dissimilarity, which generally behaves as the square of an $L^2$-norm. In this section we discuss how the established plug-in efficiency for regular HAL of pathwise differentiable features of the target function $\psi_0$ also holds for the PC-HA estimators. 
Beyond having an initial rate of convergence w.r.t.\ loss-based dissimilarity, the other key condition of this proof is that we need the PC-HA estimator to approximately solve the efficient influence-curve score equation. Given Theorem \ref{score:gl}  {and} the capability of $\theta_n$ to solve a large class of score equations that approximate any desired score equation $P_n S_{\theta_n}(f^*)$ with $f^*\in D^{(k)}({\cal E}_n)$, we can now argue that indeed PC-HA will solve this key efficient score equation at the desired level. 

Let $\Phi:{\cal M}\rightarrow\openr$ be a real valued pathwise differentiable parameter with canonical gradient $ {D^*_{\Phi,P}}$
at $P\in {\cal M}$.
Suppose that $\Phi(P)=\Phi^F(\Psi(P))$ is a function of $\Psi(P)=\arg\min_{\psi\in D^{(k)}([0,1]^d)}PL(\psi)$.
Moreover, suppose that the loss function $L(\psi)$ is such that its collection of scores
\[
\left\{ S_{\Psi(P)}(f)=d/d\Psi(P) L(\Psi(P))(f):f\in D^{(k)}_M([0,1]^d)\right\}\]  contains
$ {D^*_{\Phi,P}}$. 
That means that the loss function 1) behaves as a log-likelihood loss so that its scores are scores in the tangent space of the model and 2) its collection of scores is rich enough to cover the canonical gradient of our target parameter $\Phi(P)$.
Then, we can write $ {D^*_{\Phi,P}}=S_{\Psi(P)}(f_P)$ for some $f_P\in D^{(k)}_M([0,1]^d)$.
This allows us to parametrize $ {D^*_{\Phi,P}}= {S_{\Psi(P)}(f_P)}\equiv D^*_{\Psi(P),f_P}$.
Let $R_{\Phi()}(P,P_0)=\Phi(P)-\Phi(P_0)+P_0 {D^*_{\Phi,P}}$ be the exact second order remainder. Assume that 
we can parametrize this remainder as $R_{\Phi(),0}(\Psi(P),f(P),\psi_0,f_0)$ involving integrals of square differences $\psi-\psi_0$ and $f(P)-f_0$ or cross-products of these two differences. 

\begin{theorem}[Asymptotic efficiency of plug-in PC-HA estimators]\label{thm:efficiency}
	Let $\Phi: \mathcal{M} \to \mathbb{R}$ be a pathwise differentiable parameter with canonical gradient $D^*_{\Phi,P}$ at $P \in \mathcal{M}$. Assume the canonical gradient can be represented as a score in the tangent space: $D^*_{\Phi,P} = S_{\Psi(P)}(f_P)$ for some index function $f_P \in D^{(k)}_M([0,1]^d)$.
	Let $\theta_n$ be a PC-HA estimator converging to $\psi_0 = \Psi(P_0)$ at a rate $r_n = n^{-k^*/(2k^*+1)}$.
	
	Assume the following conditions hold for the estimator $\theta_n$ and the true influence curve component $f_0 = f_{P_0}$:
	\begin{enumerate}
		\item {Approximate score solution:} The estimator $\theta_n$ solves the empirical score equation along the direction of the canonical gradient:
		\[ P_n D^*_{\theta_n, f_0} = o_P(n^{-1/2}). \]
		\textit{(Note: This condition is satisfied if the PC-basis is sufficiently rich to approximate $f_0$ and the regularization parameter allows for adequate fit, as discussed in Supplement \ref{app:score}.)}
		
		\item {Empirical process regularity:} The class of plug-in influence curves $\{ D^*_{\theta, f_0} : \theta \in D^{(k)}_M(\mathcal{E}_n) \}$ is $P_0$-Donsker, and the $L_2(P_0)$ convergence holds:
		\[ P_0 \left( D^*_{\theta_n, f_0} - D^*_{\psi_0, f_0} \right)^2 \xrightarrow{P} 0. \]
		
		\item {Second-order remainder:} The linearization remainder satisfies:
		\[ R_{\Phi}(\theta_n, f_0, \psi_0, f_0) = \Phi(\theta_n) - \Phi(\psi_0) + P_0 D^*_{\theta_n, f_0} = o_P(n^{-1/2}). \]
	\end{enumerate}

	Under these conditions, the plug-in estimator $\Phi(\theta_n)$ is asymptotically linear and efficient:
	\[ \Phi(\theta_n) - \Phi(\psi_0) = \frac{1}{n} \sum_{i=1}^n D^*_{\Phi, P_0}(O_i) + o_P(n^{-1/2}). \]
	Consequently, $\sqrt{n}(\Phi(\theta_n) - \Phi(\psi_0)) \xrightarrow{d} \mathcal{N}(0, \sigma^2_{eff})$, where $\sigma^2_{eff} = \text{Var}_{P_0}(D^*_{\Phi, P_0})$.
\end{theorem}

%There is an interesting point to make when $f_0$ is not as smooth but e.g element of $D^{(0)}([0,1]^d)$. What can now go wrong? or does it still work. Maybe now harder to approximate $f_0$ with $\tilde{f}_{0,n}\in D^{(k)}([0,1]^n)$ and we have not solved the scores for $f\in D^{(0)}({\cal E}_n)$ or so.

\section{{Pointwise asymptotic normality of PC-HA as estimators of the target function itself}}
{We exploit the fact that all our PC-HA estimators approximately solve the score equations over the working model $D^{(k)}({\cal E}_n)$. This score equation property is crucial for deriving pointwise asymptotic normality, since it allows us to treat $\theta_n$ as an (approximate) $M$-estimator for a suitable finite-dimensional working model and to express $(\theta_n-\psi_{0,n})(x_0)$ as a sum of empirical means of scores plus small remainder terms.}{A key novelty of the argument is that we run sums of empirical process terms over a sparse working model that is fixed (non-random) and whose existence is guaranteed by Theorem 2 in \cite{van2023higher}.} 

{ {We define 
		${\cal R}_{0,n}=\{\phi_j:1\leq j\leq J_n\}$ as a collection of $J_n$ spline basis functions such that $D^{(k)}({\cal R}_{0,n})\subset D^{(k)}({\cal R}_N)$ is the sparsest model satisfying: (i) \emph{Approximation:} $\sup_{f \in D^{(k)}_M({\cal R}_N)} \inf_{\psi \in D^{(k)}_M({\cal R}_{0,n})} \| f - \psi \|_\infty = O_P(n^{-k^*/(2k^*+1)})$; (ii) \emph{Oracle convergence:} the oracle MLE $\psi_{0,n}=\arg\min_{\psi\in D^{(k)}({\cal R}_{0,n})}P_0 L(\psi)$ satisfies $\| \psi_{0,n} - \psi_0 \|_\infty = o_P((J_n/n)^{1/2})$. We choose $J_n \asymp n^{1/(2k^*+1)}$ (up to powers of $\log n$) so that the approximation bias satisfies
		$ {\lVert} \psi_{0,n}-\psi_0 {\rVert_\infty} = o_P\!\left(n^{-k^*/(2k^*+1)}\right)$; since the sharp rate is
		$n^{-k^*/(2k^*+1)}$ up to $\log n$ factors, suppressing these logs only strengthens the statement and keeps the bias of smaller order than the leading stochastic term. Existence of such a sparse working model is guaranteed by Theorem 2 in \cite{van2023higher}. We then obtain an orthonormal basis ${\cal R}_{0,n}^*=\{\phi_j^*:1\leq j\leq J_n\}$ by a linear change of basis. Since this change of basis preserves the linear span, we have $D^{(k)}({\cal R}_{0,n}^*)=D^{(k)}({\cal R}_{0,n})$, and with a slight abuse of notation we will also denote the orthonormal basis by ${\cal R}_{0,n}$ in what follows. We will use the same notation ${\cal R}_{0,n}$ both for the index set $\{1,\ldots,J_n\}$ and for the corresponding (orthonormal) basis $\{\phi_j^* :1\leq j\leq J_n\}$. Thus, whenever we write $\sum_{j\in {\cal R}_{0,n}}$, this is shorthand for summing over indices $j=1,\ldots,J_n$, with $\phi_j^*$ denoting the basis element associated with index $j$.}}

The analysis proceeds by a first-order Taylor expansion of the score equations established in Section 3, followed by a decomposition of the error into a leading empirical process term and asymptotically negligible remainders. {Specifically, we define the projection of the estimator onto the sparse working model $D^{(k)}\left(\mathcal{R}_{0, n}\right)$ and orthogonalize the basis with respect to the covariance inner product.}
{In the Supplement \ref{proof:asymnor}, the basis {$\{\phi_j: 1 \leq j \leq J_n\}$} is orthogonalized w.r.t.\ the inner product $\langle f_1,f_2\rangle_{sc,P_0}\equiv P_0 \frac{d}{d\delta_0}S_{\psi_{0,n}+\delta_0 f_2}(f_1)$ to obtain an orthonormal basis {$\{\phi_j^*: 1 \leq j \leq J_n\}$}. Write $\phi^{*,n}(x_0)=(\phi_j^*(x_0): 1\leq j \leq J_n)^{\top}$. The normalized representer of the evaluation functional at $x_0$ is $\bar{\phi}_{n,x_0}^*\equiv J_n^{-1/2}\sum_{j\in {\cal R}_{0,n}}\phi_j^*(x_0)\phi_j^*$. The projected score (influence curve) at $x_0$ is}
{$$IC_{0,n}(x_0) \equiv S_{\psi_{0,n}}(\bar{\phi}_{n,x_0}^*).$$}
{The $J_n\times J_n$ covariance matrix of the score vector $S_{\psi_{0,n}}(\phi^{*,n})$ under $P_0$ is $\Sigma_{0,n}\equiv P_0 S_{\psi_{0,n}}(\phi^{*,n}) S^{\top}_{\psi_{0,n}}(\phi^{*,n})$.} Let $\sigma_{0,n}^2(x_0)$ denote the variance of this projected score at $x_0$,
{$$\sigma_{0,n}^2(x_0) \equiv \Var_0\{IC_{0,n}(x_0)\}= P_0\{S_{\psi_{0,n}}(\bar{\phi}_{n,x_0}^*)\}^2 = J_n^{-1}\phi^{*,n}(x_0)^{\top}\Sigma_{0,n}\phi^{*,n}(x_0). $$}
{We derive an asymptotically linear representation of the estimator around the sparse oracle MLE.
	$$\psi_{0,n}=\arg\min_{\psi\in D^{(k)}({\cal R}_{0,n})}P_0L(\psi)$$

	The key quantities are the following. The sparse working model $D^{(k)}({\cal R}_{0,n})$ is a fixed $J_n$-dimensional spline model (existence guaranteed by Theorem 2 in \cite{van2023higher}) with orthonormal basis {$\{\phi_j^*: 1\leq j \leq J_n\}$}; the empirical score residual on that basis $r_n(\phi^{*,n})\equiv P_n S_{\theta_n}(\phi^{*,n})$, 
	\[R_{1n}(\phi^{*,n})\equiv -P_0\left\{S_{\tilde{\theta}_n}(\phi^{*,n})-S_{\psi_{0,n}}(\phi^{*,n})-\frac{d}{d\beta_{0,n}}S_{\beta_{0,n}}(\phi^{*,n})(\tilde{\beta}_n-\beta_{0,n})\right\} \quad \textrm{(first-order Taylor remainder)},
	\]
	\[R_{2n}(\bar{\phi}_{n,x_0}^*)=n^{1/2}(P_n-P_0)\{S_{\theta_n}(\bar{\phi}_{n,x_0}^*)-S_{\psi_{0,n}}(\bar{\phi}_{n,x_0}^*)\} \quad \textrm{(empirical process remainder)},
	\]
	
	and 
	\[
	r_n(x_0)\equiv n^{1/2}r_n(\bar{\phi}_{n,x_0}^*) -n^{1/2}R_{1n}(\bar{\phi}_{n,x_0}^*) +n^{1/2}R_{2n}(\bar{\phi}_{n,x_0}^*).
	\]
	
	{The detailed derivations are in the Supplement \ref{proof:asymnor}.}

	\begin{theorem}[Pointwise asymptotic normality and uniform convergence]\label{thm:normality}
		Let $\theta_n$ be the PC-HA estimator obtained by constraining a chosen norm ($ {\lVert\alpha\rVert_2}$, $ {\lVert\alpha\rVert_1}$, or $ {\lVert\beta(\alpha)\rVert_2}$) by a cross-validated bound $C_n$. Let $D^{(k)}(\mathcal{R}_{0,n})$ and $\psi_{0,n}$ be the sparse working model and oracle MLE.
		
		\begin{enumerate}
			\item {Variance stability: }
			We assume $\sigma_{0,n}^2(x_0) = O_P(1)$ and is bounded away from zero.
			
			\item {Regularity:} We assume $r_n(x_0) = o_P(1)$.
		\end{enumerate}
		
		Under the variance stability and regularity conditions above, the following results hold:
		
		\begin{enumerate}
			\item {Asymptotically linear representation:}
			The estimator admits the following linear expansion at any point $x_0 \in [0,1]^d$:
			\[ \sqrt{\frac{n}{J_n}}(\theta_n - \psi_{0,n})(x_0) = \sqrt{n}(P_n - P_0) S_{\psi_{0,n}}(\bar{\phi}_{n,x_0}^*) + o_P(1). \]
			{Equivalently, the leading term can be written as the empirical mean of the influence curve $IC_{0,n}(x_0)$ defined in condition 1 above:}
			{$$ \sqrt{\frac{n}{J_n}}(\theta_n - \psi_{0,n})(x_0) = \sqrt{n}(P_n - P_0) IC_{0,n}(x_0) + o_P(1).$$}
			
			\item {Pointwise asymptotic normality:}
			The standardized estimator converges in distribution to a standard normal:
			\[ \frac{\sqrt{n/J_n}(\theta_n(x_0) - \psi_0(x_0))}{\sigma_{0,n}(x_0)} \xrightarrow{d} \mathcal{N}(0,1). \]
			Consequently, $\theta_n(x_0) \pm 1.96 \hat{\sigma}_n(x_0) \sqrt{J_n/n}$ is a valid asymptotic 95\% confidence interval.
			
			\item {Uniform convergence and simultaneous bands:}
			If the remainder conditions hold uniformly over a subset $\mathcal{A} \subset [0,1]^d$, then:
			\[  {\lVert} \theta_n - \psi_{0,n}  {\rVert}_{\infty, \mathcal{A}} = O_P\left(\sqrt{\frac{J_n \log n}{n}}\right). \]
			Furthermore, $\theta_n(x) \pm \sqrt{\log n} \cdot 1.96 \hat{\sigma}_n(x) \sqrt{J_n/n}$ constitutes a valid simultaneous 95\% confidence band.
		\end{enumerate}
	\end{theorem}
	
	While Theorem \ref{thm:normality} establishes asymptotic normality, the limiting variance $\sigma_{0, n}^2\left(x_0\right)$ depends on the projection onto the unknown sparse oracle model $D^{(k)}\left(\mathcal{R}_{0, n}\right)$. Consequently, a closed-form plug-in variance estimator is difficult to construct without explicit knowledge of the sparsity pattern. For practical inference, we recommend using the nonparametric bootstrap. By fixing the PC-working model $D^{(k)}\left(\mathcal{E}_n\right)$ and bootstrapping the PC-HA optimization procedure, one can obtain valid pointwise confidence intervals and simultaneous bands that mirror the asymptotic theory.
	
	\section{{Empirical results}}
	\subsection{ {Norm scaling}}

	The purpose of this simulation study is to investigate how the different cross-validated  {PC-HA} estimators control the complexity of the induced HAL representation as the sample size increases. In particular, we study whether penalizing the PC coefficients $\alpha$ is sufficient to control the sectional variation norm of the implied HAL coefficient vector $\beta(\alpha)$.
	
	The central question is: does cross-validation over an $L_1$ or $L_2$ constraint on $\alpha$ implicitly keep $\|\beta(\alpha)\|_1$ bounded, or can $\|\beta(\alpha)\|_1$ grow with $n$?

	If $\|\beta(\alpha_n)\|_1$ diverges, the estimator effectively belongs to a sequence of expanding HAL model classes with sectional variation bound $M_n \to \infty$, which may impact entropy bounds and convergence rates. We therefore study the empirical scaling with $n$ of
	\[
	\|\alpha_n\|_2, 
	\quad 
	\|\alpha_n\|_1, 
	\quad 
	\|\alpha_n\|_\infty, 
	\quad 
	\|\beta(\alpha_n)\|_1,
	\quad 
	J_n = \#\{j:  {\alpha_n(j)} \neq 0\}.
	\]
	
	Let $J_n$ denote the number of selected PC coefficients and define $r_n = \max_{j \le J_n}  {\lvert \alpha_n(j)\rvert}.$
	Then $\|\alpha_n\|_1 \asymp J_n r_n, 
	\|\alpha_n\|_2^2 \asymp J_n r_n^2.$ Because the PC basis is orthonormal, $\|\beta(\alpha_n)\|_2 = \|\alpha_n\|_2$. Thus, the parametric stochastic order $\|\alpha_n\|_2^2 = O_P(n^{-1})$ implies $J_n r_n^2 = O_P(n^{-1})$ and yields $r_n = O_P\!\left(\frac{1}{\sqrt{n J_n}}\right) {,}
	$ so that 
	\begin{equation}\label{onenorm}
		\|\alpha_n\|_1 = O_P\!\left(\sqrt{\frac{J_n}{n}}\right).
	\end{equation}
	
	If $J_n = n^\gamma$, then $\|\alpha_n\|_1 = O(n^{(\gamma-1)/2})$. Hence, $\|\alpha_n\|_1$ may shrink to zero even while $\|\beta(\alpha_n)\|_1$ remains bounded. In contrast, if cross-validation selects $\|\alpha_n\|_2^2$ decaying slower than $1/n$, then $\|\beta(\alpha_n)\|_1$ must grow with $n$, implying an expanding effective sectional variation bound. The simulations are designed to determine which regime occurs in practice under the different penalties. For each sample size $n$, we generate $X_i \sim  {\mathrm{Unif}}(0,1), Y_i = \frac{2\sin(8\pi X_i^2)}{X_i} + \varepsilon_i,
	\varepsilon_i \sim N(0, 2^2).$
	This regression function exhibits rapid oscillations and localized curvature, requiring a rich spline representation. We estimate empirical scaling exponents by regressing $\log$-transformed quantities on $\log n$.
	
	{In the Supplement (Figure~\ref{fig:scaling_all}) we report the scaling behavior of different norms and metrics as a function of sample size $n$: each row corresponds to a different quantity---(1) $\|\alpha_n\|_2$ vs $n$, (2) $\|\alpha_n\|_1$ vs $n$, (3) $\|\alpha_n\|_\infty$ vs $n$, (4) $\|\beta(\alpha_n)\|_1$ vs $n$, and (5) $J_n$ (number of selected coefficients) vs $n$---and each column to a regularization method (PC-HAGL, PC-HAL, PC-HAR); all plots are log-log.}
	
	Figure \ref{fig:scaling_all} reports the observed scaling behavior. The simulation reveals three distinct regimes of behavior regarding how each estimator manages the complexity  {budget}. While all three estimators satisfy the fundamental  {non-overfitting} requirement $\|\beta(\alpha_n)\|_1 = O(1)$ asymptotically, their finite-sample trajectories differ significantly.
	
	The PC-HAGL estimator demonstrates superior adaptivity, particularly in the finite-sample regime ($n < 1000$). Cross-validation actively selects \textit{lower} sectional variation norms when data is scarce, effectively undersmoothing relative to the asymptotic target to minimize Mean Squared Error (MSE). This indicates that less complexity is needed to optimize the bias-variance tradeoff in small samples. Consequently, PC-HAGL emerges as the {least overfitting option}, progressively releasing the complexity constraint only as $n$ increases to converge to the true oracle complexity ($\|\beta\|_1 \approx 55$).

	The PC-HAR estimator consistently selects the {highest} sectional variation norm ($\|\beta\|_1 \approx 60$). This inflation is structural rather than adaptive: because the $L_2$ penalty cannot set coefficients to zero, the estimator retains the entire ``spectral tail'' of high-frequency basis functions. Even though these tail components contribute minimally to the fit, their non-zero weights aggregate to inflate the variation norm. Thus, PC-HAR satisfies the $O(1)$ bound but settles at an inefficiently high constant due to this ``spectral dead weight.''
	
	The PC-HAL estimator appears structurally rigid, selecting a sectional variation norm that remains roughly uniform across sample sizes ($\|\beta\|_1 \approx 55$). Unlike PC-HAGL, it lacks the flexibility to reduce its complexity significantly in small samples because the sparsity constraint forces the few selected coefficients to be large (``spiky'') to capture the signal, essentially locking the estimator into a fixed level of roughness.
	
	A quantitative analysis of the decay rates provides striking corroboration of our theoretical derivation for sparse estimators.
	
	Recall the derived necessary condition for bounding the variation norm in a sparse model  {(\ref{onenorm})}.
	In our simulation for PC-HAL, we observe that the number of selected components scales approximately as $J_n \asymp n^{1/5}$ (empirically $n^{0.193}$). Substituting this sparsity rate into our derived condition yields a predicted decay rate for the $L_1$ norm of the coefficients $\|\alpha_n\|_1 \asymp \sqrt{\frac{n^{1/5}}{n}} = \sqrt{n^{-4/5}} = n^{-2/5} = n^{-0.4}$. This theoretical prediction matches the empirical results with remarkable precision. The observed slope of $\log(\|\alpha_n\|_1)$ versus $\log(n)$ is {-0.426}, almost exactly the predicted {-0.4}. Furthermore, the $L_2$ and $L_\infty$ norms of the coefficients are observed to decay at the standard parametric rate ($n^{-1/2}$). 
	
	\subsection{Convergence rates}
	
	In this section, we empirically verify the convergence rates of the zero-order proposed PC-HA estimator. We focus on confirming two central theoretical properties: the dimension-free nature of the convergence rate for functions of bounded sectional variation, and the adaptivity of the estimator to the spectral structure of the target function.

	We consider a regression setting $Y = \psi_0(X) + \epsilon$, where $X \sim \text{Uniform}([0,1]^d)$ and $\epsilon \sim \mathcal{N}(0, \sigma^2)$ with $\sigma=0.3$. The estimator $\psi_n$ is the zero-order PC-HA for the different regularization modes, where  {the} hyperparameter is selected via $3$-fold cross-validation. We evaluate the performance using the Mean Squared Error (MSE) computed on an independent test set of size $ {N_{\mathrm{test}}}=1000$: $\text{MSE}_n = \frac{1}{ {N_{\mathrm{test}}}} \sum_{i=1}^{ {N_{\mathrm{test}}}} (\psi_n(X_i) - \psi_0(X_i))^2.$
	We investigate two distinct ground truth functions $\psi_0$ to illustrate different convergence regimes:
	\begin{enumerate}
		\item {Linear ground truth:} $\psi_0(x) = d^{-1/2} \sum_{j=1}^d x_j$. This function represents a standard target within the class of bounded variation functions but does not align perfectly with the spectral basis of the zero-order kernel (Figure \ref{fig:convergence_mean_nonfast}). 
		\item {Harmonic ground truth:} $\psi_0(x) = d^{-1/2} \sum_{j=1}^d \sin(2\pi x_j)$. This function possesses a structure that is intimately related to the eigenfunctions of the underlying covariance kernel used in the PC decomposition (Figure \ref{fig:convergence_mean_fast}  {in the Supplement}).
	\end{enumerate}
	Sample sizes range from $n=400$ to $n=1500$, and we perform 10 repetitions for dimensions $d \in \{3, 10, 20\}$.

	Theorem \ref{theoremone} establishes the convergence rate for the loss-based dissimilarity $d_0(\psi_{n}, \psi_{0})$. In the context of squared error loss, the loss-based dissimilarity $d_0(\psi, \psi_0) = P_0 L(\psi) - P_0 L(\psi_0)$ is equivalent to the squared $L_2(P_0)$-norm $\|\psi - \psi_0\|_{P_0}^2$, which corresponds precisely to the MSE computed on the noise-free test set. Theorem \ref{theoremone} dictates that this dissimilarity decays at the rate $n^{-\frac{2k^*}{2k^*+1}}$ (modulo terms involving powers of $\log n$ depending on $d$) where $k^* = k+1$ depends on the smoothness order of the spline basis. For the zero-order PC-HAGL estimator utilized in this study ($k=0$), we have $k^*=1$, implying a theoretical minimax rate of $n^{-2/3}$. Consequently, the MSE is expected to scale as $C n^{-2/3}$. Taking the logarithm of this power-law relationship yields $\log(\text{MSE}_n) \approx \log(C) - \frac{2}{3} \log(n)$, justifying the use of the slope in a log-log plot to empirically verify the convergence rate. Formally, this implies $\text{MSE}_n \approx C n^{-\gamma}.$
	Taking the natural logarithm of both sides yields a linear relationship $\log(\text{MSE}_n) \approx \log(C) - \gamma \log(n).$
	Consequently, in a log-log plot of MSE versus sample size $n$, the slope of the regression line corresponds to $-\gamma$, the negative of the convergence rate. We compare our empirical slopes against the theoretical minimax rate for zero-order HAL estimators derived in Theorem 2. Specifically, for a smoothness order $k=0$, we have $k^* = k+1 = 1$, yielding a critical rate of $n^{-2k^*/(2k^*+1)} = n^{-2/3} \approx n^{-0.66}$.

	The simulation results yield two primary conclusions regarding the efficiency and adaptivity of the PC-HA estimator. The first key finding is that the \textit{convergence rate does not degrade as the dimension $d$ increases}, validating the central premise of the HAL theory. Standard nonparametric smoothers typically suffer from the curse of dimensionality, with rates scaling as $n^{-2/(2+d)}$. In contrast, our PC-HA estimators maintain a slope steeper (more negative) than the theoretical baseline of $-2/3$ across all tested dimensions ($d=3, 10, 20$). While the estimators using $L_1$ and $L_2$ norms on the coefficients perform comparably in lower dimensions ($d=3, 5$), the explicit control of the sectional variation norm (PC-HAGL) proves most robust in maintaining the dimension-free rate in $d=10$.

	As for adaptivity, the second conclusion arises from the comparison between the linear and sine ground truths. For the sine target $\psi_0(x) = d^{-1/2} \sum \sin(2\pi x_j)$, we observe convergence rates approaching the parametric rate $n^{-1}$, significantly outperforming the worst-case nonparametric rate of $n^{-2/3}$.
	
	This phenomenon can be attributed to the spectral properties of the zero-order kernel. The principal components used in our working model are derived from the design matrix of zero-order splines, which we conjecture to be asymptotically equivalent to the covariance operator of a Brownian motion. Consequently, the target function $\psi_0$ acts as a (perhaps almost) finite linear combination of the leading basis functions of our PC working model. Because the Fourier series of the sine target decays instantly (it is sparse in the eigenbasis), the bias term in the MSE is negligible, and the error is dominated by the variance term $O_P(n^{-1})$.
	
	Conversely, for the linear target $\psi_0(x) = \sum x_j$, the representation in the sinusoidal PC basis requires a dense expansion (akin to the slow decay of the Fourier series of a sawtooth wave). In this case, the estimator must trade off bias and variance, resulting in a rate closer to the theoretical nonparametric baseline of $n^{-2/3}$. This demonstrates that HAL theory is not only minimax optimal but also adaptive: it achieves fast rates $O_P(n^{-1})$ when the target function aligns well with the PC basis, and maintains the rate $O_P(n^{-2/3})$ otherwise.

	\begin{figure}[t!]
		\centering
		
		% Row d = 3
		\begin{minipage}{\textwidth}
			\centering
			\includegraphics[width=0.3\textwidth]{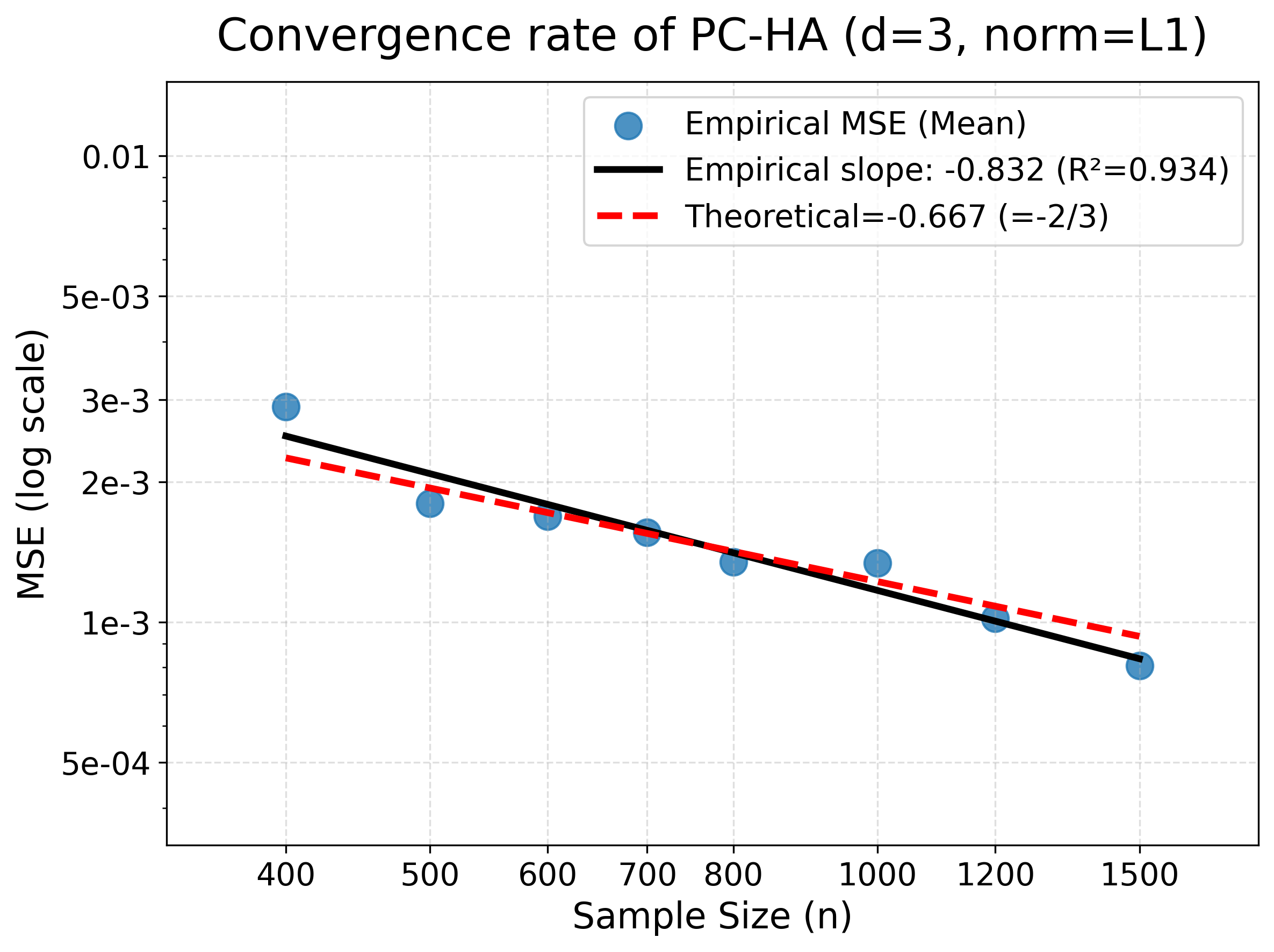}%
			\includegraphics[width=0.3\textwidth]{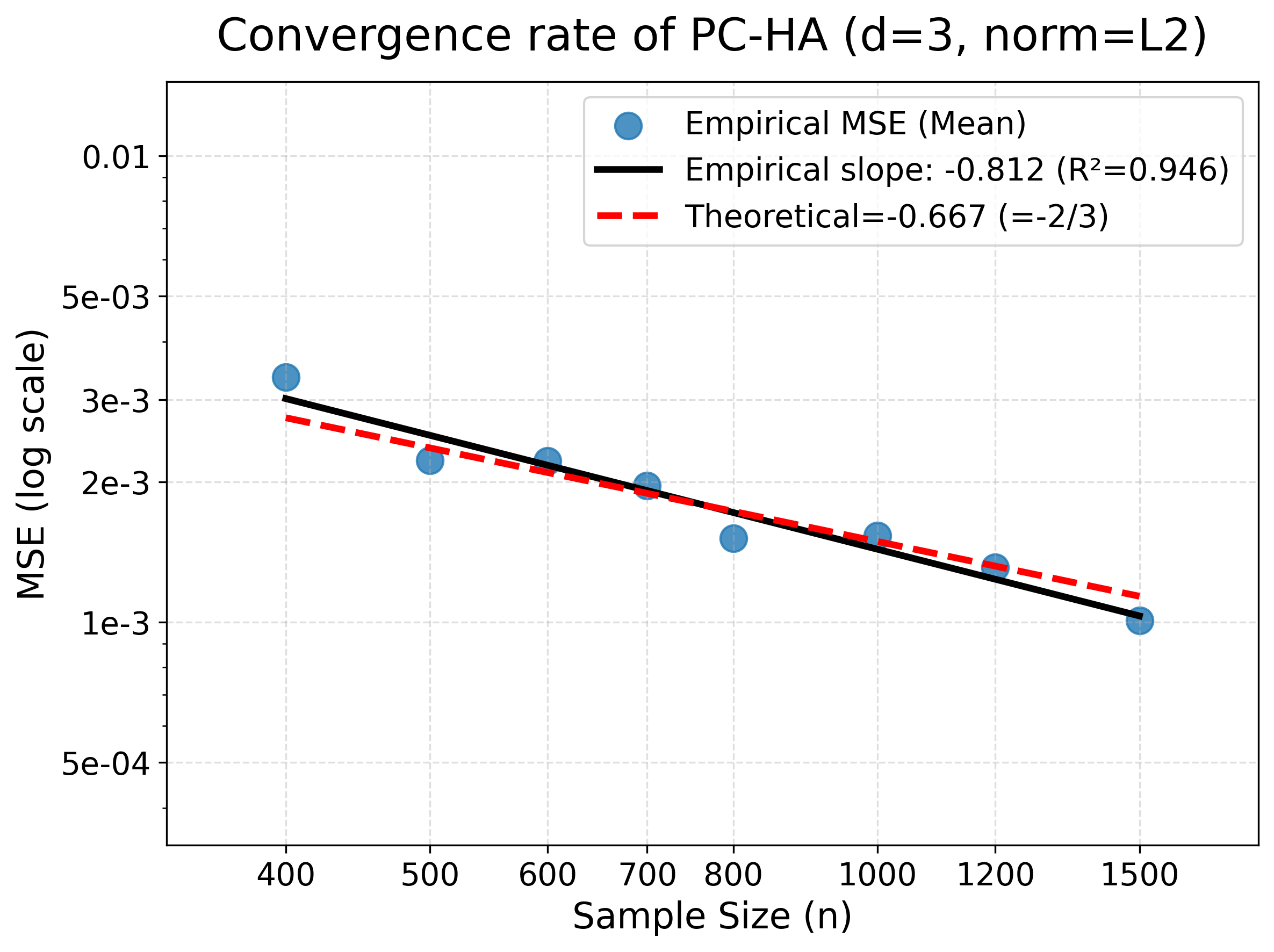}%
			\includegraphics[width=0.3\textwidth]{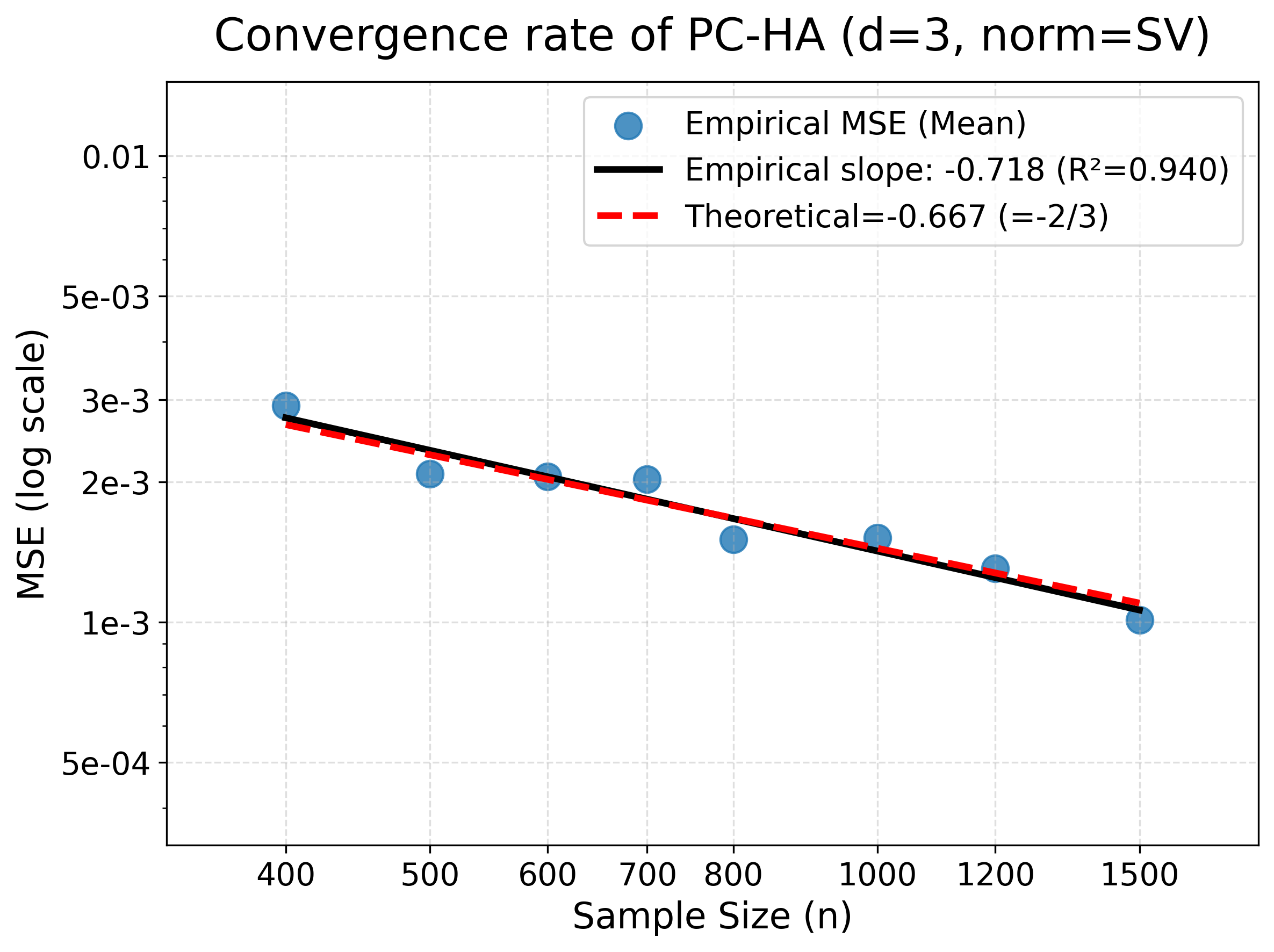}
			
			\vspace{0.2em}
			
			\includegraphics[width=0.3\textwidth]{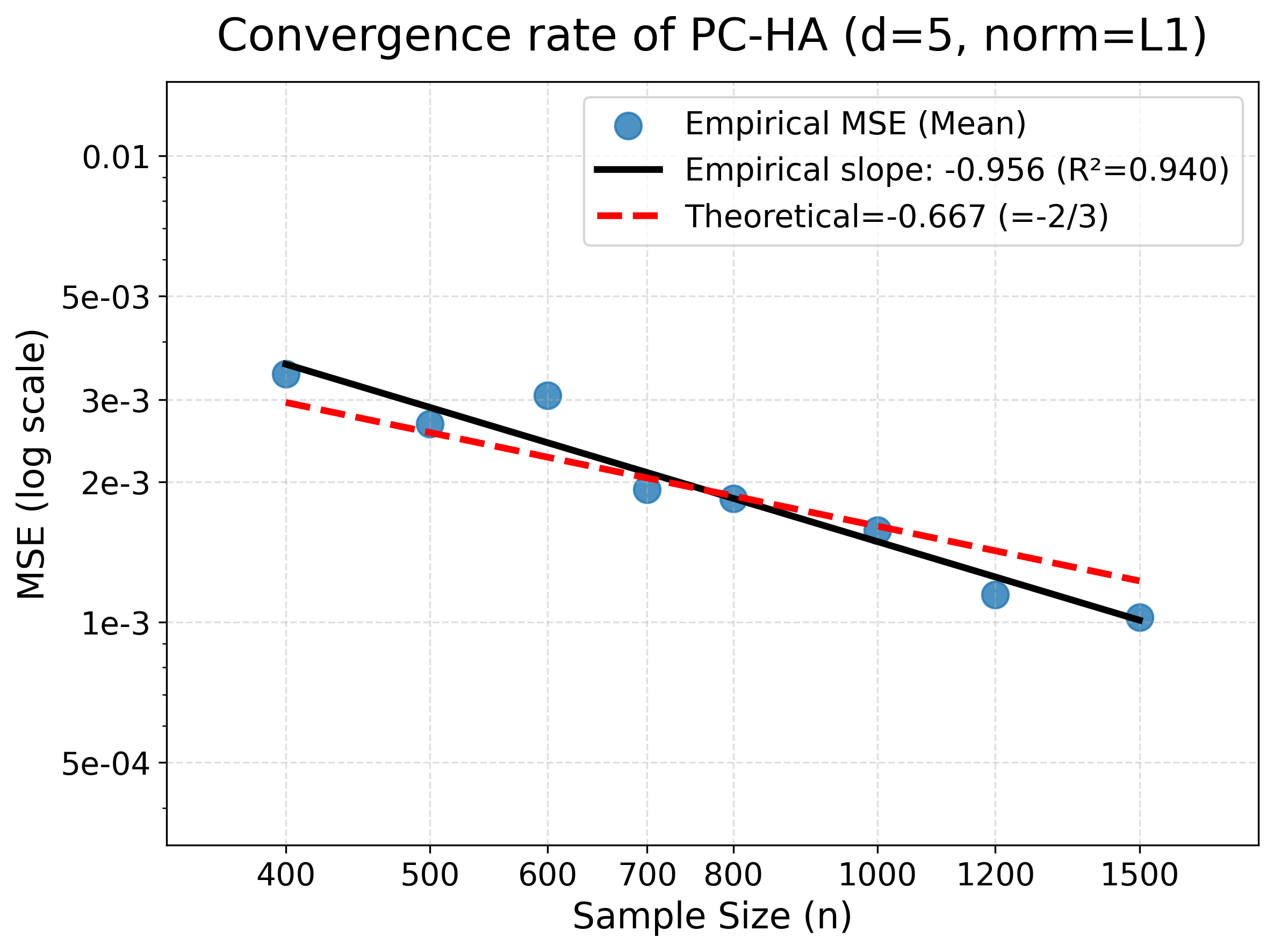}%
			\includegraphics[width=0.3\textwidth]{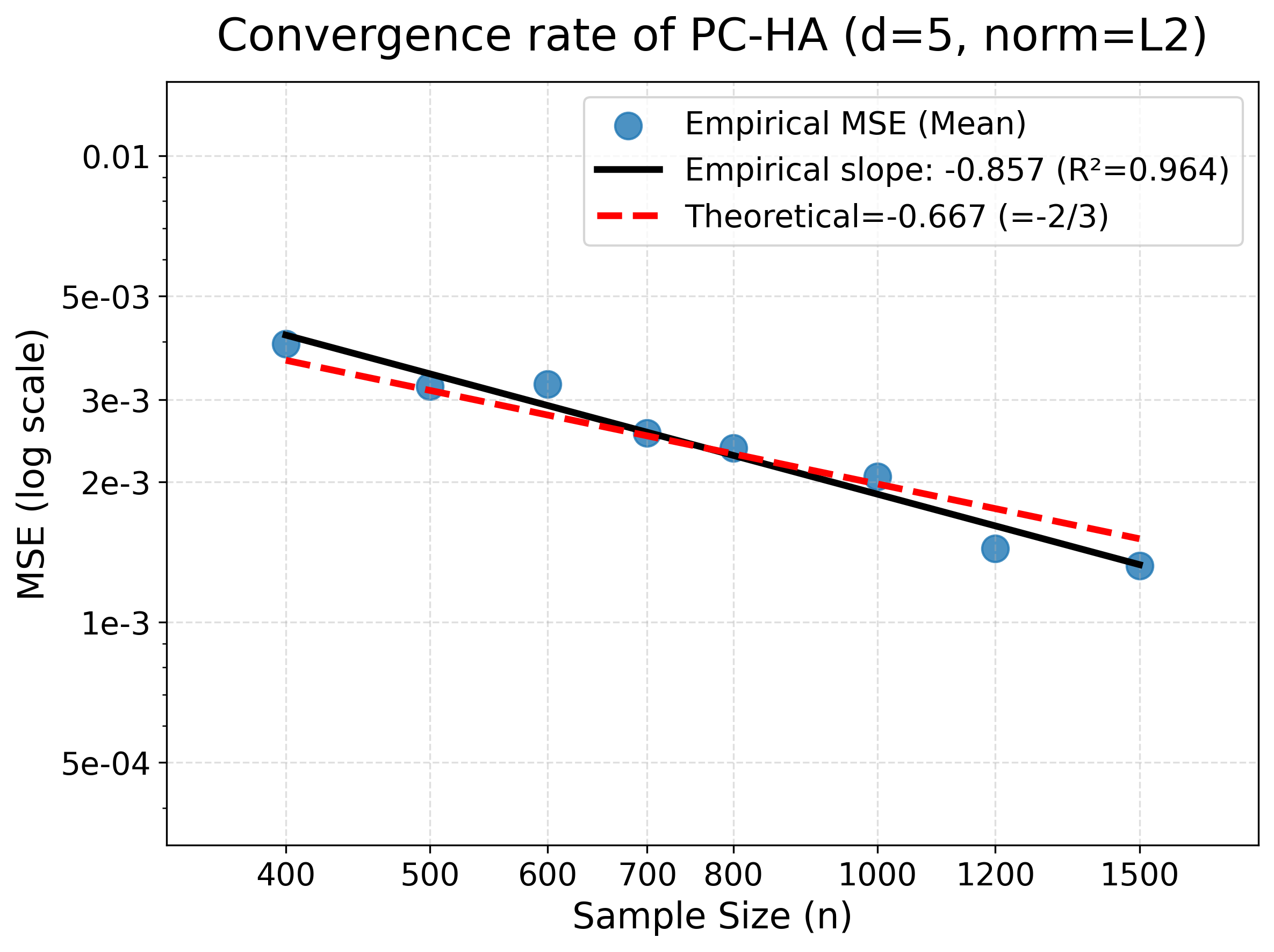}%
			\includegraphics[width=0.3\textwidth]{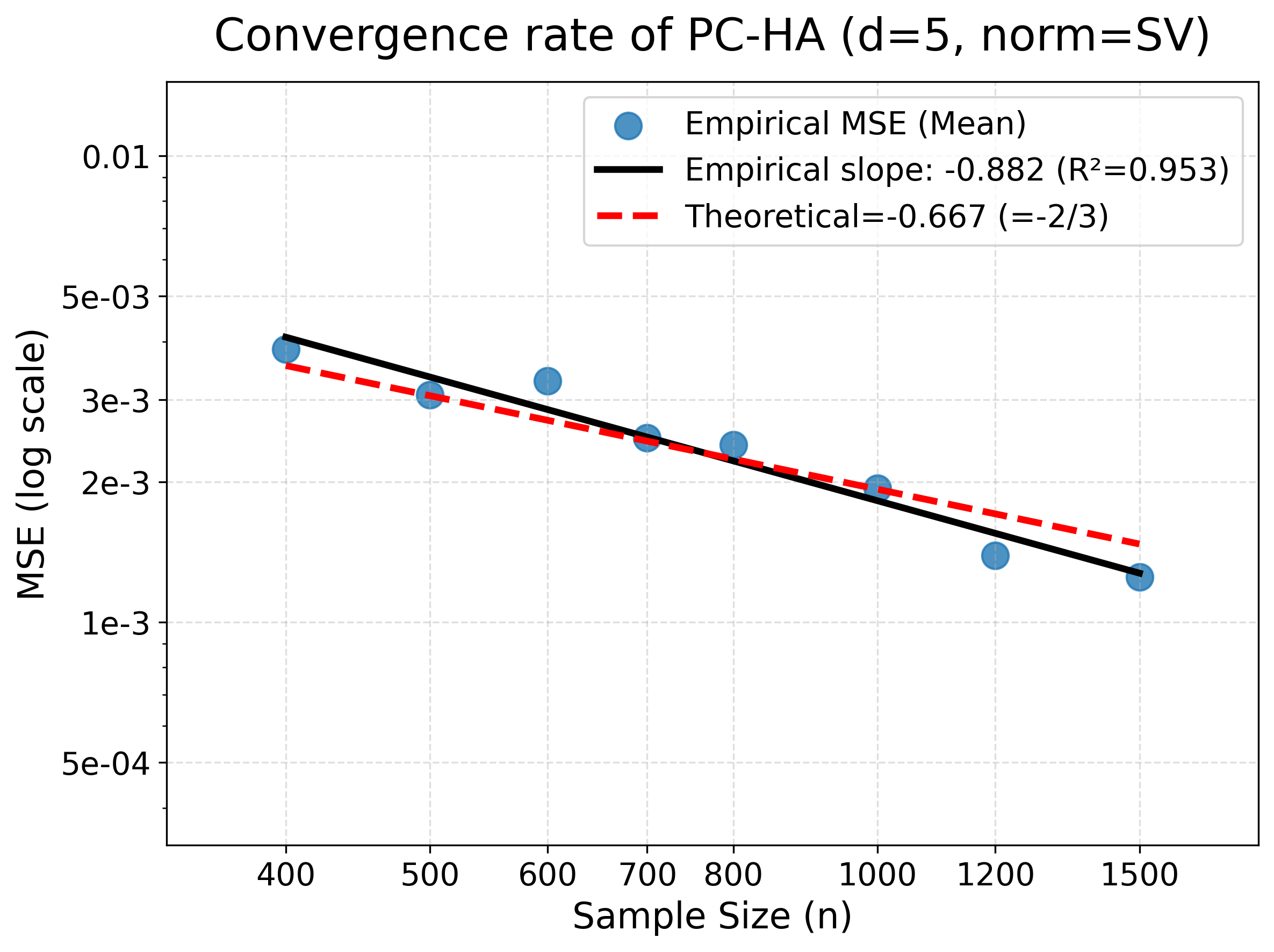}
			
			\vspace{0.2em}
			
			\includegraphics[width=0.3\textwidth]{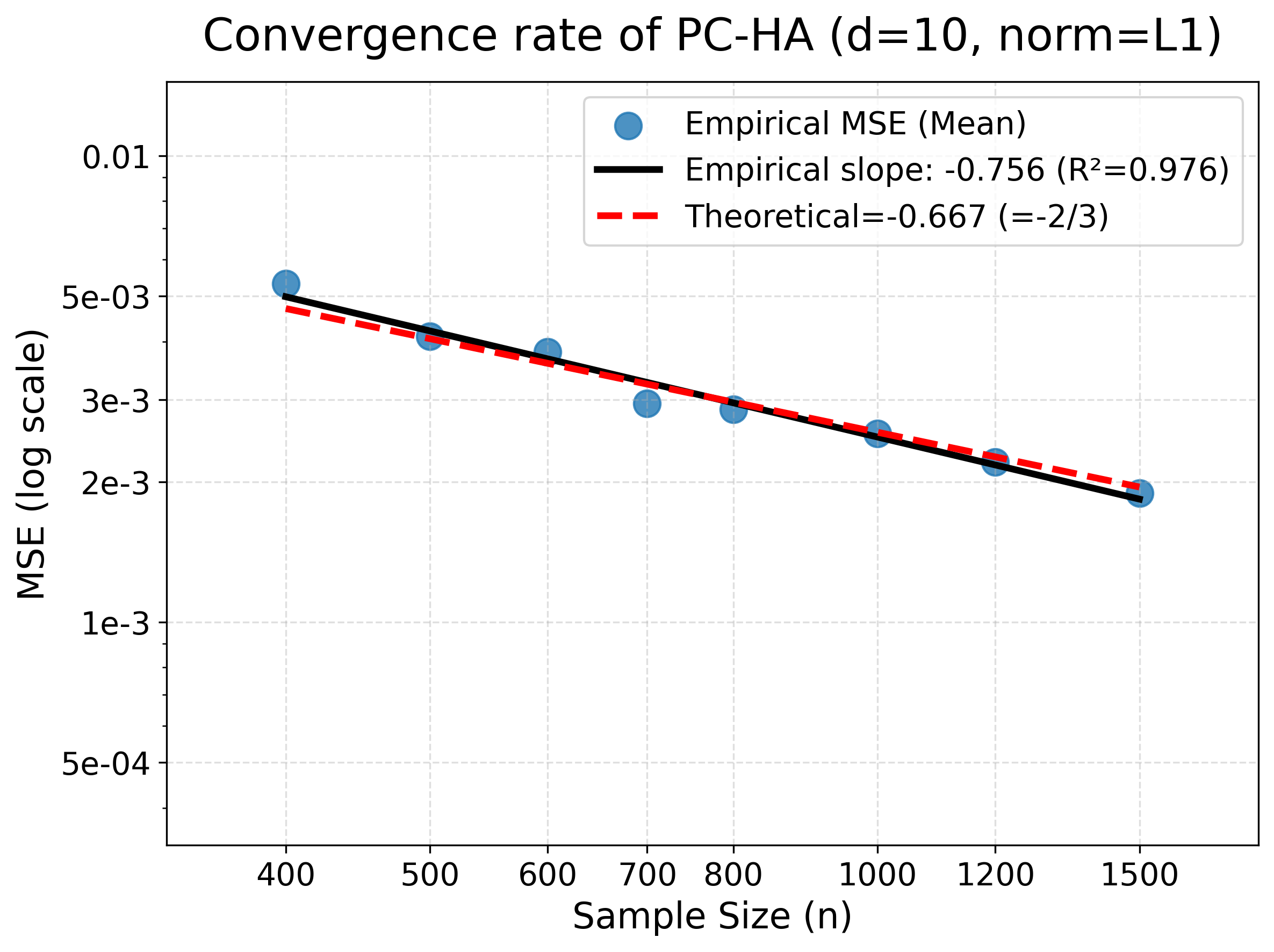}%
			\includegraphics[width=0.3\textwidth]{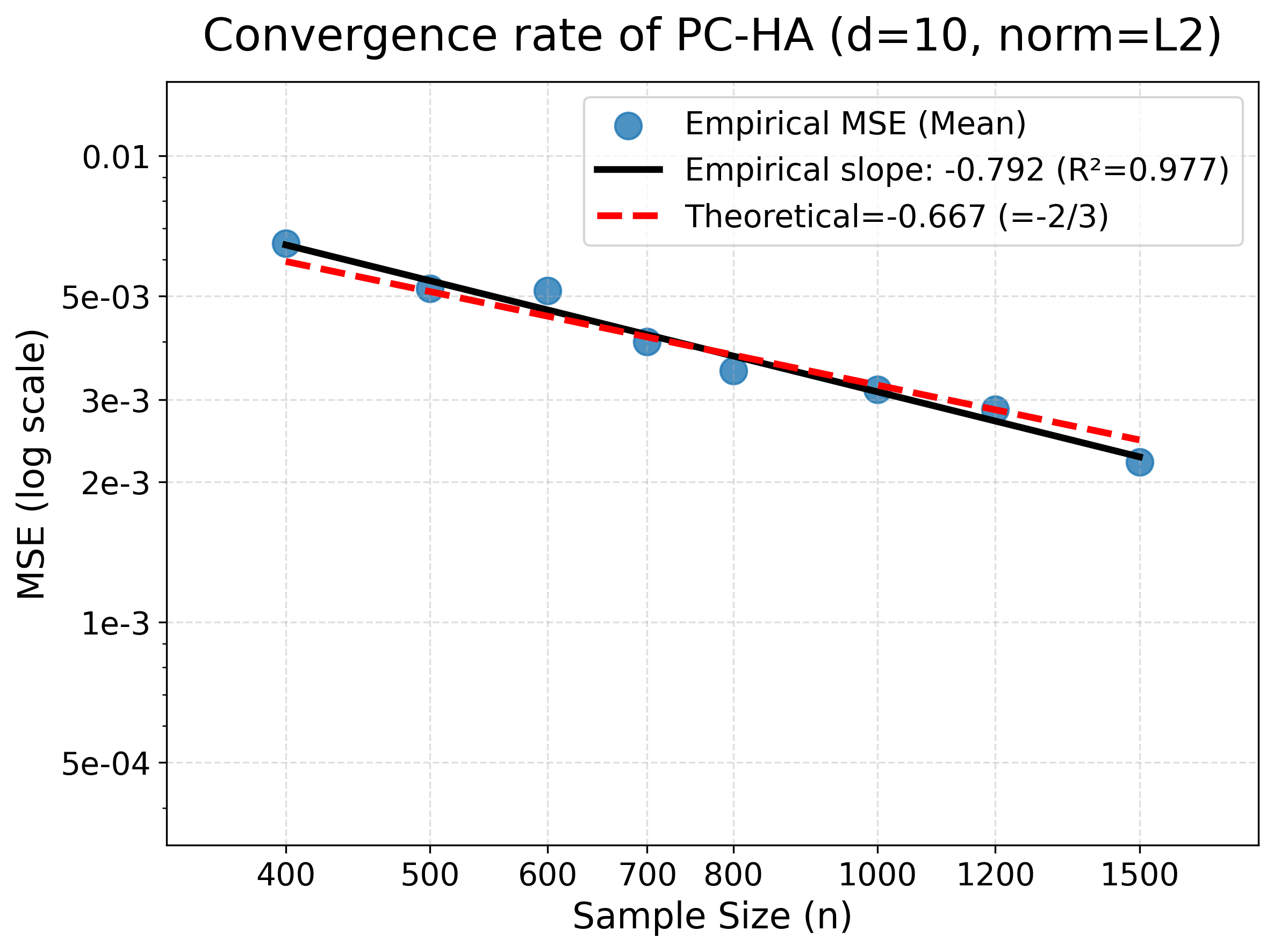}%
			\includegraphics[width=0.3\textwidth]{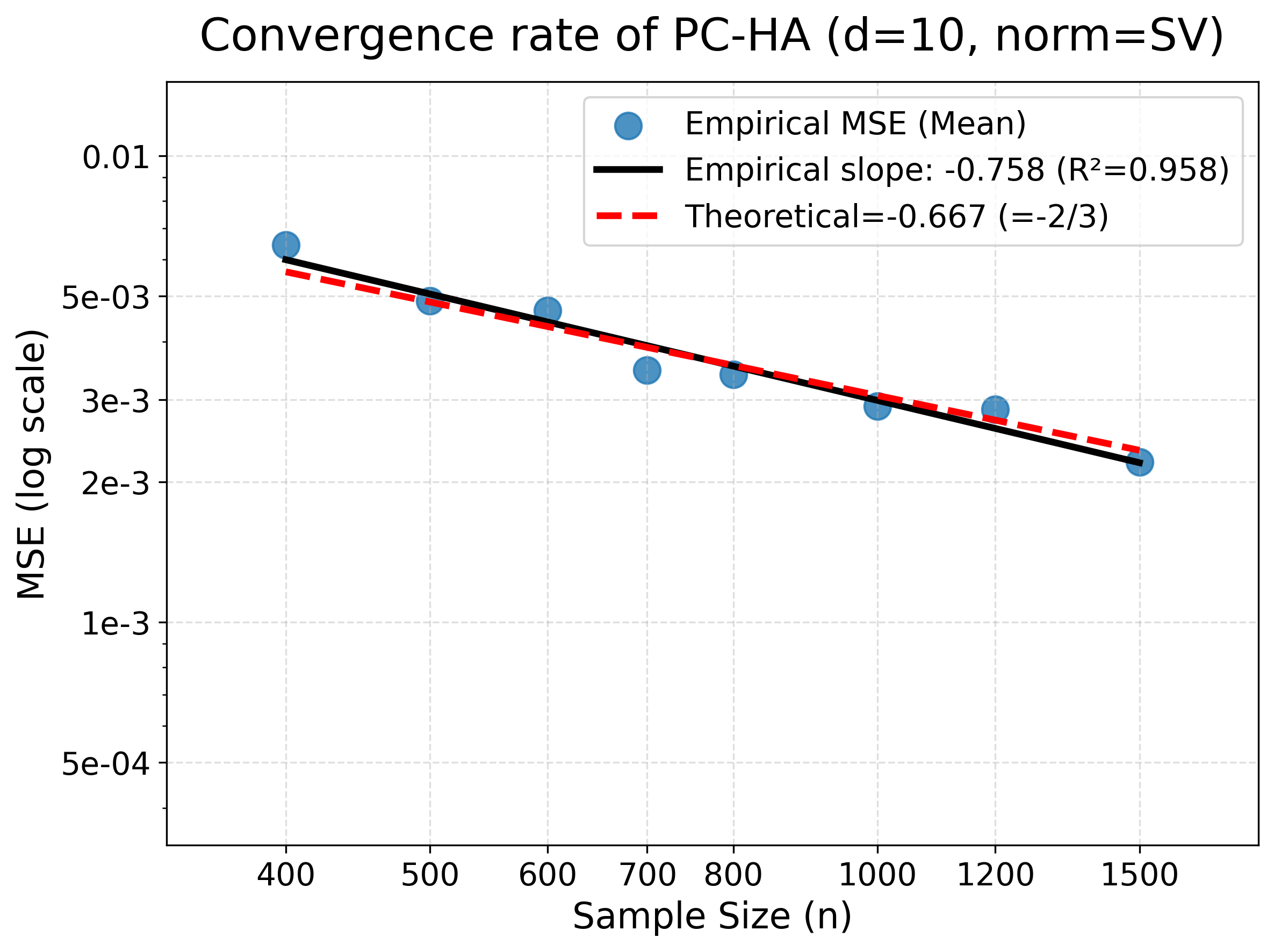}
		\end{minipage}
		
		\caption{Convergence rate of  {cross-validated} PC-HAs using the
			\emph{mean} MSE across repetitions (linear target function). Each panel
			shows log--log MSE versus sample size $n$ for a given dimension $d \in \{3,5,10\}$
			(rows) and norm choice $\{L_1,L_2,\text{sectional variation}\}$ (columns).
			The black line is the empirical convergence slope, while the red dashed line
			is the theoretical reference with exponent $-2/3$.}
		\label{fig:convergence_mean_nonfast}
	\end{figure}
	The similarity in convergence rates across the three regularization schemes is fully consistent with the theoretical guarantees provided in Theorem \ref{theoremone}. This theorem establishes that the optimal minimax rate $O_P(n^{-2/3})$ is achievable for any PC-HA estimator, provided there exists a sequence of radius constraints $r(n)$ such that the resulting parameter space both contains the projection of the true function $\psi_0$ and implies a uniform bound on the sectional variation norm $ {\lVert\beta(\alpha)\rVert_1}$. For the {PC-HAGL} estimator, this condition is satisfied by definition, as it explicitly constrains the sectional variation norm; thus, choosing $r(n)=O(1)$ suffices to capture bounded variation targets while controlling metric entropy. For {PC-HAL} ($L_1$) and {PC-HAR} ($L_2$), while the norms do not measure sectional variation directly, they act as effective surrogate controls. Under standard regularity conditions, the $L_1$ or $L_2$ balls of radius $r(n)$ that contain $\psi_0$ are embedded within a slightly larger ball of bounded sectional variation. Theoretically, this ``mismatch'' in norms may inflate the constant factor of the risk bound (requiring a slightly richer working model to approximate the truth), but it does not alter the fundamental polynomial rate of the bias-variance tradeoff. Furthermore, since the regularization parameter $C_n$ is selected via $3$-fold cross-validation---which satisfies an asymptotic oracle inequality \cite{vdlaan_dudoit_vandervaart}---the estimator adaptively selects the constraint level that minimizes risk. Consequently, as long as the rate-optimal solution lies within the regularization path, the cross-validated estimator will achieve the optimal convergence slope, rendering the specific choice of norm asymptotically equivalent in terms of rate.
	
	{In the Supplement (Figure \ref{fig:convergence_mean_fast}), we include similar convergence plots for the ``fast'' sinusoidal target function. Those plots illustrate that in cases where the target aligns well with the leading PC basis, the empirical rate can become almost parametric (slope close to $-1$) and nearly independent of dimension.}
	
	\subsection{ATE estimation with PC-HA}\label{sec:ate}
	
	We describe the data-generating process and the two estimation strategies compared in the simulation study: a plug-in estimator based on cross-validated PC-HA outcome regression, and a variant that undersmooths the outcome fit so that the bias of the plug-in, as measured by the efficient influence curve, is controlled.

	Data are drawn from a structural model with two covariates, binary treatment, and a continuous outcome. Covariates are $W = (W_1, W_2)$ with $W_1 \sim \text{Uniform}(-2, 2)$ and $W_2 \sim \mathcal{N}(0, 0.25)$. Treatment $A \in \{0,1\}$ is Bernoulli with propensity score $\pi_1(W) = P(A=1 \mid W)$ given by $\operatorname{logit} \pi_1(W) = W_1 + 0.5\,W_2 + W_1 W_2 + 0.3\,W_2^2,$ so that treatment assignment depends on $W$ in a moderately nonlinear way. The outcome is $Y = 2 W_1 - 2 W_2^2 + W_2 + W_1 W_2 + 0.5 + \varepsilon, \varepsilon \sim \mathcal{N}(0, 0.25)$.
	The structural equation for $A$ does not include $A$, so the average treatment effect $E[Y^1] - E[Y^0]$ is zero. This setup lets us assess bias and coverage of the estimators around a known truth.

	We focus on the role of the outcome model by using the \emph{known} propensity score $\pi_1(W)$ throughout. {We will evaluate the empirical mean of the EIC at the true propensity score and therefore how well PC-HA solves the EIC at the true $g_0$. Note if one solves $P_n  {D^*}(Q_n,g_0)$ then we have exact expansion $\Psi(Q_n)-\Psi(Q_0)=(P_n-P_0) {D^*}_{Q_n,g_0}$ with no exact remainder.} The outcome regression $E[Y \mid A=a, W]$ is estimated with PC-HA, $Y$ fitted on $(A, W)$. We compare the three different choices of norm. For a given outcome fit $\widehat{\mu}_1(W)$, $\widehat{\mu}_0(W)$, the plug-in estimator of the ATE is $\widehat{\psi} = P_n \bigl( \widehat{\mu}_1(W) - \widehat{\mu}_0(W) \bigr) = \frac{1}{n} \sum_{i=1}^n \bigl( \widehat{\mu}_1(W_i) - \widehat{\mu}_0(W_i) \bigr),$ where $\widehat{\mu}_a(W)$ is the predicted outcome under $A=a$ from the fitted model.
	
	The efficient influence curve (EIC) is in this case $D(O; \pi, \mu) = \varphi_1(O; \pi, \mu) - \varphi_0(O; \pi, \mu)$, where $O = (W, A, Y)$ and $\varphi_a(O; \pi, \mu) = \frac{\mathbf{1}(A=a)}{\pi_a(W)} \bigl( Y - \mu_a(W) \bigr) + \bigl( \mu_a(W) - E[\mu_a(W)] \bigr), \quad a \in \{0,1\}.$
	Replacing $\mu$ by an estimated outcome $\widehat{\mu}$ and expectations by empirical means, we obtain the estimated EIC $D(O_i; \widehat{\mu})$ at each observation and its empirical mean $P_n D = \frac{1}{n} \sum_{i=1}^n D(O_i; \widehat{\mu})$.
	The quantity $-P_n D$ reflects the bias of the plug-in estimator: when the outcome model is correctly specified, $P_n D$ is close to zero, whereas under regularization (e.g.\ cross-validation) the fit can be undersmoothed and the plug-in biased, with $P_n D$ deviating from zero. Manually controlling the $\ell_1$ norm of the outcome fit is therefore a natural way to guide the amount of  {plug-in} bias.

	We run two sets of experiments. In the first, we fit the outcome regression with PC-HA for each of the three norms and select the regularization parameter $\lambda$ by $K$-fold cross-validation (minimizing mean squared prediction error). We then form the plug-in estimator using this CV-selected fit and report its performance over replications. No undersmoothing is applied; the outcome model is used exactly as chosen by CV.
	
	In the second experiment we again obtain a CV-selected fit and its $\lambda_{\text{CV}}$, but we then \emph{undersmooth} the outcome model by relaxing the penalty: we \emph{decrease} $\lambda$ (i.e.\ increase the allowed $\ell_1$ norm of the solution) until the empirical mean of the EIC is small enough. Specifically, let $\widehat{\mu}^{(\lambda)}$ denote the  {PC-HA} fit at regularization $\lambda$ and $D^{(\lambda)}(O) = D(O; \widehat{\mu}^{(\lambda)})$ the EIC evaluated at that fit. We compute the sample standard deviation of the EIC at the CV fit, $\widehat{\sigma}(D)$, and define the threshold $\tau = \frac{ \widehat{\sigma}(D) }{ \sqrt{n}\, \log n }$.
	We then decrease $\lambda$ along a grid (starting from $\lambda_{\text{CV}}$) until we find $\lambda^*$ such that $\bigl| P_n D^{(\lambda^*)} \bigr| \leq \tau$.
	We take the smallest such $\lambda$ (the most undersmoothed fit that still satisfies the constraint) and use $\widehat{\mu}^{(\lambda^*)}$ as the plug-in. Thus we relax the penalty until the bias $-P_n D$ falls within the band $\pm \tau$, where $\tau$ is scaled by the variability of the EIC at the CV fit and by $\sqrt{n}\log n$. When we perform multiple replications, CV and $\tau$ are computed once on an initial sample; in each replication we then carry out this undersmoothing search and report the plug-in ATE at the selected $\lambda^*$.
	
	In the Supplement (Section~\ref{supp:simulation_plots}, Figure~\ref{fig:ate_eic_lambda}) we include one plot per estimator illustrating the undersmoothing process: the empirical mean of the EIC is plotted against $\lambda$, with the CV choice and the $\tau$ band indicated. Tables 1 and 2 report Monte Carlo results for training sample size $n=300$, 500 dataset replications, true propensity, and 95\% nominal coverage. Left: no undersmoothing (CV-selected outcome fit). Right: with undersmoothing (outcome $\lambda$ decreased until $|P_n D| \leq \tau$).
	Oracle coverage measures how often a nominal 95\% confidence interval would contain the true parameter (0) if we knew the true standard error of the estimator. It is computed as follows. From the Monte Carlo replications we obtain the empirical variance of the estimator (the ``true'' variance) and take its square root to get the true standard error $\sigma_{\text{MC}}$. In each replication $b$, we form the Wald-style interval $\widehat{\psi}_b \pm 1.96\,\sigma_{\text{MC}}$ and check whether the known true parameter lies in that interval. Oracle coverage is the proportion of replications where it does. Because the interval uses the true variance rather than an estimated one, any shortfall from the nominal 95\% is due to bias and/or non-normality of the estimator, not to standard error estimation. Thus oracle coverage summarizes both the bias-to-SE ratio and the quality of the normal approximation.
	
	\begin{table}[htbp]
		\footnotesize
		\centering
		\small
		\begin{minipage}[t]{0.48\textwidth}
			\centering
			\caption{No undersmoothing}
			\label{tab:nounder}
			\begin{tabular}{lccc}
				\toprule
				Metric & PCHAL & PCHAR & PCHAGL  \\
				\midrule
				Bias & 0.0357 & 0.0472 & 0.0345 \\
				
				True SE & 0.0772 & 0.0776 & 0.0779 \\
				Bias / SE & 0.453 & 0.612 & 0.442 \\
				{Oracle coverage} (95\%) & 92.6\% & 90.6\% & 92.4\% \\
				Mean of  {EIC} mean & $-0.0147$ & $-0.0256$ & $-0.0143$ \\
				Mean of  {EIC} std & 1.028 & 0.965 & 1.047 \\
				\bottomrule
			\end{tabular}
		\end{minipage}
		\hfill
		\begin{minipage}[t]{0.48\textwidth}
			\centering
			\caption{With undersmoothing}
			\label{tab:under}
			\begin{tabular}{ccc}
				\toprule
				PCHAL  & PCHAR  & PCHAGL \\
				\midrule
				0.0265 & 0.0411 & 0.0277 \\
				0.0814 & 0.0792 & 0.0807 \\
				0.326 & 0.519 & 0.344 \\
				93.6\% & 91.6\% & 93.6\% \\
				$-0.0137$ & $-0.0240$ & $-0.0134$ \\
				0.835 & 0.934 & 0.863 \\
				\bottomrule
			\end{tabular}
		\end{minipage}
		%\caption{Simulation results: no undersmoothing (left) and with undersmoothing (right). $n=300$, $n_{\text{sim}}=500$, true propensity, 95\% nominal coverage.}
		\label{tab:results}
	\end{table}
	
	{Overall, the three penalty norms yield broadly similar performance, with modest bias and oracle coverage not far from the nominal 95\%. Across both the CV-only and undersmoothing strategies, however, the PCHAL and PCHAGL norms provide the best control of plug-in bias, as reflected in smaller bias-to-SE ratios and oracle coverage closest to 95\%, together with empirical EIC means closest to zero. In contrast, PCHAR remains somewhat more biased and undercovers. These findings indicate that enforcing an $\ell_1$-type constraint on the PC-HA coefficients or on the implied HAL coefficients is most effective at (approximately) solving the EIC score equation and delivering well-calibrated ATE inference in this design.}

	\bibliographystyle{imsart-number} % Style BST file (imsart-number.bst or imsart-nameyear.bst)
	%\bibliography{bib}

	\section{Notation Index}
	
	\renewcommand{\arraystretch}{1.5}
	\begin{longtable}{p{0.25\textwidth} p{0.7\textwidth}}
		\hline
		\textbf{Symbol} & \textbf{Description} \\
		\hline
		$n$ & Sample size (number of observations). \\
		$d$ & Dimension of the multivariate target function / covariates. \\
		$N$ & Number of basis functions in the initial high-dimensional spline model (for $k=0$, $ {N=n(2^d-1)}$; more generally $ {N=O(n\,2^{kd})}$). \\
		$m_n$ & Number of evaluation points in the vector $x^n$ (typically $m_n=n$). \\
		$O$ & Observed data structure, $O \sim P_0$. \\
		$P_0$ & True data distribution. \\
		$P_n$ & Empirical distribution (based on $n$ i.i.d. observations). \\
		$x^n$ & Vector of design points $(x_i : i=1, \ldots, m_n)$. \\
		$k$ & Order of smoothness for the spline basis ($k=0, 1, \ldots$). \\
		$k^*$ & Defined as $k+1$. \\
		
		\hline
		$D^{(k)}([0,1]^d)$ & Space of càdlàg functions with finite $k$-th order sectional variation norm. \\
		$D^{(k)}_M([0,1]^d)$ & Subset of $D^{(k)}([0,1]^d)$ with sectional variation norm bounded by $M$. \\
		$D^{(k)}({\cal R}_N)$ & Finite-dimensional submodel spanned by the initial $N$ spline basis functions (HAL space). \\
		$D^{(k)}({\cal E}_n)$ & $m_n$-dimensional submodel spanned by the PC-basis functions (PC-HA working model). \\
		$D^{(k)}_{M}({\cal R}_N)$ & Subset of the HAL space with coefficient $L_1$-norm bounded by $M$. \\
		$D^{(k)}_{ {\lVert\cdot\rVert}, C}({\cal E}_n)$ & Subset of the PC-HA space where coefficients $\alpha$ satisfy the norm constraint $ {\lVert \alpha \rVert \leq C}$. \\
		$\|\cdot\|_v^*$ & Zero-order sectional variation norm. \\
		$\|\cdot\|_v$ & Variation norm. \\
		
		\hline
		
		$\phi_j$ & The $j$-th spline basis function from the original set ${\cal R}_N$. \\
		$H$ & Design matrix ($m_n \times N$), where $H_{ij} = \phi_j(x_i)$. \\
		$E_m$ & The $m$-th eigenvector of $H^\top H$ (also denoted $E_N(\cdot, m)$). \\
		$E_N^n$ & $N \times m_n$ matrix containing the leading eigenvectors. \\
		$\tilde{\phi}_m$ & The $m$-th PC-basis function (linear combination of $\phi_j$ using $E_m$). \\
		
		\hline
		
		$\beta$ & Coefficient vector of dimension $N$ for the spline basis functions. \\
		$\alpha$ & Coefficient vector of dimension $m_n$ for the PC-basis functions. \\
		$\beta(\alpha)$ & Mapping from PC coefficients $\alpha$ to spline coefficients $\beta$. \\
		$\psi_{N,\beta}$ & Function in the HAL model parameterized by $\beta$: $\sum \beta(j)\phi_j$. \\
		$\theta_{n,\alpha}$ & Function in the PC-HA model parameterized by $\alpha$: $\sum \alpha(m)\tilde{\phi}_m$. \\
		$ {\lVert \alpha \rVert_1}$ & Standard $L_1$-norm of $\alpha$ (used for PC-HAL). \\
		$ {\lVert \alpha \rVert_2}$ & Standard $L_2$-norm of $\alpha$ (used for PC-HAR). \\
		$ {\lVert \alpha \rVert_3}$ & Defined as $ {\lVert \beta(\alpha)\rVert_1}$, the implied sectional variation norm. \\
		
		\hline
		
		$\psi_0$ & $\arg\min P_0 L(\psi)$. \\
		$\psi_{N,0}$ & Best approximation of $\psi_0$ in the finite spline model (Oracle). \\
		$\psi_{N,\beta_n}$ & Standard HAL-MLE over $D^{(k)}_M({\cal R}_N)$. \\
		$\theta_{n,\alpha_n^*}$ & PC-representation of the HAL-MLE such that $P_n L(\theta_{n,\alpha_n^*}) = P_n L(\psi_{N,\beta_n})$. \\
		$\hat{\Psi}_{\text{PC-HAGL}}$ & PC-Highly Adaptive Generalized Lasso estimator ($\theta_{n, \alpha_{GL}}$). \\
		$\hat{\Psi}_{\text{PC-HAL}}$ & PC-Highly Adaptive Lasso estimator ($\theta_{n, \alpha_{L}}$). \\
		$\hat{\Psi}_{\text{PC-HAR}}$ & PC-Highly Adaptive Ridge estimator ($\theta_{n, \alpha_{R}}$). \\
		$\alpha_{n,cv}$ & Coefficient vector selected via Cross-Validation. \\
		$C_n$ & Regularization parameter selected via Cross-Validation. \\
		$r(n)$ & Rate function for the bound on $\alpha$. \\
		
		\hline
		$L(\psi)$ & True loss function. \\
		$L_n(\psi)$ & Approximation of the loss function. \\
		$d_0(\psi, \psi_0)$ & Loss-based dissimilarity $P_0 L(\psi) - P_0 L(\psi_0)$. \\
		$d_{0n}(\psi, \psi_{0n})$ & Approximate loss-based dissimilarity. \\
		\hline
	\end{longtable}

	\begin{appendix}
		
		\newpage
		
		\section{Proof of Pointwise Asymptotic Normality}\label{proof:asymnor}
		
		Theorem 2 in { \cite{van2023higher} shows that $J$-dimensional $k$-th order spline working models $D^{(k)}({\cal R}(d,{\bf J}))$ provide supremum-norm approximations of the smoothness class $D^{(k)}_M([0,1]^d)$ at rate $O\{r(d,J)^{k+1}\}\approx O(1/J^{k+1})$ (up to powers of $\log J$). Thus, for each $n$ there exists a $J_n$-dimensional sparse working model $D^{(k)}({\cal R}_{0,n})$ that approximates the true target function at the rate $O(1/J_n^{k+1})$, even though this model is not constructed by the estimator and is only known to exist.}
		{Because this oracle model $D^{(k)}({\cal R}_{0,n})$ is fixed and does not depend on the data, all sums of empirical process terms in the proof run over a non-random index set ${\cal R}_{0,n}$. We then orthonormalize the basis functions in $D^{(k)}({\cal R}_{0,n})$ with respect to an inner product based on the covariance (or information) of the scores, so that the score vector has an identity covariance (or information) matrix in this basis. After this orthonormalization, the central objects in the analysis are simply the empirical means $P_n S_{\theta_n}(\phi_j^*)$ of the scores evaluated at the orthonormal basis functions $\phi_j^*$, and we can track their contribution to $(\theta_n-\psi_{0,n})(x_0)$ through linear algebra.}

		The PC-HA estimator will never end up with a sparse working model in the original basis functions $\{\phi_j:1 \leq j \leq J_n\}$.
		We know that the working model $D^{(k)}({\cal E}_n)$ provides a close enough approximation of $D^{(k)}([0,1]^d)$ or smaller order than the rate of convergence. However, this is still an $n$-dimensional working model. 
		However, this says nothing about the existence of a  $J_n$ dimensional working model $D^{(k)}({\cal R}_n)$ of $D^{(k)}({\cal R}_N)$ with $J_n \sim^+ n^{1/(2k^*+1)}$ that has an approximation error of $D^{(k)}([0,1]^d)$ that is of same order as rate $n^{-k^*/(2k^*+1)}$.
		%Therefore, it might be reasonable to assume that there exists a $J_n$ dimensional working model $D^{(k)}({\cal E}(J_n))$
		%of $D^{(k)}({\cal E}_n)$ with $J_n\sim^+ n^{1/(2k^*+1)}$ that has an approximation error of $D^{(k)}({\cal E}_n)$ and thereby of
		%$D^{(k)}_M([0,1]^d)$%. This working model should then  be found with standard Lasso with $D^{(k)}({\cal E}_n)$ as in PC-HAL, and it might closely correspond with selecting the top $J_n$ eigenvectors w.r.t. eigenvalues. If true, then this would provide a sound strategy for establishing the asymptotic normality of $\theta_n$, analogue to our proof in next subsection with $D^{(k)}({\cal R}_{0,n})$ replaced by $D^{(k)}({\cal E}(J_n))$. 
		
		\noindent\fbox{\parbox{\dimexpr\linewidth-2\fboxsep-2\fboxrule}{
				{\bf Definition of ${\cal R}_{0,n}$ (sparse working model).}
				{We define ${\cal R}_{0,n}$ to be a set of indices of size $J_n$ such that the working model $D^{(k)}({\cal R}_{0,n})\subset D^{(k)}({\cal R}_N)$ is the sparsest (or any) model satisfying: (i)~{Approximation:} $\sup_{f \in D^{(k)}_M({\cal R}_N)} \inf_{\psi \in D^{(k)}_M({\cal R}_{0,n})}  {\lVert f - \psi \rVert_\infty} = O_P(n^{-k^*/(2k^*+1)})$; (ii)~{Oracle convergence:} the oracle MLE $\psi_{0,n}=\arg\min_{\psi\in D^{(k)}({\cal R}_{0,n})}P_0 L(\psi)$ satisfies $ {\lVert \psi_{0,n}-\psi_0\rVert_{\infty}}=o_P((J_n/n)^{1/2})$. {We choose $J_n \asymp n^{1/(2k^*+1)}$ (up to powers of $\log n$) so that the approximation bias satisfies
						$ {\lVert} \psi_{0,n}-\psi_0 {\rVert_\infty} = o_P\!\left(n^{-k^*/(2k^*+1)}\right)$; since the sharp rate is
						$n^{-k^*/(2k^*+1)}$ up to $\log n$ factors, suppressing these logs only strengthens the statement and keeps the bias of smaller order than the leading stochastic term.}
		}}}
		
		Existence of such an ${\cal R}_{0,n}$ is guaranteed by Theorem 2 in \cite{van2023higher}: $J$-dimensional $k$-th order spline working models approximate $D^{(k)}_M([0,1]^d)$ in supremum norm at rate $O(1/J^{k+1})$, so we may take ${\cal R}_{0,n}$ to be any such set (the definition is non-constructive).

		In virtue of the arguments above, let's aim an analysis not relying on $\theta_n$ being an MLE for a sparse working model $D^{(k)}({\cal E}(J_n))$.
		Instead let's rely on the independent sparse working model $D^{(k)}({\cal R}_{0,n})\subset D^{(k)}({\cal R}_N)$ of size $ {J_n\sim^+ n^{1/(2k^*+1)}}$. 
		We have a PC-HAGL $\theta_n=\theta_{n,\alpha_n}$ with $ {\lVert \beta(\alpha_n)\rVert_1}=C_n$ and $C_n=O_P(1)$.
		Moreover $ {\lVert \alpha_n\rVert_2^2=\lVert \beta(\alpha_n)\rVert_2^2}=O_P(1/n)$.
		By our result on score equations solved by PC-HAGL 
		we know that $\theta_n$ solves score equations $\sup_{f^*\in D^{(k)}({\cal E}_n), {\lVert f^*\rVert_2}\leq n^{-1/2}}P_n S_{\theta_n}(f^*)=o_P(n^{-1/2})$, and we can establish more precise rates. 
		We also know that $\sup_{f\in D^{(k)}_M({\cal R}_N)} \inf_{\theta\in D^{(k)}_M({\cal E}_n)} {\lVert f-\theta\rVert_{\infty}}=O_P(n^{-k^*/(2k^*+1)})$. If we do not screen out PC basis functions, then  $\alpha_n$ will not have zero components, which makes this score equation result better by not having to replace  $D^{(k)}({\cal E}_n)$ by a submodel $D^{(k)}({\cal E}(J_n))$. Due to $D^{(k)}({\cal R}_N)$ fast approximating $D^{(k)}([0,1]^d)$, the above bound also implies  $\sup_{f^*\in D^{(k)}_M([0,1]^d)}P_n S_{\theta_n}(f^*)=o_P(n^{-1/2})$.
		
		Let $\psi_{0,n}=\arg\min_{\psi\in D^{(k)}({\cal R}_{0,n})}P_0L(\psi)$ be the oracle MLE for the sparse working model
		$D^{(k)}({\cal R}_{0,n})\subset D^{(k)}({\cal R}_N)$.
		We wish to analyze $(\theta_n-\psi_{0,n})(x_0)$ at a point $x_0\in (0,1)^d$.
		Given that $ {\lVert} \psi_{0,n}-\psi_0 {\rVert_{\infty}}=O_P(n^{-k^*/(2k^*+1)})$, this will then also yield a linear approximation of 
		$(\theta_n-\psi_0)(x)$. 
		
		The analysis is based on treating $\theta_n$ as an approximate MLE of $\psi_{0,n}$ over  $D^{(k)}({\cal R}_{0,n})$. 
		Therefore we are concerned with score equations $P_n S_{\theta_n}(\phi_j)\approx 0$, $j\in {\cal R}_{0,n}$. Let $\phi^n=(\phi_j: j\in {\cal R}_{0,n})$.
		We know that $P_0 S_{\psi_{0,n}}(\phi^n)=0$. 
		We also have that $P_n S_{\theta_n}(\phi^n)=o_P(n^{-1/2})$ by the score equation result for $\theta_n$. 
		We have that $\theta_n\not \in D^{(k)}({\cal R}_{0,n})$ so that we will also use a projection $\tilde{\theta}_n\in D^{(k)}({\cal R}_{0,n})$ of $\theta_n$ onto $D^{(k)}({\cal R}_{0,n})$, where $ {\lVert}\tilde{\theta}_n-\theta_n {\rVert_{\infty}}=O_P(n^{-k^*/(2k^*+1)})$.
		We then somehow need that the $\tilde{\theta}_n$ still solves the scores $\theta_n$ solves: $P_n S_{\tilde{\theta}_n}(\phi^n)=o_P(n^{-1/2})$. Therefore we propose to define $\tilde{\theta}_n$ as the solution of 
		$P_n S_{\tilde{\theta}_n}(\phi^n)-S_{\theta_n}(\phi^n)=0$ so that the score equations are fully preserved by $\tilde{\theta}_n$.
		This generally corresponds with defining $\tilde{\theta}_n$ as a loss based projection of $\theta_n$ onto $D^{(k)}({\cal R}_{0,n})$.

		{\bf Inner product for orthogonalizing $D^{(k)}({\cal R}_{0,n})$:}
		Finally, we will replace $\{\phi_j:1 \leq j \leq J_n\}$ by an orthonormal $\{\phi^*_j:1 \leq j \leq J_n\}$ w.r.t. an appropriate inner product
		$\langle f_1,f_2\rangle_{sc,P_0}$. This does not affect the linear span but just the representation. 
		We have two candidate choices for this inner product.
		The first choice is $\langle f_1,f_2\rangle_{sc,P_0}=P_0 S_{\psi_{0,n}}(f_1)S_{\psi_{0,n}}(f_2)$ and the second possible choice is
		$\langle f_1,f_2\rangle_{sc,P_0}\equiv P_0 \frac{d}{d\delta_0}S_{\psi_{0,n}+\delta_0 f_2}(f_1)$. For likelihood behaving loss functions these two inner product are close approximations of each other, due to the information matrix being equal to the covariance of scores for a correct parametric model, and our $D^{(k)}({\cal E}_n)$ is approximately correct. 
		
		{The first inner product prioritizes the orthogonality of the scores in the linear approximation,  while the latter prioritizes getting a diagonal information matrix in the Taylor expansion analysis. The proof applies to both choices. 
			We will select the second one since it allows us to obtain a clean linear approximation, and we then have to deal with the scores not being fully orthogonal as we will do below. }
		
		Let $\tilde{\beta}_n$ be such that $\tilde{\theta}_n=\sum_{j\in {\cal R}_{0,n}}\tilde{\beta}_n(j)\phi_j^*$
		and $\beta_{0,n}$ so that $\psi_{0,n}=\sum_{j\in {\cal R}_{0,n}}\beta_{0,n}(j)\phi_j^*$.
		We will  {interchangeably} use $S_{\tilde{\theta}_n}(\phi^{*,n})$ and $S_{\tilde{\beta}_n}(\phi^{*,n})$ and $S_{\psi_{0,n}}(\phi^{*,n})$ and $S_{\beta_{0,n}}(\phi^{*,n})$.
		
		{\[
			\fbox{$r_n(\phi^{*,n})\equiv P_n S_{\theta_n}(\phi^{*,n})$.}
			\]}
		{\bf Starting equation:}
		We have the typical identity for $M$-estimator analysis:
		\[
		-P_0\{S_{\theta_n}(\phi^{*,n})-S_{\psi_{0,n}}(\phi^{*,n})\}=(P_n-P_0) S_{\theta_n}(\phi^{*,n})-r_n(\phi^{*,n}).\]
		We write $P_0 S_{\theta_n}(\phi^{*,n})=P_0 S_{\tilde{\theta}_n}(\phi^{*,n})+P_0 (S_{\theta_n}-S_{\tilde{\theta}_n})(\phi^{*,n})$.
		Then,
		\begin{eqnarray*}
			-P_0\{S_{\tilde{\theta}_n}(\phi^{*,n})-S_{\psi_{0,n}}(\phi^{*,n})\}
			&=&(P_n-P_0) S_{\theta_n}(\phi^{*,n})-r_n(\phi^{*,n})\\
			&&+P_0(S_{\theta_n}-S_{\tilde{\theta}_n})(\phi^{*,n}).
		\end{eqnarray*}
		{\bf First order Taylor expansion of Score Equation in $\beta_{0,n}$:}
		Let
		\[
		{\fbox{$R_{1n}(\phi^{*,n})\equiv -P_0\left\{S_{\tilde{\theta}_n}(\phi^{*,n})-S_{\psi_{0,n}}(\phi^{*,n})-d/d\beta_{0,n}S_{\beta_{0,n}}(\phi^{*,n})(\tilde{\beta}_n-\beta_{0,n})\right\}$}}.\]
		This is the exact remainder of a first order Taylor expansion of $\beta\rightarrow -P_0S_{\beta}(\phi^{*,n})$ at $\beta=\beta_{0,n}$
		evaluated at $\tilde{\beta}_n$.
		By definition of our inner product, we have that this derivative (the information matrix)  is the identity matrix, which is an $J_n\times J_n$-matrix.
		So we have
		\begin{eqnarray*}
			\tilde{\beta}_n-\beta_{0,n}&=&(P_n-P_0)S_{\theta_n}(\phi^{*,n})-r_n(\phi^{*,n})-R_{1n}(\phi^{*,n})\\
			&&+P_0(S_{\theta_n}-S_{\tilde{\theta}_n})(\phi^{*,n}).
		\end{eqnarray*}
		{\bf Equation in terms of $\tilde{\theta}_n-\psi_{0,n}$:}
		Now, on both sides we take the standard inner product with $\phi^{*,n}$ and evaluate it at $x_0$. 
		This yields
		\begin{eqnarray*}
			\sum_{j\in {\cal R}_{0,n}}(\tilde{\beta}_n-\beta_{0,n})(j)\phi_j^*(x_0)&=&(P_n-P_0)\sum_{j\in {\cal R}_{0,n}}S_{\theta_n}(\phi^*_j)\phi_j^*(x_0)-\sum_{j\in {\cal R}_{0,n}}r_n(\phi^*_j)\phi_j^*(x_0)\\
			&&-\sum_{j\in {\cal R}_{0,n}}R_{1n}(\phi^*_j)\phi_j^*(x_0)
			+\sum_{j\in {\cal R}_{0,n}}P_0(S_{\theta_n}-S_{\tilde{\theta}_n})(\phi^*_j)\phi_j^*(x_0).
		\end{eqnarray*}
		The left-hand side equals $(\tilde{\theta}_n-\psi_{0,n})(x_0)$. 
		{Notice that in the term $\sum_{j\in {\cal R}_{0,n}}r_n(\phi^*_j)\phi_j^*(x_0)$ the scalars are the point evaluations $\phi_j^*(x_0)$ and the functions are the score equations $r_n(\phi^*_j)$. In contrast, in the next definition the functions are the basis elements $\phi_j^*$ themselves, and the coefficients are again the scalars $\phi_j^*(x_0)$. Thanks to the existence result for the fixed sparse set ${\cal R}_{0,n}$, all these sums are taken over a non-random index set, which allows us to control these empirical process terms. }
		Let {\[
			\fbox{$\bar{\phi}_{n,x_0}^*\equiv J_n^{-1/2}\sum_{j\in {\cal R}_{0,n}}\phi_j^*(x_0)\phi_j^*$.}
			\]} a linear combination of the basis functions $\phi_j^*$ with coefficients equal to $ {J_n^{-1/2}\phi_j^*(x_0)}$. Thus $\bar{\phi}_{n,x_0}^*\in D^{(k)}({\cal R}^*_{0,n})$.
		Due to linearity of the objects in $\phi_j^*$, this can also be  more succinctly represented as:
		\begin{eqnarray*}
			(n/J_n)^{1/2}(\tilde{\theta}_n-\psi_{0,n})(x_0)&=&n^{1/2}(P_n-P_0)S_{\theta_n}(\bar{\phi}_{n,x_0}^*)-n^{1/2}r_n(\bar{\phi}_{n,x_0}^*) -n^{1/2}R_{1n}(\bar{\phi}_{n,x_0}^*) \\
			&&+n^{1/2}P_0 \{S_{\theta_n}(\bar{\phi}^*_{n,x_0})-S_{\tilde{\theta}_n}(\bar{\phi}^*_{n,x_0})\} .
		\end{eqnarray*}
		{\bf Empirical process remainder:}
		Let
		\[
		{\fbox{$R_{2n}(\bar{\phi}_{n,x_0}^*)\equiv n^{1/2}(P_n-P_0)\{S_{\theta_n}(\bar{\phi}_{n,x_0}^*)-S_{\psi_{0,n}}(\bar{\phi}_{n,x_0}^*)\}$}}.\]
		Then,
		\begin{eqnarray*}
			(n/J_n)^{1/2}(\tilde{\theta}_n-\psi_{0,n})(x_0)&=&n^{1/2}(P_n-P_0)S_{\psi_{0,n}}(\bar{\phi}_{n,x_0}^*)-n^{1/2}r_n(\bar{\phi}_{n,x_0}^*) -n^{1/2}R_{1n}(\bar{\phi}_{n,x_0}^*) \\
			&&+n^{1/2}R_{2n}(\bar{\phi}_{n,x_0}^*)+n^{1/2}P_0 \{S_{\theta_n}(\bar{\phi}^*_{n,x_0})-S_{\tilde{\theta}_n}(\bar{\phi}^*_{n,x_0}) \}.
		\end{eqnarray*}
		{\bf Dealing with the score difference at $\theta_n$ and $\tilde{\theta}_n$:}
		Consider the term
		\[
		\begin{array}{l}
			R_{3n}(\bar{\phi}_{n,x_0}^*)\equiv n^{1/2}P_0 \{S_{\theta_n}(\bar{\phi}^*_{n,x_0})-S_{\tilde{\theta}_n}(\bar{\phi}^*_{n,x_0}) \}.\\
			%=P_0 \{(S_{\theta_n}-S_{\psi_{0,n}})(\bar{\phi}^*_{n,x_0})-(S_{\tilde{\theta}_n}-S_{\psi_{0,n}})(\bar{\phi}^*_{n,x_0}) \}
		\end{array}
		\]
		For example, for least squares regression this term equals $P_0\bar{\phi}_{n,x_0}^*(\theta_n-\tilde{\theta}_n)$.
		For example, if we define $\tilde{\theta}_n$ as the loss-based projection of $\theta_n$ onto $D^{(k)}({\cal R}_{0,n}^*)$ with squared error loss, then
		$P_0\phi_j^*(\tilde{\theta}_n-\theta_n)=0$ for all $j\in {\cal R}_{0,n}^*$ and thus
		$P_0\bar{\phi}_{n,x_0}^*(\theta_n-\tilde{\theta}_n)=0$.
		In general, we could define the projection $\tilde{\theta}_n$ as the solution of
		\[
		{P_0\{S_{\theta_n}(\phi_j^*)-S_{\tilde{\theta}_n}(\phi_j^*)\}=0\mbox{ for all } j\in {\cal R}_{0,n}.}
		\]
		Let's assume that this is how we define $\tilde{\theta}_n$. Then, we have $R_{3n}(\bar{\phi}_{n,x_0}^*)=0$.
		Then,
		\begin{eqnarray*}
			(n/J_n)^{1/2}(\tilde{\theta}_n-\psi_{0,n})(x_0)&=&n^{1/2}(P_n-P_0)S_{\psi_{0,n}}(\bar{\phi}_{n,x_0}^*)-n^{1/2}r_n(\bar{\phi}_{n,x_0}^*) -n^{1/2}R_{1n}(\bar{\phi}_{n,x_0}^*) \\
			&&+n^{1/2}R_{2n}(\bar{\phi}_{n,x_0}^*).
		\end{eqnarray*}
		{\bf Substituting $\theta_n$ back into the left-hand side to obtain expression for $\theta_n-\psi_{0,n}$:}
		Finally, we want a linear approximation for our actual estimator $\theta_n$ instead of its projection onto $D^{(k)}({\cal R}_{0,n}^*)$. 
		Let $b_n(x_0)\equiv \tilde{\theta}_n(x_0)-\theta_n(x_0)$. We have 
		\begin{eqnarray*}
			(n/J_n)^{1/2}(\theta_n-\psi_{0,n})(x_0)&=&n^{1/2}(P_n-P_0)S_{\psi_{0,n}}(\bar{\phi}_{n,x_0}^*)-r_n(x_0)-(n/J_n)^{1/2}b_n(x_0),\end{eqnarray*}
		where
		\[
		{\fbox{$r_n(x_0)\equiv n^{1/2}r_n(\bar{\phi}_{n,x_0}^*) -n^{1/2}R_{1n}(\bar{\phi}_{n,x_0}^*) 
				+n^{1/2}R_{2n}(\bar{\phi}_{n,x_0}^*)$}}
		\]
		{\bf Dealing with bias $\tilde{\theta}_n(x_0)-\theta_n(x_0)$:}
		By assumption our working model $D^{(k)}({\cal R}_{0,n})$ was chosen so that its bias w.r.t. $\theta_n$ or any other function in $D^{(k)}({\cal R}_N)$ is $O_P(n^{-k^*/(2k^*+1)})$. Then, for a $log n$ optimal choice $J_n\sim^+ n^{1/(2k^*+1)}$ we obtain $b_n(x_0)(n/J_n)^{1/2}=o_P(1)$ and in fact, $\sup_{x_0}b_n(x_0)(n/J_n)^{1/2}=o_P(1)$.
		\nl

		{\bf Dealing with remainder term $r_n(x_0)$:}
		We will also  assume $r_n(x_0)=o_P(1)$.
		For this, we will need that
		\begin{eqnarray}
			n^{1/2}r_n(\bar{\phi}_{n,x_0}^*) &=&o_P(1)\\
			n^{1/2}R_{1n}(\bar{\phi}_{n,x_0}^*)&=&o_P(1)\\
			n^{1/2}R_{2n}(\bar{\phi}_{n,x_0}^*)&=&o_P(1).
		\end{eqnarray}
		{We summarise these three conditions in the requirement below, which is referenced in the main text as ``B. Regularity condition'', needed for asymptotic normality Theorem 5.
			\[r_n(x_0)=o_P(1)\]} 
		
		The second and third term are clearly second order differences making it plausible. The first term has to be derived from having bounds on the score equations solved by $\theta_n$ we established earlier. Precise methods for bounding these remainders are carried out in \cite{van2023higher} and will not be repeated here.
		We just suffice with stating that these assumptions are reasonable and expected to hold in great generality.
		We then have
		\begin{eqnarray*}
			(n/J_n)^{1/2}(\theta_n-\psi_{0,n})(x_0)&=&n^{1/2}(P_n-P_0)S_{\psi_{0,n}}(\bar{\phi}_{n,x_0}^*)-r_n(x_0)-(n/J_n)^{1/2}b_n(x_0),\end{eqnarray*}
		where $r_n(x_0)=o_P(1)$ and $(n/J_n)^{1/2}b_n(x_0)=o_P(1)$.
		%We assume that indeed the size $J_n$ is chosen so that $(n/J_n)^{1/2}b_n(x_0)=o_P(1)$.
		\nl
		{\bf Understanding the leading empirical mean of independent random variables:}
		We now note that the leading term is a sum of independent mean zero random variables $S_{\psi_{0,n}}(\bar{\phi}_{n,x_0}^*)$.
		Let
		\[
		\sigma^2_{0,n}(x_0)\equiv P_0\{S_{\psi_{0,n}}(\bar{\phi}_{n,x_0}^*)\}^2.\]
		Suppose that we would have selected the inner product $\langle f_1,f_2\rangle_{sc,P_0}=P_0 S_{\psi_{0,n}}(f_1)S_{\psi_{0,n}}(f_2)$ so that the covariance of the score vector is the identity matrix. Then, we would have
		\begin{eqnarray*}
			\sigma^2_{0,n}(x_0)&=&J_n^{-1}\sum_{j_1,j_2\in {\cal R}_{0,n}}P_0S_{\psi_{0,n}}(\phi_{j_1}^*)S_{\psi_{0,n}}(\phi_{j_2}^*)\phi_{j_1}^*(x_0)\phi_{j_2}^*(x_0)\\
			&=& J_n^{-1}\sum_{j\in {\cal R}_{0,n}}P_0\{S_{\psi_{0,n}}(\phi_j^*)\}^2\{\phi_j^*(x_0)\}^2\\
			&=& J_n^{-1} \sum_{j\in {\cal R}_{0,n}}\{\phi_j^*(x_0)\}^2.
		\end{eqnarray*}
		In general, for log-likelihood behaving loss functions the covariance score inner product is approximately equal to the derivative inner product we used.
		In general, one can still show that the expression for $\sigma^2_0(x_0)$ is bounded.
		We have
		\[
		\sigma^2_{0,n}(x_0)= J_n^{-1}\phi^{*,n}(x_0)^{\top}\Sigma_{0,n}\phi^{*,n}(x_0),\]
		where
		\[
		\Sigma_{0,n}\equiv P_0 S_{\psi_{0,n}}(\phi^{*,n}) S^{\top}_{\psi_{0,n}}(\phi^{*,n})\]
		is the $J_n\times J_n$ covariance matrix of $S_{\psi_{0,n}}(\phi^{*,n})$. (Using $\bar{\phi}_{n,x_0}^*= J_n^{-1/2}\sum_j \phi_j^*(x_0)\phi_j^*$, one has $\Var_0(S_{\psi_{0,n}}(\bar{\phi}_{n,x_0}^*))= J_n^{-1}\phi^{*,n}(x_0)^{\top}\Sigma_{0,n}\phi^{*,n}(x_0)$.)
		Let $\Sigma_{0,n}=S_n D_n S_n^t$ be the eigenvalue decomposition with diagonal eigenvalue matrix $D_n$. 
		Then, $\sigma^2_{0,n}(x_0)= J_n^{-1} {\lVert D_n^{1/2} S_n\phi^{*,n}(x_0)\rVert_2^2}$ is $J_n^{-1}$ times the square of the Euclidean norm of $D_n^{1/2} S_n\phi^{*,n}(x_0)$. The operator norm of $D_n^{1/2}$ is given by  {$\lambda_n^{1/2}$}, the square root of the maximal eigenvalue of $\Sigma_{0,n}$. The Euclidean norm of $S_n\phi^{*,n}(x_0)$ equals the norm of $\phi^{*,n}(x_0)$.
		Therefore, 
		\[
		\sigma^2_{0,n}(x_0)\leq J_n^{-1}\lambda_n \sum_{j\in {\cal R}_{0,n}}\{\phi_j^*(x_0)\}^2 = \lambda_n \frac{1}{J_n}\sum_{j\in {\cal R}_{0,n}}\{\phi_j^*(x_0)\}^2.
		\]
		We can assume that $\lambda_n=O_P(1)$ which then implies that $\sigma^2_{0,n}(x_0)=O_P(1)$. Either way, we have
		\[
		(n/J_n)^{1/2}(\theta_n-\psi_{0,n})(x_0)/\sigma_{0,n}(x_0)\Rightarrow_d N(0,1).\]
		If $\sigma^2_{0,n}(x_0)=O_P(1)$ as expected then this corresponds with a rate of convergence given by $(J_n/n)^{1/2}$.

		{\bf Pointwise confidence intervals for $\psi_{0,n}(x_0)$:}
		This yields pointwise confidence intervals $\theta_n(x_0)\pm 1.96 \sigma_n(x_0)(J_n/n)^{1/2}$ for $\psi_{0,n}(x_0)$, where
		$\sigma_n^2(x_0)$ is a consistent estimator of $\sigma^2_{0,n}(x_0)$. 
		
		{\bf Variance estimation:}
		The construction of a variance estimator $\sigma^2_n(x_0)$ is not as straightforward as with HAL. When we use regular HAL we obtain a sparse working model $D^{(k)}({\cal R}_n)$ that can be viewed as an approximation of the fixed working model $D^{(k)}({\cal R}_{0,n})$. This allows one then to use a delta-method working model based confidence interval as if it was a fixed parametric working model. For the PC-HA estimators we do not have a sparse working model $D^{(k)}({\cal R}_{n,0})$ and we cannot view a delta-method variance estimator based on a  potentially sparse working model $D^{(k)}({\cal E}(J_n))$ as an approximation of this variance $\sigma^2_{0,n}(x_0)$ corresponding with $D^{(k)}({\cal R}_{n,0}^*)$. Moreover, we want to allow that PC-HA does not use any screening of the PC- basis functions (for example, PC-HAR applied to $D^{(k)}({\cal E}_n)$)  as a first step. Therefore, the above theoretical derivation does not provide a constructive variance estimator. 
		% This might be an argument for going for PC-HA estimators applied to a screened working model$D^{(k)}({\cal E}(J_n))$ so that one can use the $\delta$-method variance estimator for the selected working model. 
		Therefore, we recommend using the nonparametric bootstrap. For this one can fix the working model $D^{(k)}({\cal E}_n)$ and only bootstrap the PC-HAL starting from this working model construction. This should still make it relatively fast to compute.
		
		\subsection{Simultaneous confidence bands}
		Above we obtained the exact identity:
		\begin{eqnarray*}
			(n/J_n)^{1/2}(\theta_n-\psi_{0,n})(x_0)&=&n^{1/2}(P_n-P_0)S_{\psi_{0,n}}(\bar{\phi}_{n,x_0}^*)-r_n(x_0)-(n/J_n)^{1/2}b_n(x_0),\end{eqnarray*}
		Suppose that $\sup_{x_0}\mid r_n(x_0)\mid =o_P(1)$ and we already assumed $\sup_{x_0}(n/J_n)^{1/2}b_n(x_0)=o_P(1)$. We could assume this uniformity in a subset ${\cal A}\subset [0,1]^d$ instead of uniformly over $[0,1]^d$.
		Then, we have
		\[
		(n/J_n)^{1/2}(\theta_n-\psi_{0,n})(x_0)=n^{1/2}(P_n-P_0)S_{\psi_{0,n}}(\bar{\phi}_{n,x_0}^*)+o_P(1),\]
		where the remainder is $o_P(1)$ uniformly in $x_0$.
		
		{\bf Supremum norm rate of convergence and simultaneous confidence band:}
		This can now be used to obtain a rate of convergence in supremum norm and a simultaneous confidence band. 
		To see this, consider the class of functions ${\cal F}_n=\{J_n^{-1/2} S_{\psi_{0,n}}(\bar{\phi}_{n,x_0}^*)/\sigma_{0,n}(x_0): x_0\in {\cal A}\}$.
		This is a nice $d$-dimensional class of uniformly bounded functions. Therefore its covering number satisfies $N(\epsilon,{\cal F}_n, {\lVert\cdot\rVert_{\infty}})\sim \epsilon^{-d}$. Moreover, $\sup_{x_0\in {\cal A}} {\lVert} S_{\psi_{0,n}}(\bar{\phi}_{n,x_0}^*)/J_n^{1/2} {\rVert_{P_0}}=
		O_P(J_n^{-1/2})$, so that $\sup_{f\in {\cal F}_n} {\lVert} f {\rVert_{P_0}}=O(J_n^{-1/2})$. Therefore we can bound the expectation of 
		$\sup_{f\in {\cal F}_n}\mid G_n(f)\mid$ by $O(J_{\infty}(J_n^{-1/2},{\cal F}_n))$, where $J(\delta,{\cal F})$ denotes the entropy integral 
		$\int_0^{\delta} \sqrt{\log N(\epsilon,{\cal F}_n, {\lVert\cdot\rVert_{\infty}})}d\epsilon$, which behaves as $\delta$ up till $\log \delta$ factor. 
		Therefore, $\sup_{x_0\in {\cal A}}(n/J_n)^{1/2}(\theta_n-\psi_{0,n})(x_0)/\sigma_{0,n}(x_0)=o_P(\log n)$.
		This proves that
		\[
		{\lVert} \theta_n-\psi_{0,n} {\rVert_{\infty}}=O_P^+((J_n/n)^{1/2}), \]
		establishing converges in supremum norm at the rate $O_P^+(n^{-k^*/(2k^*+1)})$. 
		Moreover, it shows that $\theta_n(x_0)\pm \log n 1.96 \sigma_n(x_0)(J_n/n)^{1/2}$ represents a simultaneous $0.95$ confidence band, possibly asymptotically conservative. One could create a sharper confidence band by basing it on a multivariate normal approximation for a grid of $x_0(j)$-values estimating the correlation  matrix of $S_{\psi_{0,n}}(\bar{\phi}_{n,x_0}^*)$.

		Now we present an extended version of Theorem 6. 
		Let $\theta_n=\theta_{n,\alpha_n}$ be a PC-HA estimator 
		defined by $\theta_n=\arg\min_{\theta\in  {D^{(k)}_{\lVert\cdot\rVert,C_n}}({\cal E}_n)}P_n L(\theta)$, constraining one of the three norms (i.e.,  {$\lVert\alpha\rVert_2$},  {$\lVert \alpha\rVert_1$},  {$\lVert\beta(\alpha)\rVert_2$}) by a cross-validated bound $C_n$.
		Recall $S_{\theta}(f)=\frac{d}{d\theta}L(\theta)(f)$ as the score operator.
		There exists a  $D^{(k)}({\cal R}_{0,n})\subset D^{(k)}({\cal R}_N)$  of size $ {J_n\sim^+ n^{1/(2k^*+1)}}$ that strongly approximates $D^{(k)}_M([0,1]^d)$ in the sense that
		$\sup_{f\in D^{(k)}_M([0,1]^d)}\inf_{\psi\in D^{(k)}_M({\cal R}_{0,n})} {\lVert} f-\psi {\rVert_{\infty}}=O_P(n^{-k^*/(2k^*+1)})$.
		
		%By our result on score equations solved by PC-HAGL we know that $\theta_n$ solves score equations $\sup_{f^*\in D^{(k)}({\cal E}_n),\pl \beta(f^*)\pl_2\leq n^{-1/2}}P_n S_{\theta_n}(f^*)=o_P(n^{-1/2})$.
		%We also know that $\sup_{f\in D^{(k)}_M({\cal R}_N)} \inf_{\theta\in D^{(k)}_M({\cal E}_n)}\pl f-\theta\pl_{\infty}=O_P(n^{-k^*/(2k^*+1)})$. 
		
		Let $\psi_{0,n}=\arg\min_{\psi\in D^{(k)}({\cal R}_{0,n})}P_0L(\psi)$ be the oracle MLE for this sparse working model
		$D^{(k)}({\cal R}_{0,n})\subset D^{(k)}({\cal R}_N)$.
		We have that $ {\lVert} \psi_{0,n}-\psi_0 {\rVert_{\infty}}=O_P(n^{-k^*/(2k^*+1)})$. Let $\tilde{\theta}_n\in D^{(k)}({\cal R}_{0,n})$ be a projection of $\theta_n$ onto this sparse working model satisfying $P_n \{S_{\tilde{\theta}_n}(\phi^n)-S_{\theta_n}(\phi^n)\}=0$ so that the score equations are fully preserved by $\tilde{\theta}_n$.
		
		We replace $\{\phi_j:1 \leq j \leq J_n\}$ by an orthonormal $\{\phi^*_j:1 \leq j \leq J_n\}$ w.r.t. inner product
		$\langle f_1,f_2\rangle_{sc,P_0}\equiv P_0 \frac{d}{d\delta_0}S_{\psi_{0,n}+\delta_0 f_2}(f_1)$.
		Let $\tilde{\beta}_n$ be such that $\tilde{\theta}_n=\sum_{j\in {\cal R}_{0,n}}\tilde{\beta}_n(j)\phi_j^*$
		and $\beta_{0,n}$ so that $\psi_{0,n}=\sum_{j\in {\cal R}_{0,n}}\beta_{0,n}(j)\phi_j^*$.
		We will interchangeably use $S_{\tilde{\theta}_n}(\phi^{*,n})$ and $S_{\tilde{\beta}_n}(\phi^{*,n})$ and $S_{\psi_{0,n}}(\phi^{*,n})$ and $S_{\beta_{0,n}}(\phi^{*,n})$.
		We define the following quantities: 
		\[
		\begin{array}{l}
			r_n(\phi^{*,n})\equiv P_n S_{\theta_n}(\phi^{*,n})\\
			R_{1n}(\phi^{*,n})\equiv -P_0\left\{S_{\tilde{\theta}_n}(\phi^{*,n})-S_{\psi_{0,n}}(\phi^{*,n})-d/d\beta_{0,n}S_{\beta_{0,n}}(\phi^{*,n})(\tilde{\beta}_n-\beta_{0,n})\right\}\\
			\bar{\phi}_{n,x_0}^*\equiv J_n^{-1/2}\sum_{j\in {\cal R}_{0,n}}\phi_j^*(x_0)\phi_j^*\\
			R_{2n}(\bar{\phi}_{n,x_0}^*)\equiv n^{1/2}(P_n-P_0)\{S_{\theta_n}(\bar{\phi}_{n,x_0}^*)-S_{\psi_{0,n}}(\bar{\phi}_{n,x_0}^*)\}\\
			R_{3n}(\bar{\phi}_{n,x_0}^*)\equiv n^{1/2}P_0 \{S_{\theta_n}(\bar{\phi}^*_{n,x_0})-S_{\tilde{\theta}_n}(\bar{\phi}^*_{n,x_0}) \}\\
			b_n(x_0)\equiv \tilde{\theta}_n(x_0)-\theta_n(x_0)\\
			r_n(x_0)\equiv n^{1/2}r_n(\bar{\phi}_{n,x_0}^*) -n^{1/2}R_{1n}(\bar{\phi}_{n,x_0}^*) 
			+n^{1/2}R_{2n}(\bar{\phi}_{n,x_0}^*)\\
			\Sigma_{0,n}=P_0 S_{\psi_{0,n}}(\phi^{*,n})S^{\top}_{\psi_{0,n}}(\phi^{*,n})\\
			\sigma^2_{0,n}(x_0)=\phi^{*,n}(x_0)^{\top}\Sigma_{0,n}\phi^{*,n}(x_0).
		\end{array}
		\]
		
		\begin{lemma}
			Assume  {$\lVert \alpha_n\rVert_2^2=\lVert \beta(\alpha_n)\rVert_2^2$}$=O_P(1/n)$;
			$J_n\sim^+  n^{1/(2k^*+1)}$; maximal eigenvalue $\lambda_n $ of $\Sigma_{0,n}$ satisfies $\lambda_n=O_P(1)$;
			$n^{1/2}r_n(\bar{\phi}_{n,x_0}^*) =o_P(1)$;
			$n^{1/2}R_{1n}(\bar{\phi}_{n,x_0}^*)=o_P(1)$;
			$n^{1/2}R_{2n}(\bar{\phi}_{n,x_0}^*)=o_P(1)$.
			
			Then 
			\begin{eqnarray*}
				(n/J_n)^{1/2}(\theta_n-\psi_{0,n})(x_0)&=&n^{1/2}(P_n-P_0)S_{\psi_{0,n}}(\bar{\phi}_{n,x_0}^*)-r_n(x_0)-(n/J_n)^{1/2}b_n(x_0),\end{eqnarray*}
			where $r_n(x_0)=o_P(1)$ and $(n/J_n)^{1/2}b_n(x_0)=o_P(1)$.
		\end{lemma}
		%We assume that indeed the size $J_n$ is chosen so that $(n/J_n)^{1/2}b_n(x_0)=o_P(1)$.
		The variance of the empirical mean satisfies
		\[
		\sigma^2_{0,n}(x_0)\leq \frac{\lambda_n}{J_n}\sum_{j\in {\cal R}_{0,n}}\{\phi_j^*(x_0)\}^2
		=O_P(1).\]
		In particular, 
		\[
		(n/J_n)^{1/2}(\theta_n-\psi_{0,n})(x_0)/\sigma_{0,n}(x_0)\Rightarrow_d N(0,1).\]
		Since $ {\lVert} \psi_{0,n}-\psi_0 {\rVert_{\infty}}=o_P((J_n/n)^{1/2})$ this implies
		\[
		(n/J_n)^{1/2}(\theta_n-\psi_0)(x_0)/\sigma_{0,n}(x_0)\Rightarrow_d N(0,1).\]
		
		{\bf Pointwise confidence intervals for $\psi_{0,n}(x_0)$:}
		$\theta_n(x_0)\pm 1.96 \sigma_n(x_0)(J_n/n)^{1/2}$ is a consistent $0.95$-confidence interval for $\psi_{0,n}(x_0)$ (and also for $\psi_0(x_0)$)  if 
		$\sigma_n^2(x_0)$ is a consistent estimator of $\sigma^2_{0,n}(x_0)$.

		{\bf Sup-norm rate of convergence:}
		Suppose that $\sup_{x_0\in {\cal A}}\mid r_n(x_0)\mid =o_P(1)$ for a subset ${\cal A}\subset [0,1]^d$.
		%and we already assumed $\sup_{x_0}(n/J_n)^{1/2}b_n(x_0)=o_P(1)$.
		Then, we have
		\[
		(n/J_n)^{1/2}(\theta_n-\psi_{0,n})(x_0)=n^{1/2}(P_n-P_0)S_{\psi_{0,n}}(\bar{\phi}_{n,x_0}^*)+o_P(1),\]
		where the remainder is $o_P(1)$ uniformly in $x_0$.
		Moreover,  $\sup_{x_0\in {\cal A}}(n/J_n)^{1/2}(\theta_n-\psi_{0,n})(x_0)/\sigma_{0,n}(x_0)=o_P(\log n)$, so that 
		\[
		{\lVert} \theta_n-\psi_{0,n} {\rVert_{\infty}}=O_P^+((J_n/n)^{1/2}). \]
		
		{\bf Simultaneous confidence band:}
		$\theta_n(x_0)\pm \log n 1.96 \sigma_n(x_0)(J_n/n)^{1/2}$ represents a simultaneous asymptotic $0.95$ confidence band, possibly asymptotically conservative.
		
		\subsection{Understanding the transformation to inner product that orthogonalizes the scores}
		One can define an orthonormal basis $\phi^*_m$ of $\tilde{\phi}_m$ w.r.t. the inner product the covariance of
		$S_{Q_0}(f_1)$ and $S_{Q_0}(f_2)$. This makes the scores $S_{Q_0}(\phi^*_m)$ orthogonal.
		Let $\Sigma$ be the $J_n\times J_n$ covariance matrix of $S_{Q_0}(\tilde{\phi}_m)$, $m\in {\cal E}(J_n)$. 
		One might define 
		\[
		\lambda_n=\arg\max_{ {\lVert h\rVert_2}=1} h^{\top} \Sigma h\]
		as the maximal eigenvalue of the matrix $\Sigma$.
		One can bound the variance of $J_n^{-1/2} \sum_m S_{Q_0}(\tilde{\phi}_m)$ by the variance of $J_n^{-1/2}\sum_m S_{Q_0}(\phi^*_m)$ 
		times $\lambda_n$. 
		Given that $\tilde{\phi}_m$ is an orthogonal  basis w.r.t. an inner product $\langle\cdot,\cdot\rangle_n$, we expect that $\lambda_n$ will grow slowly with $J_n$.
		In particular, when we select $\langle f_1,f_2\rangle_n=P_n S_{\theta_{n,\alpha_n}}(f_1)S_{\theta_{n,\alpha_n}}(f_2)$ as the empirical covariance of the scores at a good estimator $\theta_{n,\alpha_n}$ of $\psi_0$, so that the matrix $A$ becomes close to the identity matrix, we should see a slow convergence of $\lambda_n$.
		
		\section{Proof of Theorem \ref{theoremone}}
		
		Let's consider a general norm PC-HA estimator that computes an MLE over $ {D^{(k)}_{\lVert\cdot\rVert,C}}({\cal E}_n)$ for all $C$ and selects $C$ with the cross-validation selector. 
		So
		\[
		\alpha_{n,C}=\arg\min_{\alpha, {\lVert \alpha\rVert}<C}P_n L(\theta_{n,\alpha}).\]
		Then, $\theta_{n,\alpha_{n,C}}$ is the minimum empirical risk estimator over $ {D^{(k)}_{\lVert\cdot\rVert,C}}({\cal E}_n)$. 
		Let a rate $r(n)$ be defined so that 
		\[
		\sup_{\alpha, {\lVert\alpha\rVert}<r(n)} {\lVert} \beta(\alpha) {\rVert_1}=O(1).\]
		Here $r(n)$ represents the size of constraint on $ {\lVert\alpha\rVert}$ that guarantees that $ {D^{(k)}_{\lVert\cdot\rVert,r(n)}}({\cal E}_n)\subset D^{(k)}_M({\cal R}_N)$ for some $M<\infty$. This is not yet an assumption but it requires bounding $ {\lVert} \beta(\alpha) {\rVert_1}$ by a constant times $ {\lVert\alpha\rVert}$, where this constant gives the size $r(n)$.
		Define the HAL-MLE $\psi_{N,\beta_n}$ over $D^{(k)}_M({\cal R}_N)$ for an $M$ large enough so that $\psi_{N,0}\in D^{(k)}_M({\cal R}_N)$.
		Then $P_n L(\psi_{N,\beta_n})\leq P_n L(\psi_{N,0})$.  
		Determine a PC solution $\alpha_n^*$ so that
		\[
		P_n L(\theta_{n,\alpha_n^*})=P_n L(\psi_{N,\beta_n}).\]
		This is possible since $L(\psi_{N,\beta_n})(O^n)$ only depends on $\psi_{N,\beta_n}(x^n)$ and the linear span of $\{\tilde{\phi}_m(x^n):m=1,\ldots,m_n\}$ is identical to the linear span of $\{\phi_j(x^n): j\in  {{\cal R}_N}\}$.
		The key assumption we make is that
		\[
		{\lVert} \alpha_n^* {\rVert}=O_P(r(n)).\]
		Then, we can select a $C(n)=O(r(n))$ so that $D^{(k)}_{\pl\cdot\pl,C_n}({\cal E}_n)$ contains $\theta_{n,\alpha_n^*}$ and we also have that
		$D^{(k)}_{\pl \cdot\pl,C_n}({\cal E}_n)\subset D^{(k)}_M({\cal R}_N)$ for some $M$. 
		Let $\alpha_n=\arg\min_{\alpha, {\lVert \alpha\rVert}<C_n}P_n L(\theta_{n,\alpha})$ so that $\theta_{n,\alpha_n}$ represents the minimizer of the empirical risk over $ {D^{(k)}_{\lVert\cdot\rVert,C_n}}({\cal E}_n)$. 
		Since $P_n L(\theta_{n,\alpha_n^*})=P_n L(\psi_{N,\beta_n})$ we have
		\[
		P_nL(\theta_{n,\alpha_n})\leq P_n L(\theta_{n,\alpha_n^*})=P_n L(\psi_{N,\beta_n})\leq P_n L(\psi_{N,0}).\]
		We can now carry out the rate of convergence proof for $d_0(\theta_{n,\alpha_n},\psi_{N,0})$ as follows.
		Let $L(\theta,\psi)\equiv L(\theta)-L(\psi)$ as short-hand notation. 
		We have 
		\begin{eqnarray*}
			d_0(\theta_{n,\alpha_n},\psi_{N,0})&\leq& -(P_n-P_0)L(\theta_{n,\alpha_n},\psi_{N,0})\\
			&&+P_nL(\theta_{n,\alpha_n})-P_n L(\psi_{N,0})\\
			&\leq & -(P_n-P_0)L(\theta_{n,\alpha_n},\psi_{N,0}).
		\end{eqnarray*}
		In addition, since $\{\theta_{n,\alpha}: {\lVert\alpha\rVert}<C_n\}\subset D^{(k)}_M({\cal R}_N)$, under a weak condition on $\psi\rightarrow L(\psi)$, we also have that $\{L(\theta_{n,\alpha},\psi_{N,0}): {\lVert \alpha\rVert}< C_n\}$ is contained in a class of functions with same covering number as
		$\{\theta_{n,\alpha}-\psi_{N,0}: {\lVert \alpha\rVert}<C_n\}\subset D^{(k)}_M({\cal R}_N)$.
		We also assumed $P_0\{L(\psi,\psi_0)\}^2=O(d_0(\psi,\psi_0))$. 
		The iterative general proof for establishing a rate of convergence of an MLE over a class of functions with specified bound on entropy integral now applies. This results in the rate of convergence $d_0(\theta_{n,\alpha_n},\psi_{N,0})=O^+_P(n^{-2k^*/(2k^*+1)})$ implied by the covering number/entropy integral of $D^{(k)}_M([0,1]^d)$.
		
		\section{Proof of Theorem \ref{rate}}

		%\subsection{Rate of convergence for general PC-HA minimum empirical risk estimators}
		We have the following lemma for establishing the desired rate of convergence for PC-HA minimum empirical risk estimators. It shows that the main challenge to be addressed is that the  loss-based dissimilarity projection $\psi_{n,0}$ onto $D^{(k)}({\cal E}_n)$ yields the desired approximation of $\psi_{N,0}$. 
		
		\begin{lemma}
			Define $\psi_{n,0}=\arg\min_{\psi\in D^{(k)}({\cal E}_n)}P_0L(\psi)$ and let $\alpha_{0,n}$ be so that $\psi_{n,0}=\theta_{n,\alpha_{0,n}}$. Recall
			\[
			\theta_{n,\alpha_{0,n}}=\sum_{j\in {\cal R}_N}\beta(\alpha_{0,n})(j)\phi_j=\sum_{m\in {\cal E}_n}\alpha_{0,n}(m)\tilde{\phi}_m.\]
			Propose a norm $ {\lVert \alpha\rVert}$. Compute a corresponding regularized MLE $\psi_n=\theta_{n,\alpha_n}$ with $\alpha_n=\arg\min_{ {\lVert \alpha\rVert}<C_n}P_n L(\theta_{n,\alpha})$ with $C_n$ a cross-validation selector. We wish to analyze $d_0(\psi_n,\psi_{N,0})$ with $\psi_{N,0}=\arg\min_{\psi\in D^{(k)}({\cal R}_N)}P_0 L(\psi)$.
			
			Assume $\psi_{N,0}\in D^{(k)}_M({\cal R}_N)$ for some $M<\infty$, and
			$\psi_{n,0}\in D^{(k)}_{M}({\cal E}_n)$ for some finite $M<\infty$ (i.e., $ {\lVert} \beta(\alpha_{0,n}) {\rVert_1}<M$). In addition, show $d_0(\psi_{n,0},\psi_{N,0})=O^+(n^{-k^*/(2k^*+1)})$.

			%{\bf Step 2:} Assume there exists a $C_{0,n}$ so that $\pl \alpha_{0,n}\pl<C_{0,n}$ and that $\sup_{\pl\alpha\pl<C_{0,n}}\pl \beta(\alpha)\pl_1 =O(1)$. 

			Then, $d_0(\psi_n,\psi_{N,0})=O^+_P(n^{-k^*/(2k^*+1)})$. Combined with $d_0(\psi_{N,0},\psi_0)=O^+_P(n^{-k^*/(2k^*+1)})$ this also yields
			$d_0(\psi_n,\psi_0)=O^+_P(n^{-k^*/(2k^*+1)})$. 
		\end{lemma}
		
		{\bf Proof:}
		Let $C_{0,n}= {\lVert} \alpha_{0,n} {\rVert}$. Then, by assumption, $ {D^{(k)}_{\lVert\cdot\rVert,C_{0,n}}}({\cal E}_n)$ contains $\psi_{n,0}$ and it is contained in $D^{(k)}_M({\cal R}_N)$ for some $M<\infty$. 
		We can then define $\alpha_{n,C_{0,n}}=\arg\min_{ {\lVert \alpha\rVert}<C_{0,n}}P_n L(\theta_{n,\alpha})$ and it can be analyzed as an estimator of $\alpha_{0,n}$, where $ {D^{(k)}_{\lVert\cdot\rVert,C_{0,n}}}({\cal E}_n)$ contains $\psi_{n,0}$ and it is contained in $D^{(k)}_M({\cal R}_N)$ for some $M<\infty$. 
		The typical rate of convergence proof shows $d_0(\theta_{n,\alpha_{n,C_{0,n}}},\theta_{n,\alpha_{0,n} } )=O^+_P(n^{-2k^*/(2k^*+1)})$.  Let $\alpha_n=\alpha_{n,C_{n,cv}}$ be the above estimator using the cross-validation selector $C_{n,cv}$. 
		By the oracle inequality for the cross-validation selector the same rate applies to $\theta_{n,\alpha_n}$: $d_0(\theta_{n,\alpha_n},\theta_{n,\alpha_{0,n}})=O^+_P(n^{-2k^*/(2k^*+1)})$. 
		By assumption, we have $d_0(\psi_{n,0},\psi_{N,0})=O^+(n^{-k^*/(2k^*+1)})$ so that we also obtain
		$d_0(\theta_{n,\alpha_n},\psi_{N,0})=O^+(n^{-k^*/(2k^*+1)})$.
		This proves the lemma. $\Box$
		
		Recall $\alpha_{0,n}=\arg\min_{\alpha}P_0 L(\theta_{n,\alpha})$ and $\psi_{n,0}=\theta_{n,\alpha_{n,0}}=\sum_{m=1}^n \alpha_{0,n}(m)\tilde{\phi}_m$.\nl
		We now also  define $\tilde{\alpha}_{0,n}=\arg\min_{\alpha}P_{X,n}(\sum_{m=1}^{m_n}\alpha(m)\tilde{\phi}_m-\psi_{N,0})^2$ and $\tilde{\psi}_{n,0}=\theta_{n,\tilde{\alpha}_{0,n}}$.
		
		\begin{lemma}
			Given $ {\lVert} \beta(\psi_{N,0}) {\rVert_2}=O(n^{-1/2})$, and thereby also $\psi_{N,0}\in D^{(k)}_M({\cal R}_N)$, we have that $\psi_{n,0},\tilde{\psi}_{n,0}\in D^{(k)}_M({\cal E}_n)$ for some $M$, or, equivalently, that $ {\lVert} \beta(\alpha_{0,n}) {\rVert_1}<M$ and $ {\lVert} \beta(\tilde{\alpha}_{0,n}) {\rVert_1}<M$ with probability tending to 1. Specifically, $ {\lVert} \beta(\psi_{n,0}) {\rVert_2}=O(n^{-1/2})$ and $ {\lVert} \beta(\tilde{\psi}_{n,0}) {\rVert_2}=O(n^{-1/2})$. In addition, $d_{2,0}(\tilde{\psi}_{n,0}, {\psi_{N,0}})= {O^+_P(n^{-2k^*/(2k^*+1)})}$, where
			$d_{2,0}(\psi,\psi_{N,0})=\int (\psi-\psi_{N,0})^2dP_X(x)$. This implies also $d_0(\psi_{n,0},\psi_{N,0})= {O^+_P(n^{-2k^*/(2k^*+1)})}$.\end{lemma}
		{\bf Proof:}
		\begin{itemize}
			\item Let's first consider the properties of $\tilde{\phi}_m(x^n)=HE_m$, $m=1,\ldots,N$, in $L^2(P_{X,n})$, where $H$ is an $n\times N$-matrix with values $H(i,j)=\phi_j(x_i)$. 
			We have $\langle HE_{m_1},HE_{m_2}\rangle_n=\langle H^{\top}H E_{m_1},E_{m_2}\rangle =\lambda_{m_1} \langle E_{m_1},E_{m_2}\rangle =0$ for all $m_1\not =m_2$. 
			Moreover, for $m>n$ with $\lambda_m=0$, we have $\langle HE_m,HE_m\rangle_n=\langle H^{\top}HE_m,E_m\rangle
			=\lambda_m \langle E_m,E_m\rangle =\lambda_m=0$. So we know that all the eigenvectors with eigenvalues zero result in basis functions $\tilde{\phi}_m$ that equal zero in $L^2(P_{X,n})$.
			We will use this result below.
			\item We can consider $D^{(k)}({\cal E}_n)$ as functions in $L^2(P_{X,n})$ and carry out a projection of $\psi_{N,0}$ onto $D^{(k)}({\cal E}_n)$. This projection $\Pi_{L^2(P_{X,n})}(\psi_{N,0}\mid D^{(k)}({\cal E}_n))$ is the minimizer of 
			$P_{X,n}(\psi_{N,0}-\theta_{n,\alpha})^2$ over $D^{(k)}({\cal E}_n)=\{\theta_{n,\alpha}:\alpha\}$.
			Thus, $\tilde{\psi}_{n,0}=\arg\min_{\psi\in D^{(k)}({\cal E}_n) }P_{X,n}(\psi-\psi_{N,0})^2$ equals this projection.
			One approach for obtaining an expression for $ {\lVert} \beta(\tilde{\psi}_{n,0}) {\rVert_2^2}$ is the following.
			We have that $\tilde{\phi}_m$ are orthogonal in $L^2(P_{X,n})$. Thus,
			\[
			\tilde{\psi}_{n,0}=\sum_{m=1}^n \langle \psi_{N,0},\tilde{\phi}_m\rangle_n/ {\lVert}\tilde{\phi}_m {\rVert_n^2}\, \tilde{\phi}_m.\]
			Thus, $ {\lVert} \alpha(\tilde{\psi}_{n,0}) {\rVert_2^2}= {\lVert} \beta(\tilde{\psi}_{n,0}) {\rVert_2^2}$ equals
			\[
			{\lVert} \beta(\tilde{\psi}_{n,0}) {\rVert_2^2}=\sum_{m=1}^n \langle \psi_{N,0},\tilde{\phi}_m^*\rangle_n^2,\] where we use $\tilde{\phi}_m^*=\tilde{\phi}_m/ {\lVert} \tilde{\phi}_m {\rVert_n}$. 
			This strongly suggests that indeed $ {\lVert} \beta(\tilde{\psi}_{n,0}) {\rVert_2^2}=O(n^{-1})$, but below we provide a formal proof through another argument.
			\item Let's now do the projection of $\psi_{N,0}$ onto $D^{(k)}({\cal E}_n)$ w.r.t. norm (and inner product) $ {\lVert\cdot\rVert_N}$ instead of $ {\lVert \cdot\rVert_n}$. We will show that the resulting projection $\tilde{\psi}_{n,0}^*$ equals $\tilde{\psi}_{n,0}$ defined above subsequently.
			In this Hilbert space with inner product $\langle\dot,\cdot\rangle_N$ we can write $\psi_{N,0}=\sum_{m=1}^N \langle \beta(\psi_{N,0}),E_m\rangle_N \tilde{\phi}_m$, where we now work in Hilbert space with inner product $\langle f_1,f_2\rangle=\sum_j \beta(f_1)(j)\beta(f_2)(j)$. In that space $\tilde{\phi}_m, m=1,\ldots,N$ is an orthonormal basis.
			If we project $\psi_{N,0}$ in that space on $D^{(k)}({\cal E}_n)$, then we obtain
			$\tilde{\psi}_{n,0}^*\equiv \Pi_N(\psi_{N,0}\mid D^{(k)}({\cal E}_n))=\sum_{m=1}^n \langle \beta(\psi_{N,0}),E_m\rangle_N \tilde{\phi}_m$.
			For this projection we have
			\[
			{\lVert} \beta(\tilde{\psi}_{n,0}^*) {\rVert_2^2}=\sum_{m=1}^n \langle \beta(\psi_{N,0}),E_m\rangle_N^2.\]
			This is smaller than $\sum_{m=1}^N  {\langle} \beta(\psi_{N,0}) {,}E_m {\rangle_N^2}= {\lVert} \beta(\psi_{N,0}) {\rVert_2^2}=O(1/n)$ (by assumption).
			Thus, this proves  that this projection  $\tilde{\psi}_{0,n}^*$ does have an $L_2$-norm that is $O(n^{-1/2})$ and thereby also that
			$ {\lVert}\beta(\tilde{\psi}_{0,n}^*) {\rVert_1}=O(1)$.
			
			\item We will now prove that $\tilde{\psi}_{n,0}^*$ equals $\tilde{\psi}_{n,0}$ so that we obtain the desired $ {\lVert} \beta(\tilde{\psi}_{n,0}) {\rVert_2^2}=O(n^{-1})$.  Above we showed that $\tilde{\phi}_m(x_i)=0$ for all $x_i$ and for all $m>n$.
			Therefore, \[
			P_{X,n}\left(\psi_{N,0}-\theta_{n,\alpha}\right)^2=
			P_{X,n}\left( \psi_{n,0}^*-\theta_{n,\alpha}\right)^2.\]
			But this is minimized at 0 by $\theta_{n,\alpha}=\psi_{n,0}^*$. So this proves that indeed
			$\tilde{\psi}_{n,0}=\psi_{n,0}^*$. 
			This then also proves that 
			\[
			{\lVert} \beta(\tilde{\psi}_{n,0}) {\rVert_2^2}=O(1/n),\]
			and thereby that $ {\lVert} \beta(\tilde{\psi}_{n,0}) {\rVert_1}=O(1)$.
			This proves the first claim  about  $\tilde{\psi}_{0,n}$ of the lemma.
			\item Lemma \ref{lemmaa1} below proves that if $\tilde{\psi}_{n,0}\in D^{(k)}_M({\cal E}_n)$, then   $d_{2,0}(\tilde{\psi}_{n,0},\psi_{N,0})=O^+_P(n^{-2k^*/(2k^*+1)})$.
			Thus, we can now also claim that $d_{2,0}(\tilde{\psi}_{n,0}, {\psi_{N,0}})= {O^+_P(n^{-2k^*/(2k^*+1)})}$, where
			$d_{2,0}(\psi,\psi_{N,0})=\int (\psi-\psi_{N,0})^2dP_X(x)$.  This proves this claim of the lemma. At this point, we   have not shown yet that $\psi_{0,n}\in D^{(k)}_M({\cal E}_n)$. 
			\item By our equivalence assumption,   $d_0(\tilde{\psi}_{n,0},\psi_{N,0})=O^+(d_{2,0}(\tilde{\psi}_{n,0},\psi_{N,0}))$. 
			This proves that $d_0(\tilde{\psi}_{n,0},\psi_{N,0})=O^+(n^{-2k^*/(2k^*+1)})$.  
			This implies that
			$d_0(\psi_{n,0},\psi_{N,0})\leq d_0(\tilde{\psi}_{n,0},\psi_{N,0})=O^+(n^{-2k^*/(2k^*+1)})$.
			So we succeeded to prove the desired rate of convergence for $\psi_{n,0}$ as well. The only claim we did not show yet is that $\psi_{n,0}$ is an element in $D^{(k)}_M({\cal E}_n)$.
			\item We have proven all claims of the lemma, except that $ {\lVert}\beta(\psi_{n,0}) {\rVert_2^2}=O(n^{-1})$. Even though we proved the desired rate of convergence for $d_0(\psi_{n,0},\psi_{N,0})$, for the analysis of the MLE of $\psi_{n,0}$ we will need this result.
			Our approach is to characterize the minimizer $\psi_{n,0}=\theta_{n,\alpha_{n,0}}$ over $\alpha\rightarrow P_0 L(\theta_{n,\alpha})$
			as the result of a steepest descent algorithm that starts at $\tilde{\psi}_{n,0}$ and show that this algorithm cannot change $ {\lVert} \beta(\tilde{\psi}_{n,0}) {\rVert_2}$ by an order.
			Consider the criterion $\alpha\rightarrow R_{0n}(\alpha)\equiv P_0 L(\theta_{n,\alpha})$. Define a collection of paths $\alpha_{\delta}^h=\alpha+\delta h$ for any $h$ where $h$ lives in $\openr^n$ with standard Euclidean norm $ {\lVert h\rVert_2^2}=\sum_{m=1}^n h(m)^2$.
			The pathwise derivative of the criterion is given by $d/d\delta_0R_{0n}(\alpha_{\delta_0}^h)=
			P_0 \frac{d}{d\theta_{n,\alpha}}L(\theta_{n,\alpha})(\sum_{m=1}^n h(m)\tilde{\phi}_m)$. 
			This can be written as an inner product $\langle D^*_{\alpha},h\rangle=\sum_{m=1}^n h(m) P_0 d/d\theta_{n,\alpha}L(\theta_{n,\alpha})(\tilde{\phi}_m)$.
			Thus, the canonical gradient of this pathwise derivative of $R_{0n}(\alpha)$ is given by $D^*_{\alpha}(m)\equiv P_0 d/d\theta_{n,\alpha}L(\theta_{n,\alpha})(\tilde{\phi}_m)$, $m=1,\ldots,n$. 
			A local steepest descent path is given by $\alpha^{lsd}_{\delta}=\alpha+\delta D^*_{\alpha}$.
			We have $\frac{d}{d\delta_0}R_{0n}(\alpha^{lsd}_{\delta_0})=\langle D^*_{\alpha},D^*_{\alpha}\rangle$ at $\delta_0=0$.
			This defines a universal steepest descent path $\alpha^{usd}_{\delta}$ by recursively following the local steepest descent path.
			The universal steepest descent path satisfies for all $\delta$, not just for $\delta=\delta_0=0$:
			\[
			\frac{d}{d\delta}R_{0n}(\alpha^{usd}_{\delta})=\langle D^*_{\alpha^{usd}_{\delta}},D^*_{\alpha^{usd}_{\delta}}\rangle .
			\]
			Now define $\alpha^0=\tilde{\alpha}_{0,n}$, which we know satisfies $ {\lVert} \tilde{\alpha}_{0,n} {\rVert_2^2}=O(n^{-1})$.
			Then, we have a path $\alpha^{0,usd}_{\delta}$ starting at $\tilde{\alpha}_{0,n}$ with $\delta=0$. 
			Let $\delta_n=\arg\min_{\delta}R_{0n}(\alpha^{0,usd}_{\delta})$. Then, $\alpha_{0,n}=\alpha^{0,usd}_{\delta_n}$.
			This defines a path from $\tilde{\alpha}_{0,n}$ at $\delta=0$ to $\alpha_{0,n}$ at $\delta=\delta_n$.
			We want to bound $ {\lVert} \alpha^{0,usd}_{\delta_n} {\rVert_2}$. Due to the rate of convergence established for $\theta_{n,\tilde{\alpha}_{0,n}}$ we have $R_{0n}(\tilde{\alpha}_{0,n})-R_{0n}(\alpha_{0,n})\leq R_{0n}(\tilde{\alpha}_{0,n})-R_{0n}(\psi_{N,0})=O^+(n^{-2k^*/(2k^*+1)})$.
			Consider the pessimistic scenario for this initial $\tilde{\alpha}_{0,n}$ that
			this rate is not faster than $r(n)=n^{-2k^*/(2k^*+1)}$ (up till $\log n$-factors), thereby allowing a maximal move $\delta_n$ to change this initial $ {\lVert\cdot\rVert_2}$-norm. This corresponds with
			$P_0 d/d\theta_{n,\alpha}L(\theta_{n,\alpha})(\tilde{\phi}_m)\sim r(n)$ due to using $P_0d/d\psi_0L(\psi_0)(\tilde{\phi}_m)=0$ (for example,  for the  least squares loss we have $D^*_{\tilde{\alpha}_{0,n}}(m)= P_0 \tilde{\phi}_m(\tilde{\psi}_{0,n}-\psi_0)$).
			Therefore, this pessimistic scenario corresponds with $ {\lVert} D^*_{\tilde{\alpha}_{0,n}} {\rVert_2}\sim n^{1/2}r(n)$. 
			
			For notational convenience, let $r(n)=n^{-2k^*/(2k^*+1)}$ and $R_{0n}(\tilde{\alpha}_{0,n})=R_{0n}(\delta_0)$ and 
			$R_{0n}(\alpha_{0,n})=R_{0n}(\delta_n)$. So we have
			$R_{0n}(\delta_n)-R_{0n}(0)=O(r(n))$.
			We also know that, by  {Taylor} expansion at $\delta=0$, $R_{0n}(\delta)-R_{0n}(0)\sim \delta \langle D^*_{\tilde{\alpha}_{0,n}},D^*_{\tilde{\alpha}_{0,n}}\rangle$. 
			So $\delta_n$ is of order $r(n)/\langle D^*_{\tilde{\alpha}_{0,n}},D^*_{\tilde{\alpha}_{0,n}}\rangle$. 
			So $\alpha_{0,n}\sim \tilde{\alpha}_{0,n}+r(n)/ {\lVert} D^*_{\tilde{\alpha}_{0,n}} {\rVert_2^2}\, D^*_{\tilde{\alpha}_{0,n}}$. 
			So we need to bound $r(n)/ {\lVert} D^*_{\tilde{\alpha}_{0,n}} {\rVert_2^2}\, {\lVert} D^*_{\tilde{\alpha}_{0,n}} {\rVert_2}$, which equals
			$r(n)/ {\lVert} D^*_{\tilde{\alpha}_{0,n}} {\rVert_2}$. Given that $ {\lVert} D^*_{\tilde{\alpha}_{0,n}} {\rVert_2}\sim n^{1/2} r(n)$, this yields the bound
			$ {\lVert} \alpha_{0,n}-\tilde{\alpha}_{0,n} {\rVert_2}=O(n^{-1/2})$, as desired. 
		\end{itemize}
		This completes the proof of the lemma. $\Box$

		The above proof used the  following lemma.
		\begin{lemma}\label{lemmaa1}
			Under Assumptions \ref{assumption1},\ref{assumption2}, and if  $\tilde{\psi}_{n,0},\psi_{n,0}\in D^{(k)}_M({\cal R}_N)$ for some $M<\infty$, we have $d_{2,0}(\tilde{\psi}_{n,0},\psi_{N,0})=O^+_P(n^{-2k^*/(2k^*+1)})$ and  $d_0(\psi_{n,0},\psi_{N,0})=O^+_P(n^{-2k^*/(2k^*+1)})$. 
		\end{lemma}
		{\bf Proof:}
		Consider $d_{2,n}(\psi,\psi_{N,0})=1/n \sum_i (\psi(x_i)-\psi_{N,0}(x_i))^2$. 
		Let  $\tilde{\psi}_{n,0}=\arg\min_{\psi\in D^{(k)}({\cal E}_n)}d_{2,n}(\psi,\psi_{N,0})$. We have that $\tilde{\psi}_{n,0}(x)=\psi_{N,0}(x)$ for all $x\in \{x_i: i\}$.
		Let $P_{X,n}$ be the empirical measure for $\{x_i: i\}$. Let $P_{X,0}$ be its counterpart. 
		Let $d_{2,0}(\psi,\psi_{N,0})=P_{X,0}(\psi-\psi_{N,0})^2$. Then, 
		\begin{eqnarray*}
			0&\leq &d_{2,0}(\tilde{\psi}_{n,0},\psi_{N,0})\\
			&=& -(P_{X,n}-P_{X,0})(\tilde{\psi}_{n,0}-\psi_{N,0})^2+P_{X,n}(\tilde{\psi}_{n,0}-\psi_{N,0})^2\\
			&=&-(P_{X,n}-P_{X,0})(\tilde{\psi}_{n,0}-\psi_{N,0})^2\\
		\end{eqnarray*}
		By assumption $\{x_i: i\}$ is an i.i.d. sample from $P_{X,0}$, then we can carry out the typical rate of convergence proof, using that $\tilde{\psi}_{n,0},\psi_{N,0}\in D^{(k)}_M({\cal R}_N)$ for some $M<\infty$. Then, it follows that $d_{2,0}(\tilde{\psi}_{n,0},\psi_{N,0})=O^+_P(n^{-2k^*/(2k^*+1)})$.
		By assumption,  $d_0(\tilde{\psi}_{n,0},\psi_{N,0})=O^+(d_{2,0}(\tilde{\psi}_{n,0},\psi_{N,0}))$. Then, we also have that $d_0(\tilde{\psi}_{n,0},\psi_{N,0})=O^+(n^{-2k^*/(2k^*+1)})$. This implies that
		$d_0(\psi_{n,0},\psi_{N,0})\leq d_0(\tilde{\psi}_{n,0},\psi_{N,0})=O^+(n^{-2k^*/(2k^*+1)})$, which completes the proof. 
		$\Box$
		\section{Proof of Theorem \ref{score:gl}}

		Consider paths $\alpha_{\delta}^h(m)=(1+\delta h(m))\alpha(m)$ through $\alpha$ that preserve $ {\lVert} \beta(\alpha) {\rVert_1}=C$ and are indexed by a direction $h$ that satisfies
		$\langle h,a(\alpha)\rangle=0$. We showed above that these paths preserve the $L_1$-norm of $\beta(\alpha)$ locally. 
		%\[a(\alpha)(m)\equiv \sum_j E(j,m) \ell(\beta(\alpha)(j)).\]
		
		Let's reparametrize the paths  $(1+\delta h)\alpha$ as $\alpha_{\delta}^{h}=\alpha(m)+\delta h_1(m)$ with $h_1(m)=\alpha(m)h(m)$. 
		Then, we obtain a class of paths  $\alpha_{\delta}^{h_1}=\alpha+\delta h_1$, where $h_1$ needs to be an element of the following $n-1$-dimensional subspace of $\openr^n$:
		\[
		{\cal H}_1(\alpha)\equiv \{h-\langle h,a_1(\alpha)\rangle/\langle a_1(\alpha),a_1(\alpha)\rangle a_1(\alpha): h=h I(\alpha\not =0)\},\]
		where $\langle h_1,h_2\rangle=\sum_m h_1(m)h_2(m)$ is the standard inner product and
		\[
		a_1(\alpha)(m)=\sum_{j\in {\cal R}_N}E(j,m)\ell(\beta(\alpha))(j).\]
		Thus, $a(\alpha)=\alpha a_1(\alpha)$. 
		Note that for $h\in {\cal H}_1(\alpha)$ we have that $h(m)$ equals zero if $\alpha(m)=0$.

		Thus the set of score equations solved by the generalized LASSO solution $\theta_n=\theta_{n,\alpha_n}$  is given by \[
		P_n \frac{d}{d\theta_n}L(\theta_n)(f)=0\mbox{ with }
		f\in {\cal T}_{\theta_n}({\cal E}_n)\equiv \left\{\sum_m I(\alpha_n(m)\not =0) h(m)\tilde{\phi}_m: h\perp a_1(\alpha)\right\}.
		\]
		This proves the following lemma about the class of score equations solved by PC-HAGL.
		\begin{lemma}
			Let $\tilde{\phi}_m=\sum_j E(j,m)\phi_j$ for a  given $p\times n$-matrix $(E(j,m):j=1,\ldots,p,m=1,\ldots,n)$, and, accordingly, define the vector $a_1(\alpha)(m)=\sum_j E(j,m)\ell(\beta(\alpha)(j))$.
			% Here $E$ could be the eigenvectors of of $H^t H$ defined above, but we could have orthonormalized $\tilde{\phi}_m$ w.r.t. another inner product as well. We have to keep in  mind that to have$\pl \beta(\alpha)\pl_2=\pl \alpha\pl_2$ we need the original eigenvector definition.
			
			Let $\alpha_n$ be the optimizer of $R_n(\alpha)=P_n L(\sum_m \alpha(m)\tilde{\phi}_m)$ under constraint $ {\lVert} \beta(\alpha) {\rVert_1}=C$. 
			Consider now any such path $\alpha_{n,\delta}^h=\alpha_n+\delta h$ with $h\in {\cal H}_1(\alpha_n)=\{h\in \openr^n: h\perp a_1(\alpha_n), h=hI(\alpha_n\not =0)\}$ through $\alpha_n$ at $\delta=0$.
			Since these paths preserve the $L_1$-norm constraint, $ {\lVert}\beta(\alpha_{n,\delta}^h) {\rVert_1}=C$ for small enough $\delta\approx 0$, we have that $\alpha_n$ solves the following score equations
			\[
			\frac{d}{d\delta_0}R_n(\alpha_{n,\delta_0}^h)=0\mbox{ for all $h\in {\cal H}_1(\alpha_n)$.}\]
			We have $d/d\delta_0L(\theta_{n,\alpha_{n,\delta_0}^h})=\frac{d}{d\theta_{n,\alpha_n} }L(\theta_{n,\alpha_n})(\sum_m h(m)\tilde{\phi}_m)$. Let's denote the latter term with $\dot{L}(\theta_{n,\alpha_n})(\sum_m h(m)\tilde{\phi}_m)$.
			Let $S_h(\alpha_n)=\dot{L}(\theta_{n,\alpha_n})(\sum_m h(m)\tilde{\phi}_m)$. 
			Then, we have
			\[
			P_n S_h(\alpha_n)=0\mbox{ for all $h\in {\cal H}_1(\alpha_n)=\{h:h\perp a_1(\alpha_n),h(m)=I( {\alpha_n(m)\neq 0})h(m)\}$}.\]
		\end{lemma}
		
		{We prove next that a particular non-penalized score equation is solved at $o_P(n^{-1/2})$.} Let $J_n\leq n$ be the number of non-zero coefficients among $\alpha_n(j)$, $j=1,\ldots,n$ in the fit $\theta_n=\sum_{m=1}^n \alpha_n(m)\tilde{\phi}_m$. If we run the MLE over all $D^{(k)}({\cal E}_n)$, then one generally will have that $J_n=n$, but one could screen out some of the PC-basis functions  {beforehand} (e.g {.}, based on  {the} size  {of the} eigenvalue). Let ${\cal E}(J_n)=\{j: \alpha_n(j)\not =0\}\subset {\cal E}_n$ be the indices for which the coefficient $\alpha_n$ is non-zero. Let $D^{(k)}({\cal E}(J_n))$ be the $J_n$-dimensional linear subspace of $D^{(k)}({\cal E}_n)$ defined by 
		$\sum_{m\in {\cal E}(J_n)}\alpha(m)\tilde{\phi}_m$.  For a $f\in D^{(k)}({\cal E}(J_n))$ we define the corresponding vector $(f(m): m\in {\cal E}(J_n))$ by $f=\sum_{m\in {\cal E}(J_n)} f(m)\tilde{\phi}_m$.
		Consider the inner product $\langle f_1,f_2\rangle \equiv \sum_{m\in {\cal E}_n} f_1(m)f_2(m)$ on $D^{(k)}({\cal E}_n)$ and thus also on $D^{(K)}({\cal E}(J_n))$.
		Then, $\tilde{\phi}_m$, $m=1,\ldots,n$, is an orthonormal basis of $D^{(k)}({\cal E}_n)$ w.r.t.\ this inner product. Let $f_{a(\alpha)}\equiv \sum_{m\in {\cal E}(J_n)} a_1(\alpha)(m)/ {\lVert a_1(\alpha)\rVert_2}\, \tilde{\phi}_m$, and define
		${\cal D}_n({\cal E}(J_n))=\{f\in D^{(k)}({\cal E}(J_n)): f\perp f_{a(\alpha_n)}\}$. For notational convenience, we will now and then suppress the dependence of $f_{a(\alpha)}$ on $\alpha$ and denote $f_{a(\alpha_n)}=f_{a_n}$. Note that $f_{a_n}$ is normalized to have norm 1.
		Let $T_{\theta_n}({\cal E}(J_n))=\{S_h(\alpha_n):h\in {\cal H}_1(\alpha_n)\}$ be the set of scores whose empirical means are set to zero under  $\theta_n$ be the previous lemma. 
		Let $D_n({\cal E}(J_n))=\{\sum_{m\in {\cal E}(J_n)}h(m)\tilde{\phi}_m: h\perp a_1(\alpha_n)\}$ so that
		$T_{\theta_n}({\cal E}(J_n))=\left\{\frac{d}{d\theta_n}L(\theta_n)(f):f\in  {D_n({\cal E}(J_n))}\right\}$.
		
		{\bf Goal:}
		We wish to analyze $P_n \frac{d}{d\theta_n}L(\theta_n)(f^*=\sum_{m\in {\cal E}(J_n)} f^*(m)\tilde{\phi}_m)$ for a given $f^*\in D^{(k)}_M({\cal E}(J_n))$, where we do not require that $f^*$ is orthogonal to $f_{a_n}$.
		
		{\bf Analysis:}
		We can decompose $f^*=\Pi(f^*\mid f_{a_n}^{\perp})+\Pi(f^*\mid f_{a_n})$, where  latter equals $\langle f^*,f_{a_n}\rangle f_{a_n}$. We have that $\Pi(f^*\mid f_{a_n}^{\perp})\in T_{\theta_n}({\cal E}(J_n))$ so that $P_n \frac{d}{d\theta_n}L(\theta_n)(\Pi(f^*\mid f_{a_n}^{\perp})) =0$. 
		Thus, it remains to bound $P_n \frac{d}{d\theta_n}L(\theta_n)(\Pi(f^*\mid f_{a_n}))$. This shows that for $f^*\in D^{(k)}({\cal E}(J_n))$
		\[
		P_n \frac{d}{d\theta_n}L(\theta_n)(f^*)=
		\langle f^*,f_{a_n}\rangle P_n \frac{d}{d\theta_n}L(\theta_n)(f_{a_n}).
		\]
		A conservative bound is $ {\lvert\langle f^*,f_{a_n}\rangle\rvert\leq \lVert f^*\rVert_2\lVert f_{a_n}\rVert_2=\lVert f^*\rVert_2}$, using that $ {\lVert} f_{a_n} {\rVert_2}=1$.
		Let $f_{a_n}^*=f_{a_n}/n^{1/2}$ so that $ {\lVert} \beta(f_{a_n}^*) {\rVert_1}=O(1)$.
		Then we obtain
		\[
		\mid P_n d/d\theta_nL(\theta_n)(f^*)\mid \leq 
		\mid P_n d/d\theta_nL(\theta_n)(f_{a_n}^*)\mid .
		\]
		Let $\tilde{f}_{a_n}\equiv \Pi_{P_0}(f_{a_n}^*\mid D_n({\cal E}(J_n)))$, where the projection is in $L^2(P_0)$.
		Then, we can write
		\[
		P_n d/d\theta_nL(\theta_n)(f_{a_n}^*)=P_n d/d\theta_nL(\theta_n)(f_{a_n}^*-\tilde{f}_{a_n}).\]
		This results in the following expression:
		\begin{eqnarray*}
			P_n d/d\theta_nL(\theta_n)(f_{a_n}^*)&=&(P_n-P_0)d/d\theta_nL(\theta_n)(f_{a_n}^*-\tilde{f}_{a_n})\\
			&&+P_0 (d/d\theta_nL(\theta_n)-d/d\psi_0L(\psi_0))(f_{a_n}^*-\tilde{f}_{a_n}).
		\end{eqnarray*}
		Under the assumption that $ {\lVert} \frac{d}{d\theta_n}L(\theta_n)( f_{a_n}^*-\tilde{f}_{a_n}) {\rVert_{P_0}}=o_P(1)$, it follows that the empirical process term is $o_P(n^{-1/2})$. 
		For the second term we assume that it can be bounded by
		$ {\lVert} \theta_n-\psi_0 {\rVert_{\infty}}$ times $ {\lVert} f_{a_n}^*-\tilde{f}_{a_n} {\rVert_{1,P_0}}$, the $L_1(P_0)$-norm. 
		Alternatively, one uses Cauchy--Schwarz inequality to bound it by $d_0^{1/2}(\theta_n,\psi_0)$ times $ {\lVert} f_{a_n}^*-\tilde{f}_{a_n} {\rVert_{P_0}}$, the $L^2(P_0)$-norm. We obtain a better result using the $L_1(P_0)$-norm bound so we will use that, utilizing that we have sup-norm rates of convergence for $\theta_n-\psi_0$.
		For example, for first order HAL, we would have
		$ {\lVert} \theta_n-\psi_0 {\rVert_{\infty}}=O^+_P(n^{-2/5})$ so that we only need $ {\lVert} f_{a_n}^*-\tilde{f}_{a_n} {\rVert_{1,P_0}}=O_P(n^{-1/10-\delta})$ for some $\delta>0$.
		
		\section{Approximate score solution for PC-HAL and PC-HAR}\label{app:score}

		\subsection{Result for PC-HAL}
		Let $\ell_{\alpha_n}(m)=I(\alpha_n(m)>0)-I(\alpha_n(m)<0)$ be the sign of $\alpha_n(m)$, $m\in {\cal E}(J_n)$. 
		The set ${\cal E}(J_n)$ can be as large as $n$ in the PC-HAGL and PC-HAR case but for PC-HAL, the Lasso will set various coefficients equal to zero, so that even when one starts with all $n$ PC-basis functions ${\cal E}_n$, the set ${\cal E}(J_n)$ will have been downsized.
		We will also denote $\ell_{\alpha_n}$ with $\ell_n$.
		Let $f_{\ell_n}=\sum_{m\in {\cal E}(J_n)}\ell_n(m)/J_n^{1/2}\, \tilde{\phi}_m$ and note that $ {\lVert} f_{\ell_n} {\rVert_2^2}=\sum_{m\in {\cal E}(J_n)}(\ell_n(m)/J_n^{1/2})^2=1$, so $f_{\ell_n}$ is normalized to have norm 1 w.r.t.\ the standard inner product. 
		We also define $f^*_{\ell_n}=f_{\ell_n}/J_n^{1/2}= {J_n^{-1}}\sum_{m\in {\cal E}(J_n)} \ell_n(m)\tilde{\phi}_m$, which has bounded $L_1$-norm $ {\lVert} \beta(f^*_{\ell_n}) {\rVert_1}=1$, while $ {\lVert} f^*_{\ell_n} {\rVert_2}=J_n^{-1/2}$. 
		
		Consider an $f^*\in D_M({\cal E}(J_n))$ so that $f^*=\sum_{m\in {\cal E}(J_n)}\alpha(f^*)(m)\tilde{\phi}_m$ and $ {\lVert} \alpha(f^*) {\rVert_1}\leq M$.
		We will assume the slightly stronger assumption that $ {\lVert} \alpha(f^*) {\rVert_2}=O(J_n^{-1/2})$, which implies $\pl\alpha(f^*)\pl_1=O(1)$.
		Let $\theta_n=\arg\min_{\theta\in D_{C_n}({\cal E}(J_n))}P_n L(\theta)$ be the PC-HAL constraining the $L_1$ norm to be bounded by $C_n$ such as a cross-validation selector. We note that $\theta_n$ solves the following score equations:
		\[
		P_n \frac{d}{d\theta_n}L(\theta_n)\left(\sum_{m\in {\cal E}(J_n)}\alpha(m)\tilde{\phi}_m\right)\mbox{ for $\alpha\perp\ell_{\alpha_n}$.}\]
		Let's denote $D({\cal E}(J_n))$ with an inner product $\langle f_1,f_2\rangle_n=\sum_{m\in {\cal E}(J_n)}\alpha(f_1)(m)\alpha(f_2)(m)$ for $f_1=\sum_{m\in {\cal E}(J_n)}\alpha(f_1)(m)\tilde{\phi}_m$ and similarly for $f_2$.
		Let $D_n({\cal E}(J_n))\equiv \{f\in D({\cal E}(J_n)): f\perp f_{\ell_n}\}$ be the $J_n-1$-dimensional subspace of $D({\cal E}(J_n))$. Then, 
		$P_n d/d\theta_nL(\theta_n)(f)=0$ for all $f\in D_n({\cal E}(J_n))$.
		We can decompose $f^*=(f^*-\langle f^*,f_{\ell_n}\rangle_n f_{\ell_n})+\langle f^*,f_{\ell_n}\rangle_n f_{\ell_n}$ where the first term is an element of $D_n({\cal E}(J_n))$. Therefore
		\begin{eqnarray*}
			P_n d/d\theta_nL(\theta_n)(f^*)
			&=&J_n^{1/2} \langle f^*,f_{\ell_n}\rangle_n  P_n d/d\theta_nL(\theta_n)(f_{\theta_n}^*).
		\end{eqnarray*}
		A conservative bound is $ {\lvert\langle f^*,f_{\ell_n}\rangle_n\rvert \le \lVert f^*\rVert_2\lVert f_{\ell_n}\rVert_2 = \lVert f^*\rVert_2 = J_n^{-1/2}}$, using that $ {\lVert} f_{\ell_n} {\rVert_2}=1$.
		Then we obtain
		\[
		\mid P_n d/d\theta_nL(\theta_n)(f^*)\mid \leq 
		\mid P_n d/d\theta_nL(\theta_n)(f_{\ell_n}^*)\mid .
		\]
		Let $\tilde{f}_{\ell_n}\equiv \Pi_{P_0}(f_{\ell_n}^*\mid D_n({\cal E}(J_n)))$, where the projection is in $L^2(P_0)$.
		Then, we can write
		\[
		P_n d/d\theta_nL(\theta_n)(f_{\ell_n}^*)=P_n d/d\theta_nL(\theta_n)(f_{\ell_n}^*-\tilde{f}_{\ell_n}).\]
		This results in the following expression:
		\begin{eqnarray*}
			P_n d/d\theta_nL(\theta_n)(f_{\ell_n}^*)&=&(P_n-P_0)d/d\theta_nL(\theta_n)(f_{\ell_n}^*-\tilde{f}_{\ell_n})\\
			&&+P_0 (d/d\theta_nL(\theta_n)-d/d\psi_0L(\psi_0))(f_{\ell_n}-\tilde{f}_{\ell_n}).
		\end{eqnarray*}
		Under the assumption that $ {\lVert} \frac{d}{d\theta_n}L(\theta_n)( f_{\ell_n}^*-\tilde{f}_{\ell_n}) {\rVert_{P_0}}=o_P(1)$, it follows that the empirical process term is $o_P(n^{-1/2})$. 
		For the second term we assume that it can be bounded by
		$ {\lVert} \theta_n-\psi_0 {\rVert_{\infty}}$ times $ {\lVert} f_{\ell_n}^*-\tilde{f}_{\ell_n} {\rVert_{1,P_0}}$, the $L_1(P_0)$-norm. 
		Alternatively, one uses Cauchy--Schwarz inequality to bound it by $d_0^{1/2}(\theta_n,\psi_0)$ times $ {\lVert} f_{\ell_n}^*-\tilde{f}_{\ell_n} {\rVert_{P_0}}$, the $L^2(P_0)$-norm. We obtain a better result using the $L_1(P_0)$-norm bound so we will use that, utilizing that we have sup-norm rates of convergence for $\theta_n-\psi_0$.
		For example, for first order HAL, we would have
		$ {\lVert} \theta_n-\psi_0 {\rVert_{\infty}}=O^+_P(n^{-2/5})$ so that we only need $ {\lVert} f_{\ell_n}^*-\tilde{f}_{\ell_n} {\rVert_{1,P_0}}=O_P(n^{-1/10-\delta})$ for some $\delta>0$.
		
		Let's now present a conservative bound for $ {\lVert} f_{\ell_n}^*-\tilde{f}_{\ell_n} {\rVert_{1,P_0}}$.
		For a given $f_{\alpha}=\sum_{m\in {\cal E}(J_n)}\alpha(m)\tilde{\phi}_m$, we have
		$f_{\ell_n}^*-f_{\alpha}=\sum_{m\in {\cal E}(J_n)}(\ell_n(m)/J_n -\alpha(m))\tilde{\phi}_m$ for some $\alpha$.
		Suppose we select $\alpha(m)=\ell_n(m)/J_n$ for all $m\not =m^*$ so that $f_{\ell_n}^*-f_{\alpha}=(\ell_n(m^*)/J_n-\alpha(m^*)) \tilde{\phi}_{m^*}$. Then, it remains to set $\alpha(m^*)$ so that 
		$\sum_{m\in {\cal E}(J_n)}\alpha(m)\ell_n(m)=0$. It follows that $\alpha(m^*)=-\sum_{m\not =m^*}\ell_n(m)^2 J_n^{-1}/\ell_n(m^*)$. Recall $\ell_n(m)^2=1$ so that $\alpha(m^*)=-\sum_{m\not =m^*}J_n^{-1}/\ell_n(m^*)$.
		Then, we obtain
		\[
		f_{\ell_n}^*-f_{\alpha}=\ell_n(m^*)\tilde{\phi}_{m^*}.\]
		Thus, we obtain the bound
		\[
		{\lVert} f_{\ell_n}^*-\tilde{f}_{\ell_n} {\rVert_{P_0}}=O_P(1)\, {\lVert} \tilde{\phi}_{m^*} {\rVert_{1,P_0}}.\]
		Thus, our sufficient condition is satisfied when $\min_{m} {\lVert}\tilde{\phi}_m {\rVert_{1,P_0}}=O_P(n^{-1/10-\delta})$ for some $\delta>0$.
		
		This proves the following theorem.
		\begin{theorem}
			{\bf Definitions;}
			Let $\ell_n=\ell_{\alpha_n}(m)=I(\alpha_n(m)>0)-I(\alpha_n(m)<0)$ be the sign of $\alpha_n(m)$, $m\in {\cal E}(J_n)$. 
			Let $f_{\ell_n}(m)=\ell_n(m)/J_n^{1/2}$ so that $ {\lVert} f_{\ell_n} {\rVert_2}=1$ is normalized to have norm 1 w.r.t.\ the standard inner product. 
			We also define $f^*_{\ell_n}=f_{\ell_n}/J_n^{1/2}= {J_n^{-1}}\sum_{m\in {\cal E}(J_n)} \ell_n(m)\tilde{\phi}_m$ which now has bounded $L_1$-norm $ {\lVert} \beta(f^*_{\ell_n}) {\rVert_1}=1$, while $ {\lVert} f^*_{\ell_n} {\rVert_2}=J_n^{-1/2}$.  
			Consider an $f^*\in D_M({\cal E}(J_n))$ so that $f^*=\sum_{m\in {\cal E}(J_n)}\alpha(f^*)(m)\tilde{\phi}_m$ and $ {\lVert} \alpha(f^*) {\rVert_1}\leq M$.
			We will assume the slightly stronger assumption that $ {\lVert} \alpha(f^*) {\rVert_2}=O(n^{-1/2})$, which implies $\pl\alpha(f^*)\pl_1=O(1)$.
			Let $\theta_n$ be an PC-HAL constraining the $L_1$-norm  by $C_n$, and let ${\cal E}(J_n)$ be the set of basis functions $\tilde{\phi}_m$ with non-zero coefficients. Then, $\theta_n=\arg\min_{Q\in D_{1,C_n}({\cal E}(J_n))}P_n L(Q)$. Due to latter,  $\theta_n$ solves the following score equations:
			\[
			P_n \frac{d}{d\theta_n}L(\theta_n)\left(\sum_{m\in {\cal E}(J_n)}\alpha(m)\tilde{\phi}_m\right )\mbox{ for $\alpha\perp\ell_{\alpha_n}$.}\]
			Let's endow $D({\cal E}(J_n))$ with an inner product $\langle f_1,f_2\rangle_n=\sum_{m\in {\cal E}(J_n)}\alpha(f_1)(m)\alpha(f_2)(m)$ for $f_1=\sum_{m\in {\cal E}(J_n)}\alpha(f_1)(m)\tilde{\phi}_{ {m}}$ and similarly for $f_2$.
			Let $D_n({\cal E}(J_n))\equiv \{f\in D({\cal E}(J_n)): f\perp f_{\ell_n}\}$ which is a  $J_n-1$-dimensional subspace of $D({\cal E}(J_n))$.
			Let $\tilde{f}_{\ell_n}=\arg\min_{f\in D_n({\cal E}(J_n))}P_0(f_{\ell_n}^*-f)^2$.
			
			{\bf Assumptions:}
			Covering number of $\{d/d\theta_nL(\theta_n)(f): f\in D_M({\cal E}(J_n)),\theta_n\in D_M({\cal E}(J_n))\}$ is of same order as covering number of $D_M([0,1]^d)$;
			$ {\lVert} \theta_n-\psi_0 {\rVert_{\infty}} {\lVert} f_{\ell_n}^*-\tilde{f}_{\ell_n} {\rVert_{1,P_0}}=o_P(n^{-1/2})$ or $d_0^{1/2}(\theta_n,\psi_0) {\lVert} f_{\ell_n}^*-\tilde{f}_{\ell_n} {\rVert_{P_0}}=o_P(n^{-1/2})$. A conservative bound for $ {\lVert} f_{\ell_n}^*-\tilde{f}_{\ell_n} {\rVert}$ is given by $\min_{m\in {\cal E}(J_n)} {\lVert} \tilde{\phi}_m {\rVert}$ for either the $L^2(P_0)$ or $L_1(P_0)$-norm.

			{\bf Conclusion:} We have $P_n d/d\theta_nL(\theta_n)(f^*)=P_n d/d\theta_nL(\theta_n)(f_{\ell_n}^*-\tilde{f}_{\ell_n})$. Under the above conditions, we have
			\[
			\sup_{f^*\in D_M({\cal E}(J_n)), {\lVert} f^* {\rVert_2}\leq J_n^{-1/2}}\mid P_n \frac{d}{d\theta_n}L(\theta_n)( f^*)\mid =o_P(n^{-1/2}).\]
			
		\end{theorem}

		\subsection{Analogue result for PC-HAR}
		Let $\alpha_n$ be the PC-HAR estimator where $ {\lVert \alpha_n\rVert_2}=C_n$. 
		Consider a path $\alpha_n+\delta h$ and $ {\lVert}\alpha_n+\delta h {\rVert_2^2}=\sum_m (\alpha_n(m)+\delta h(m))^2$.
		This Euclidean norm equals $\sum_m \alpha_n(m)^2+\sum_m 2\delta \alpha_n(m) h(m)+\sum_m \delta^2 h(m)^2$.
		If we assume that $h$ satisfies $\sum_m h(m)\alpha_n(m)=0$, then $d/d\delta_0\, {\lVert} \alpha_n^h(\delta) {\rVert_2}=0$ at $\delta_0=0$.
		Therefore, it suffices to restrict to paths $\{\alpha_n+\delta h: h\perp \alpha_n\}$.
		Then, $\frac{d}{d\delta_0}P_n L(\theta_{n,\alpha_n^h(\delta_0)})=0$ for all $h\perp \alpha_n$.
		Thus, $\theta_n=\theta_{n,\alpha_n}$ solves the following class of score equations:
		\[
		0=P_n S_{\theta_n}\left(\sum_m h(m)\tilde{\phi}_m\right )\mbox{ for all $h\perp \alpha_n$.}
		\]
		We can now go through same proof as above with $f_{\ell_n}$ replaced by $f_{\alpha_n}\equiv \sum_{m\in {\cal E}(J_n)}\alpha_n(m)\tilde{\phi}_m$.
		
		In the last part of the proof we have to obtain a bound for $ {\lVert} f_{\alpha_n}^*-\tilde{f}_{\alpha_n} {\rVert_{1,P_0}}$.
		For a given $f_{\alpha}=\sum_{m\in {\cal E}(J_n)}\alpha(m)\tilde{\phi}_m$, we have
		$f_{\alpha_n}^*-f_{\alpha}=\sum_{m\in {\cal E}(J_n)}(\alpha_n(m)/J_n -\alpha(m))\tilde{\phi}_m$ for some $\alpha$.
		Suppose we select $\alpha(m)=\alpha_n(m)/J_n$ for all $m\not =m^*$ so that $f_{\alpha_n}^*-f_{\alpha}=(\alpha_n(m^*)/J_n-\alpha(m^*)) \tilde{\phi}_{m^*}$. Then, it remains to set $\alpha(m^*)$ so that 
		$\sum_{m\in {\cal E}(J_n)}\alpha(m)\alpha_n(m)=0$. It follows that $\alpha(m^*)=-
		\alpha_n(m^*)^{-1} \sum_{m\not =m^*}\alpha_n(m)^2 J_n^{-1}$. 
		Then, we obtain
		\[
		\begin{array}{l}
			{f_{\alpha_n}^*}-f_{\alpha}=\left(\alpha_n(m^*)/J_n+\alpha_n(m^*)^{-1}\sum_{m\not =m^*}\alpha_n(m)^2J_n^{-1}\right) \tilde{\phi}_{m^*}\\
			{=}\ \alpha_n(m^*)^{-1} {\lVert \alpha_n\rVert_2^2}\,J_n^{-1}\tilde{\phi}_{m^*}\\
			\equiv c_n(m^*)\tilde{\phi}_{m^*}.
		\end{array}
		\]
		
		For example, if  $ {\lVert} \alpha_n {\rVert_2^2}=O(1/n)$ and $\alpha_n(m^*)\sim 1/n$, then $c_n(m^*)=O(1)$. Thus, we obtain the conservative bound
		\[
		{\lVert} f_{\alpha_n}^*-\tilde{f}_{\alpha_n} {\rVert_{1, P_0}}=O_P(c_n(m^*)\, {\lVert} \tilde{\phi}_{m^*} {\rVert_{1,P_0}}) {.}\]
		
		\begin{theorem}
			{\bf Definitions;}
			Let $f_{\alpha_n}=\sum_{j\in {\cal E}(J_n)}\alpha_n(j)/ {\lVert \alpha_n\rVert_2}\, \tilde{\phi}_j$. 
			Note that $ {\lVert} f_{\alpha_n} {\rVert_2}=1$ is normalized to have norm 1 w.r.t.\ the standard inner product. 
			We also define $f^*_{\alpha_n}=f_{\alpha_n}/J_n^{1/2}= {J_n^{-1/2}\lVert \alpha_n\rVert_2^{-1}}\sum_{j\in {\cal E}(J_n)}{\alpha_n}(j)\tilde{\phi}_j$ so that $ {\lVert} f^*_{\alpha_n} {\rVert_2}=J_n^{-1/2}$. 
			Consider an $f^*\in D_M({\cal E}(J_n))$ satisfying $ {\lVert} \alpha(f^*) {\rVert_2}=O(n^{-1/2})$, which implies $ {\lVert}\beta(\alpha(f^*)) {\rVert_1}=O(1)$.
			Let $\theta_n$ be an PC-HAR constraining the $L_2$-norm  by $C_n$, and let ${\cal E}(J_n)$ be the set of basis functions $\tilde{\phi}_j$ with non-zero coefficients. Then, $\theta_n=\arg\min_{\theta\in D_{2,C_n}({\cal E}(J_n))}P_n L(\theta)$. Due to latter,  $\theta_n$ solves the following score equations:
			\[
			P_n \frac{d}{d\theta_n}L(\theta_n)\left(\sum_{m\in {\cal E}(J_n)}\alpha(m)\tilde{\phi}_m\right )\mbox{ for $\alpha\perp {\alpha_n}$.}\]
			Let's endow $D({\cal E}(J_n))$ with an inner product $\langle f_1,f_2\rangle_n=\sum_{m\in {\cal E}(J_n)}\alpha(f_1)(m)\alpha(f_2)(m)$. Let $D_n({\cal E}(J_n))\equiv \{f\in D({\cal E}(J_n)): f\perp f_{\alpha_n}\}$ which is a  $J_n-1$-dimensional subspace of $D({\cal E}(J_n))$  {.}  Let $\tilde{f}_{\alpha_n}=\arg\min_{f\in  {D_n({\cal E}(J_n))}}P_0(f_{\alpha_n}^*-f)^2$ {.}
			
			{\bf Assumptions:}
			Covering number of $\{d/d\theta_nL(\theta_n)(f): f\in D_M({\cal E}(J_n)),\theta_n\in D_M({\cal E}(J_n))\}$ is of same order as covering number of $D_M([0,1]^d)$;
			%We have\begin{eqnarray*}\mid P_n d/dQ_nL(Q_n)(f^*)\mid &\leq&  \mid (P_n-P_0)d/dQ_nL(Q_n)(f_{\ell_n}^*-\tilde{f}_{\ell_n})\mid \\&&+\mid P_0 (d/dQ_nL(Q_n)-d/dQ_0L(Q_0))(f_{\ell_n}-\tilde{f}_{\ell_n}) \mid .\end{eqnarray*}
			%Under the assumption that $\pl d/dQ_nL(Q_n)( f_{\ell_n}^*-\tilde{f}_{\ell_n})\pl_{P_0}=o_P(1)$, it follows that the empirical process term is $o_P(n^{-1/2})$. For the second term we assume that it can be bounded by $\pl Q_n-Q_0\pl_{\infty}\pl f_{\ell_n}^*-\tilde{f}_{\ell_n}\pl_{1,P_0}$.Alternatively we bound it by $d_0^{1/2}(Q_n,Q_0)\pl f_{\ell_n}^*-\tilde{f}_{\ell_n}\pl_{P_0}$.  For example, for first order HAL, we would have$d_0^{1/2}(Q_n,Q_0)=O^+_P(n^{-2/5})$ so that we only need $\pl f_{\ell_n}^*-\tilde{f}_{\ell_n}\pl_{P_0}=O_P(n^{-1/10-\delta})$ for some $\delta>0$.Finally, we have \[\pl f_{\ell_n}^*-\tilde{f}_{\ell_n}\pl_{P_0}=O_P\left(\min_{j\in {\cal R}_n}\pl \phi_j\pl_{1,P_0}\right).\]
			$ {\lVert} \theta_n-\psi_0 {\rVert_{\infty}} {\lVert} f_{\alpha_n}^*-\tilde{f}_{\alpha_n} {\rVert_{1,P_0}}=o_P(n^{-1/2})$ or $d_0^{1/2}(\theta_n,\psi_0) {\lVert} f_{\alpha_n}^*-\tilde{f}_{\alpha_n} {\rVert_{P_0}}=o_P(n^{-1/2})$. 
			A conservative bound $ {\lVert} f_{\alpha_n}^*-\tilde{f}_{\alpha_n} {\rVert_{1,P_0}}$ for $ {\lVert} f_{\alpha_n}^*-\tilde{f}_{\alpha_n} {\rVert}$ is given by 
			$\min_{m\in {\cal E}(J_n)}c_n(m) {\lVert} \tilde{\phi}_m {\rVert}$, where $c_n(m)\sim \alpha_n(m)^{-1} {\lVert \alpha_n\rVert_2^2}$.  
			
			{\bf Conclusion:} 
			We have $P_n d/d\theta_n L(\theta_n)(f^*)=P_n d/d\theta_nL(\theta_n)(f_{\alpha_n}^*-\tilde{f}_{\alpha_n})$. Under the above conditions, we have
			Then,
			\[
			\sup_{f^*\in D_M({\cal E}(J_n)), {\lVert} f^* {\rVert_2}\leq J_n^{-1/2}}\mid P_n \frac{d}{d\theta_n}L(\theta_n)( f^*)\mid =o_P(n^{-1/2}).\]
			
		\end{theorem}
		\section{Proof of asymptotic efficiency of plug-in PC-HA}

		Consider one of our  PC-HA  estimators $\theta_n=\theta_{n,\alpha_n}\in D^{(k)}_M({\cal E}_n)$.
		To analyze the plug-in estimator $\Phi^F(\theta_{n,\alpha_n})$ one needs that
		$P_n S_{\theta_{n,\alpha_n}}(f_{\theta_n,f_n})=o_P(n^{-1/2})$ for some estimator $f_n$ of $f_0$ or one could select $f_n=f_0$.
		Equivalently, we need that $P_n D^*_{\theta_n,f_n}=o_P(n^{-1/2})$ or $P_n D^*_{\theta_n,f_0}=o_P(n^{-1/2})$.
		Above we provided regularity conditions under which $\theta_n$  solves scores $P_n S_{\theta_{n,\alpha_n}}(f)=o_P(n^{-1/2})$ uniformly in 
		$f\in D^{(k)}_M({\cal E}_n)$ for finite $M<\infty$ with $ {\lVert} f {\rVert_2}=O(n^{-1/2})$. Given that $D^{(k)}_M({\cal E}_n)$ approximates 
		$D^{(k)}_M([0,1]^d)$ at a rate $r(n)=O^+(n^{-k^*/(2k^*+1)})$, we should be able to determine a sequence
		$\tilde{f}_n\in D^{(k)}_M({\cal E}(J_n))$ so that $ {\lVert} \tilde{f}_n-f_{\theta_n,f_n} {\rVert}=O_P^+(n^{-k^*/(2k^*+1)})$ while
		$P_n S_{\theta_n}(\tilde{f}_n)=o_P(n^{-1/2})$. 
		Then, it remains to establish that $P_n S_{\theta_n}(\tilde{f}_n-f_{\theta_n,f_n})=o_P(n^{-1/2})$. 
		The latter is shown by noticing that (we can also set $f_n=f_0$ below)
		\begin{eqnarray*}
			P_n S_{\theta_n}(\tilde{f}_n-f_{\theta_n,f_n})&=&(P_n-P_0)S_{\theta_n}(\tilde{f}_n-f_n) \\
			&&+P_0\{S_{\theta_n}(\tilde{f}_n-f_n)-S_{\psi_0}(\tilde{f}_n-f_n)\}.
		\end{eqnarray*}
		The first term is an empirical process term that is $o_P(n^{-1/2})$ under a weak consistency condition on $\tilde{f}_n-f_n$, while the second term is a second order difference that can typically be bounded by product of $ {\lVert} \theta_n-\psi_0 {\rVert}$ and 
		$\tilde{f}_n-f_n$. 
		
		\begin{lemma}
			Suppose $D^*_{\Phi(),P}=S_{\Psi(P)}(f_P)$ with $f_P\in D^{(k)}_M([0,1]^d)$ for all $P\in {\cal M}$.
			Suppose we have $\sup_{f\in D^{(k)}({\cal E}(J_n)) {,\ \lVert} f {\rVert_2^2\leq 1/n}}P_n S_{\theta_n}(f)=o_P(n^{-1/2})$. 
			Let $\tilde{f}_{0,n}=\arg\min_{f\in D^{(k)}({\cal E}(J_n))}d_0(f,f_{P_0})$ be the projection of $f_{P_0}$ onto $D^{(k)}({\cal E}(J_n))$, and assume that
			$ {\lVert} \tilde{f}_{0,n} {\rVert_2^2}=O_P(n^{-1})$. 
			Suppose that $\tilde{f}_{0,n}$ converges to $f_0$ so that $P_n S_{\theta_n}(\tilde{f}_{0,n}-f_0)=o_P(n^{-1/2})$, which follows if $(P_n-P_0)S_{\theta_n}(\tilde{f}_{0,n}-f_0)=o_P(n^{-1/2})$
			and $P_0(S_{\theta_n}-S_{\psi_0})(\tilde{f}_{0,n}-f_0)=o_P(n^{-1/2})$.
			Then, it follows that $P_n S_{\theta_n}(f_0)=o_P(n^{-1/2})$ and thereby $P_n D^*_{\theta_n,f_0}=o_P(n^{-1/2})$.
		\end{lemma}
		
		It could be that one needs to undersmooth the regularization parameter $M$ in $ {\lVert \alpha\rVert_1}\leq M$ for the PC-HAL relative to its cross-validation selector $M_{n,cv}$ so that the linear span of the $\{\tilde{\phi}_m: \alpha_n(m)\not =0\}$ is rich enough to approximate the $f_{\theta_n,G_n}$ or $f_{\theta_n,G_0}$. It appears that when we use PC-HAR or PC-HAGL with $D^{(k)}({\cal E}_n)$, then for each constraint $M$ on $ {\lVert\alpha\rVert_2}$ or $ {\lVert\beta(\alpha)\rVert_1}$, we end up using all basis functions $\{\tilde{\phi}_m: m=1,\ldots,n\}$ so that the span of scores solved is non-sensitive to the bound $M$ in $ {\lVert \alpha\rVert_2}<M$ or $ {\lVert \beta(\alpha)\rVert_1}<M$.
		However, if one implements these PC-HA estimators with two tuning parameters, one that selects top $m_n$ $\tilde{\phi}_1,\ldots,\tilde{\phi}_{m_n}$ and one that controls the norm of $\alpha$ in the resulting model $D^{(k)}({\cal E}(J_n))$, then some undersmoothing for the first tuning parameter $m_n$ might still be needed so that its choice is not just adaptive to fitting $\psi_0$ but also makes sure allows fitting of $f_{\psi_0,G_0}$.

		Once we have $P_n D^*_{\theta_n,f_n}=o_P(n^{-1/2})$ or $P_n D^*_{\theta_n,f_0}=o_P(n^{-1/2})$, the analysis of $\Phi(\theta_n)$ proceeds as in any TMLE or plug-in HAL analysis. 
		That is, we have $\Phi^F(\theta_n)-\Phi^F(\psi_0)=(P_n-P_0)D^*_{\theta_n,f_n}+o_P(n^{-1/2})+R_{\Phi(),0}(\theta_n,f_n,\psi_0,f_0)$, where $f_n=f_0$ if we have $P_n D^*_{\theta_n,f_0}=o_P(n^{-1/2})$.
		Given the rates $\theta_n-\psi_0$ and $f_n-f_0$, one has to prove that $R_{\Phi(),0}(\theta_n,f_n,\psi_0,f_0)=o_P(n^{-1/2})$, which generally quadratic in the rates  of $\theta_n-\psi_0=O^+(n^{-k^*/(2k^*+1)})$ and $f_n-f_0$.
		In addition, we need that $D^*_{\theta_n,f_n}$ falls in a Donsker class with probability tending to 1, which follows if $\theta_n\in D^{(k)}_M({\cal E}_n)$ and $f_n\in D^{(k)}_M([0,1]^d)$  and that the covering number of the image functions under $(\theta,f)\rightarrow D^*_{\theta,f}$ is bounded accordingly.
		One should also easily establish $P_0\{D^*_{\theta_n,f_n}-D^*_{\psi_0,f_0}\}^2=o_P(1)$. This then proves
		\[
		\Phi^F(\theta_n)-\Phi^F(\psi_0)=P_n D^*_{\Phi(),P_0}+o_P(n^{-1/2}),\]
		where $D^*_{\Phi(),P_0}=S_{\psi_0}(f_0)$
		and thereby establishes asymptotic normality and  efficiency of $\Phi^F(\theta_n)$ as estimator of $\Phi^F(\psi_0)$. 
		
		Suppose that the conclusion of the above lemma applies so that $P_n S_{\theta_n}(f_0)=o_P(n^{-1/2})$ with $f_0\in D^{(k)}_M([0,1]^d)$. and 
		$D^*_{\Phi(),P}=D^*_{\theta(P),f_P}$ so that we have $P_n  D^*_{\theta_n,f_0}=o_P(n^{-1/2})$.
		Assume also that $S_{\theta_n}(f_0)$ is an element of a Donsker class with probability tending to 1 (presumably with covering number of similar order as the one for $D^{(k)}_M([0,1]^d)$; $R(\theta_n,f_0,\theta_0,f_0)=o_P(n^{-1/2})$; and $P_0\{S_{\theta_n}(f_0)-S_{\psi_0}(f_0)\}^2\xrightarrow{P} 0$. 
		Then, $\Phi^F(\theta_n)$ is an efficient estimator of $\Phi^F(\psi_0)$:
		\[
		\Phi^F(\theta_n)-\Phi^F(\theta_0)=P_n D^*_{\Phi(),P_0}+o_P(n^{-1/2}).\]

		\section{Definition of sectional variation norm}\label{app:secv}

		For an index set $s \subseteq \{1,\dots,d\}$, let $x(s) = (x_j)_{j \in s}$ denote the subvector of coordinates in $s$. We define the \textbf{sectioning operator} $x^{(s)}$ as the vector $x$ where all coordinates $j \notin s$ are set to zero:
		\[ x^{(s)} = (x(s), \mathbf{0}_{-s}). \]

		Let $\mathcal{S}^{k+1}(d)$ be the collection of nested $(k+1)$-tuples $\bar{s}(k+1) = (s_1, \dots, s_{k+1})$, such that $s_{k+1} \subseteq s_k \subseteq \dots \subseteq s_1 \subseteq \{1, \dots, d\}$. 
		
		\begin{definition}[$k$-th order recursive derivative]
			For a function $Q: [0,1]^d \to \mathbb{R}$, the recursive derivatives $Q^{(j)}_{\bar{s}(j)}$ are defined as follows:
			
			\begin{enumerate}
				\item \textbf{Initialization ($j=1$):} 
				Define the derivative of the $s_1$-section of $Q$ as:
				\[ Q^{(1)}_{s_1}(x(s_1)) = \left( \prod_{i \in s_1} \frac{\partial}{\partial x_i} \right) Q(x^{(s_1)}). \]
				
				\item \textbf{Recursion ($j = 1, \dots, k$):} 
				Given $Q^{(j)}_{\bar{s}(j)}$, define the next derivative by sectioning the variables in $s_j \setminus s_{j+1}$ to zero and differentiating with respect to the remaining variables $s_{j+1}$:
				\[ Q^{(j+1)}_{\bar{s}(j+1)}(x(s_{j+1})) = \left( \prod_{i \in s_{j+1}} \frac{\partial}{\partial x_i} \right) Q^{(j)}_{\bar{s}(j)}(x^{(s_{j+1})}). \]
				
				\item \textbf{Termination (constant terms):} 
				If a subset $s_m = \emptyset$ is reached, the derivative becomes a constant equal to the previous derivative evaluated at the origin:
				\[ Q^{(m)}_{\bar{s}(m)} = Q^{(m-1)}_{\bar{s}(m-1)}(\mathbf{0}). \]
			\end{enumerate}
		\end{definition}
		
		\subsection*{Concrete example}
		
		Consider a function $Q(x_1, x_2)$ and a nested sequence for $k=1$ defined by $s_1 = \{1, 2\}$ and $s_2 = \{1\}$. We follow the recursive steps to find the second-order derivative $Q^{(2)}_{\bar{s}(2)}$:
		
		\begin{description}
			\item[Step 1: Initialization ($j=1$)] \hfill \\
			We evaluate the mixed partial derivative of the section $s_1$. Since $s_1$ contains all variables:
			\[
			Q^{(1)}_{s_1}(x_1, x_2) = \frac{\partial^2 Q(x_1, x_2)}{\partial x_1 \partial x_2}.
			\]
			
			\item[Step 2: Recursion ($j=1 \to 2$)] \hfill \\
			We now take $Q^{(1)}_{s_1}$ and apply the sectioning operator for $s_2$. Since $s_2 = \{1\}$, we set $x_2 = 0$ and differentiate with respect to $x_1$. Note that this results in $x_1$ being differentiated a second time:
			\[
			Q^{(2)}_{\bar{s}(2)}(x_1) = \frac{\partial}{\partial x_1} \left[ Q^{(1)}_{s_1}(x_1, 0) \right] = \frac{\partial}{\partial x_1} \left[ \left. \frac{\partial^2 Q(x_1, x_2)}{\partial x_1 \partial x_2} \right|_{x_2=0} \right].
			\]
			
			\item[Result] \hfill \\
			This specific nested sequence targets the variation of the interaction term specifically along the $x_1$ axis, effectively involving a 3rd-order derivative $\partial_1^2 \partial_2$. This higher-order behavior is necessary to support "hybrid" basis functions (e.g., linear in $x_1$, jump in $x_2$).
		\end{description}
		
		\begin{definition}[$k$-th order sectional variation norm]
			Let $k \in \{0, 1, \dots\}$ be the order of smoothness. The $k$-th order sectional variation norm $\|Q\|_{v,k}^*$ is the sum of the variation norms of all recursive derivatives across all possible nested sequences:
			
			\[
			\|Q\|_{v,k}^* = \sum_{\bar{s}(k+1) \in \mathcal{S}^{k+1}(d)} \|Q^{(k)}_{\bar{s}(k+1)}\|_v
			\]
			
			where the term $\| \cdot \|_v$ is defined based on the nesting structure of $\bar{s}(k+1) = (s_1, \dots, s_{k+1})$:
			
			\begin{enumerate}
				\item \textbf{Continuous variation (non-empty $s_{k+1}$):} 
				If the final set $s_{k+1} \neq \emptyset$, the norm is the Total Variation of the $k$-th derivative on the unit cube of dimension $|s_{k+1}|$:
				\[
				\|Q^{(k)}_{\bar{s}(k+1)}\|_v = \int_{[0,1]^{|s_{k+1}|}} \left| d Q^{(k)}_{\bar{s}(k+1)}(u_{s_{k+1}}) \right|.
				\]
				
				\item \textbf{Discrete variation (empty $s_{k+1}$):} 
				If $s_{k+1} = \emptyset$, let $m$ be the largest index such that $s_m \neq \emptyset$ (i.e., the step before the sequence terminated). The norm is the absolute value of the constant term at that step:
				\[
				\|Q^{(k)}_{\bar{s}(k+1)}\|_v = | Q^{(m)}_{\bar{s}(m)}(\mathbf{0}) |.
				\]
				\textit{(Note: If $s_1 = \emptyset$, then $m=0$ and this reduces to $|Q(\mathbf{0})|$.)}
			\end{enumerate}
		\end{definition}
		
		\subsection*{Concrete example: $d=2, k=1$}
		
		To illustrate the $k$-th order norm, we must sum the variation over \textbf{all possible nested sequences}. Each sequence $\bar{s}$ generates a distinct recursive derivative $Q^{(1)}_{\bar{s}}$.
		
		Consider the function:
		\[ Q(x_1, x_2) = \underbrace{3 x_1^2 x_2}_{\text{Interaction}} + \underbrace{2 x_1^2}_{\text{Main Effect}} + \underbrace{5}_{\text{Intercept}} \]
		
		We calculate the components for three distinct sequences.
		
		\begin{enumerate}
			\item \textbf{Sequence A: $\bar{s}_a = (\{1, 2\}, \{1, 2\})$ } \\
			This sequence keeps all variables active. We differentiate the full function.
			\begin{itemize}
				\item Derivative: $Q^{(1)}_{\bar{s}_a} = \frac{\partial^2}{\partial x_1 \partial x_2} Q(x_1, x_2) = 6 x_1$.
				\item Variation: 
				\[ V_a = \| Q^{(1)}_{\bar{s}_a} \|_v = \int_0^1 \int_0^1 |6| \, dx_1 dx_2 = \mathbf{6}. \]
			\end{itemize}
			
			\item \textbf{Sequence B: $\bar{s}_b = (\{1\}, \{1\})$} \\
			This sequence sections $x_2$ to zero \textit{before} differentiation.
			\begin{itemize}
				\item Sectioning: $Q(x^{(s_1)}) = Q(x_1, 0) = 2 x_1^2 + 5$. (Note: The interaction $3x_1^2x_2$ is gone).
				\item Derivative: $Q^{(1)}_{\bar{s}_b} = \frac{d}{dx_1} (2 x_1^2 + 5) = 4 x_1$.
				\item Variation: 
				\[ V_b = \| Q^{(1)}_{\bar{s}_b} \|_v = \int_0^1 |4| \, dx_1 = \mathbf{4}. \]
			\end{itemize}
			
			\item \textbf{Sequence C: $\bar{s}_c = (\emptyset, \emptyset)$ } \\
			This sequence sections everything to zero.
			\begin{itemize}
				\item Value: $Q(\mathbf{0}) = 5$.
				\item Variation: 
				$V_c = \| Q^{(0)}_{\bar{s}_c} \|_v = |5| = \mathbf{5}.$
			\end{itemize}
		\end{enumerate}

		The total $k=1$ sectional variation norm is the sum of the variation of every possible sequence:
		\[
		\| Q \|_{v,1}^* = \sum_{\bar{s}} \| Q^{(1)}_{\bar{s}} \|_v = V_a + V_b + V_c + \dots = 6 + 4 + 5 + \dots = \mathbf{15}.
		\]
		\textit{(Note: All the $x_2$ involved sequences contribute zeros to the sum)}

		\section{Minor proofs}
		\subsection{Proof of Proposition 1}
		
		$$
		\begin{aligned}
			\left.\frac{d}{d\delta}\right|_{\delta=0} R_n(\alpha_\delta^h)
			&= \sum_m \left.\frac{\partial R_n}{\partial \alpha(m)}\right|_{\alpha}\cdot
			\frac{d}{d\delta}\Big|_{\delta=0}\alpha_\delta^h(m)\\
			&= \sum_m \frac{\partial R_n}{\partial \alpha(m)}\cdot \big(h(m)\alpha(m)\big)\\
			&= \sum_m h(m)\,\underbrace{\big(\alpha(m)\,\partial R_n/\partial\alpha(m)\big)}_{=D(\alpha)(m)}.
		\end{aligned}
		$$
		
		\subsection{Proof of Proposition \ref{prop:grad}}
		
		(MSE) Let 
		\[
		f_\alpha(x)=\sum_m \alpha(m)\tilde\phi_m(x),
		\qquad 
		r = Y - f_\alpha.
		\]
		The empirical MSE risk is
		\[
		R_n(\alpha)
		=
		P_n\!\left( Y - f_\alpha \right)^2
		=
		P_n(r^2).
		\]
		We compute the directional derivative
		\[
		D(\alpha)(m)
		=
		\frac{\partial}{\partial \alpha(m)} R_n(\alpha).
		\]
		
		We begin by differentiating the pointwise loss. For a single observation,
		\[
		\ell(r)=r^2,
		\qquad
		\ell'(r)=2r.
		\]
		Since 
		\[
		r = Y - f_\alpha,
		\qquad
		\frac{\partial r}{\partial \alpha(m)}
		= -\,\tilde\phi_m,
		\]
		the chain rule gives
		\[
		\frac{\partial}{\partial \alpha(m)}\,\ell(r)
		=
		\ell'(r)\,
		\frac{\partial r}{\partial \alpha(m)}
		=
		2r\,(-\tilde\phi_m)
		=
		-2(Y - f_\alpha)\tilde\phi_m.
		\]
		
		Applying $P_n$ yields the gradient component
		\[
		D(\alpha)(m)
		=
		P_n\!\left[
		-2\,(Y - f_\alpha)\,\tilde\phi_m
		\right].
		\]
		Equivalently,
		\[
		D(\alpha)(m)
		=
		2\,P_n\!\left[
		\tilde\phi_m\,(f_\alpha - Y)
		\right].
		\]
		Substituting $f_\alpha=\sum_m \alpha(m)\tilde\phi_j$ completes the derivation.
		
		(Logistic) Let
		$f_\alpha(x)=\sum_j \alpha(j)\tilde\phi_j(x), z=-Y f_\alpha$.
		The empirical logistic risk is
		$R_n(\alpha)=P_n\ell(z)=P_n\log\big(1+e^{z}\big)$.
		Write $\ell(z)=\log(1+e^{z})$ so that
		$\ell'(z)=\frac{e^{z}}{1+e^{z}}=\sigma(z)$,
		where $\sigma$ denotes the logistic function. Since $z=-Y f_\alpha$ we have
		$\frac{\partial z}{\partial \alpha(m)}=-Y\,\tilde\phi_m$.
		By the chain rule,
		$\frac{\partial}{\partial\alpha(m)}\ell(z)
		=\ell'(z)\cdot\frac{\partial z}{\partial\alpha(m)}
		=\sigma(z)\,(-Y)\tilde\phi_m$.
		Averaging over the empirical distribution gives the ordinary gradient coordinate
		$\frac{\partial R_n(\alpha)}{\partial\alpha(m)}
		= P_n\!\big[\sigma(-Y f_\alpha)\,(-Y)\tilde\phi_m\big]
		= -\,P_n\!\big[Y\,\tilde\phi_m\,\sigma(-Y f_\alpha)\big]$.

		\subsection{Proof of Lemma \ref{lem:sub}}
		We can show $ {\lVert \beta(\alpha)\rVert_2=\lVert \alpha\rVert_2}$ as follows:
		\begin{eqnarray*}
			{\lVert \beta(\alpha_{n,2})\rVert_2^2}&=&\sum_j \beta(\alpha_{n,2})^2(j)\\
			&=&\sum_j\left\{\sum_m \alpha_{n,2}(m)E_N^n(j,m)\right\}^2\\
			&=&\sum_j \sum_{m_1,m_2}\alpha_{n,2}(m_1)\alpha_{n,2}(m_2)E_N^n(j,m_1)E_N^n(j,m_2)\\
			&=&\sum_{m_1,m_2}\alpha_{n,2}(k_1)\alpha_{n,2}(k_2)\sum_j E_N^n(j,m_1)E_N^n(j,m_2)\\
			&=&\sum_{m_1,m_2}\alpha_{n,2}(m_1)\alpha_{n,2}(m_2)I(m_1=m_2)\\
			&=&\sum_m \alpha_{n,2}(m)^2\\
			&=& {\lVert \alpha_{n,2}\rVert_2^2},
		\end{eqnarray*}
		where we used that the eigenvectors $E_N^n(\cdot,m)$ across $m$ are orthonormal.  This result only depends on our chosen standard inner product for $\openr^N$, and applies for different choices of inner products in $\openr^n$. 
		
		Let's now prove the claimed bounding of $ {\lVert \beta(\alpha)\rVert_1=O(\lVert \alpha\rVert_1)}$.  We have \[
		{\lVert \beta(\alpha)\rVert_2^2=\lVert \alpha\rVert_2^2}=\sum_{l=1}^n \alpha(l)^2.\]
		We also know that if $ {\lVert \beta(\alpha)\rVert_2^2}=O(1/N)$, then
		\[
		{\lVert \beta(\alpha)\rVert_1^2}\leq N  {\lVert \beta(\alpha)\rVert_2^2}=O(1).\]
		Thus, to control $ {\lVert\beta(\alpha)\rVert_1}$ it suffices to control $\sum_{l=1}^n \alpha(l)^2=O(1/N)$. 
		In our typical initial working model we have $N\sim n$, so that this corresponds with $\sum_l \alpha(l)^2=O(1/n)$.
		This corresponds with $\alpha(l)\sim 1/n$ and, if we assume $\max_l \mid \alpha(l)\mid =O(1/n)$, then  
		we can conclude that to control $ {\lVert \alpha\rVert_2^2}=O(1/n)$ corresponds with controlling $ {\lVert \alpha\rVert_1}=O(1)$.
		Therefore it is possible to control the $L_1$-norm of $\beta(\alpha)$ by controlling $ {\lVert \alpha\rVert_1} =O(1)$ as long as we do not allow for $\alpha$ whose maximal coefficient is still $O(1/n)$. 
		This proves the statement on bounding $ {\lVert \beta(\alpha)\rVert_1}$ in terms of $ {\lVert \alpha\rVert_1}$.
		
		The other statements are straightforward. $\Box$
		
		\section{Verification of Assumptions \ref{assumption1} and \ref{assumption2} for Common Loss Functions}\label{app:loss}
		
		We verify that the norm equivalence (Assumption \ref{assumption1}) and the i.i.d. sampling condition (Assumption \ref{assumption2}) hold for the standard loss functions discussed in the main text. In all cases, the verification relies on a second-order Taylor expansion of the loss-based dissimilarity $d_0(\psi, \psi_0)$ around the truth $\psi_0$. Under the boundedness constraint $\psi \in D^{(k)}_M([0,1]^d)$, the convexity of the loss ensures that the dissimilarity is equivalent to a weighted $L_2$-norm, which corresponds to the standard $L_2(P_{X,0})$ norm for an appropriate choice of $P_{X,0}$.
		
		\subsection*{Example 1: squared error loss (regression)}
		\textbf{Loss:} $L(\psi)(O) = (Y - \psi(X))^2$. \\
		\textbf{Target:} $\psi_0(X) = E_0[Y|X]$.
		
		\noindent \textit{Verification of Assumption \ref{assumption1}:}
		The exact dissimilarity is given by:
		\begin{align*}
			d_0(\psi, \psi_0) &= P_0 \left[ (Y - \psi(X))^2 - (Y - \psi_0(X))^2 \right] \\
			&= P_0 \left[ (\psi(X) - \psi_0(X))^2 + 2(Y-\psi_0(X))(\psi_0(X)-\psi(X)) \right].
		\end{align*}
		By definition of the conditional mean, $E_0[Y-\psi_0(X) \mid X] = 0$, causing the cross-term to vanish. Thus:
		\[
		d_0(\psi, \psi_0) = \int (\psi(x) - \psi_0(x))^2 dP_0(x) = \|\psi - \psi_0\|_{L_2(P_{X,0})}^2,
		\]
		where $P_{X,0}$ is the marginal distribution of the covariates $X$. The equivalence holds with constants equal to 1.
		
		\noindent \textit{Verification of Assumption \ref{assumption2}:}
		Assumption \ref{assumption2} requires $x^n$ to be an i.i.d. sample from $P_{X,0}$. Since $P_{X,0}$ is the marginal distribution of $X$, setting $x^n = (X_1, \ldots, X_n)$ (the observed covariate values) satisfies this condition.
		
		\subsection*{Example 2: logistic loss (binary classification)}
		\textbf{Loss:} $L(\psi)(O) = -Y\psi(X) + \log(1 + \exp(\psi(X)))$. \\
		\textbf{Target:} $\psi_0(X) = \text{logit}(P_0(Y=1|X))$.
		
		\noindent \textit{Verification of Assumption \ref{assumption1}:}
		Define $p_\psi(x) = (1+\exp(-\psi(x)))^{-1}$. A second-order Taylor expansion of the function $f(u) = -yu + \log(1+e^u)$ around $\psi_0(x)$ yields:
		\[
		L(\psi)(O) \approx L(\psi_0)(O) + \frac{\partial L}{\partial \psi}(\psi_0)(\psi - \psi_0) + \frac{1}{2} \frac{\partial^2 L}{\partial \psi^2}(\psi^*)(\psi - \psi_0)^2.
		\]
		The first-order term has expectation zero at the minimizer $\psi_0$. The second derivative is the variance of the Bernoulli distribution: $\frac{\partial^2 L}{\partial \psi^2} = p_\psi(1-p_\psi)$. Thus:
		\[
		d_0(\psi, \psi_0) \asymp E_0 \left[ p_{\psi^*}(X)(1 - p_{\psi^*}(X)) (\psi(X) - \psi_0(X))^2 \right],
		\]
		where $\psi^*$ lies between $\psi$ and $\psi_0$. Since $\psi \in D^{(k)}_M$, we have $\|\psi\|_\infty \le M$, implying that the probabilities $p_\psi$ are bounded away from 0 and 1. Hence, there exist constants $0 < c < C < \infty$ such that $c \le p_{\psi^*}(1-p_{\psi^*}) \le C$. Therefore:
		\[
		d_0(\psi, \psi_0) \asymp E_0 [(\psi(X) - \psi_0(X))^2] = \|\psi - \psi_0\|_{L_2(P_{X,0})}^2,
		\]
		where $P_{X,0}$ is the marginal distribution of $X$.
		
		\noindent \textit{Verification of Assumption \ref{assumption2}:}
		Similar to the regression case, setting $x^n = (X_1, \ldots, X_n)$ provides the required i.i.d. sample from $P_{X,0}$.
		
		\subsection*{Example 3: log-likelihood with exponential link (density estimation)}
		\textbf{Loss:} $L(\psi)(O) = -\psi(O) + \log \int \exp(\psi(u)) du$. \\
		\textbf{Target:} The true density $p_0(o) = \exp(\psi_0(o)) / \int \exp(\psi_0)$.
		
		\noindent \textit{Verification of Assumption \ref{assumption1}:}
		The loss-based dissimilarity corresponds to the Kullback-Leibler divergence $KL(p_0 \mid\mid p_\psi)$. Under the constraint $\psi \in D^{(k)}_M$, the densities are uniformly bounded away from zero and infinity. For such densities, the KL-divergence is equivalent to the squared Hellinger distance, which is locally equivalent to the squared $L_2$-norm of the log-densities (up to centering). Specifically:
		\[
		d_0(\psi, \psi_0) = KL(p_0 \mid\mid p_\psi) \asymp \int (\psi(o) - \psi_0(o) - c(\psi))^2 dP_0(o),
		\]
		where $c(\psi)$ is a normalization constant. If we restrict the domain to centered functions or consider the quotient space, this is equivalent to:
		\[
		d_0(\psi, \psi_0) \asymp \|\psi - \psi_0\|_{L_2(P_0)}^2.
		\]
		Here, $P_{X,0}$ is identified with the data-generating distribution $P_0$ itself.
		
		\noindent \textit{Verification of Assumption \ref{assumption2}:}
		Since the relevant norm is defined with respect to $P_0$, the observations $x^n = (O_1, \ldots, O_n)$ constitute an i.i.d. sample from $P_{X,0} = P_0$, satisfying the assumption.
		
		\section{Additional simulation plots}\label{supp:simulation_plots}
		
		\subsection{Fast (sinusoidal) target: convergence rates}
		
		\begin{figure}[b!]
			\centering
			\begin{minipage}{\textwidth}
				\centering
				\includegraphics[width=0.3\textwidth]{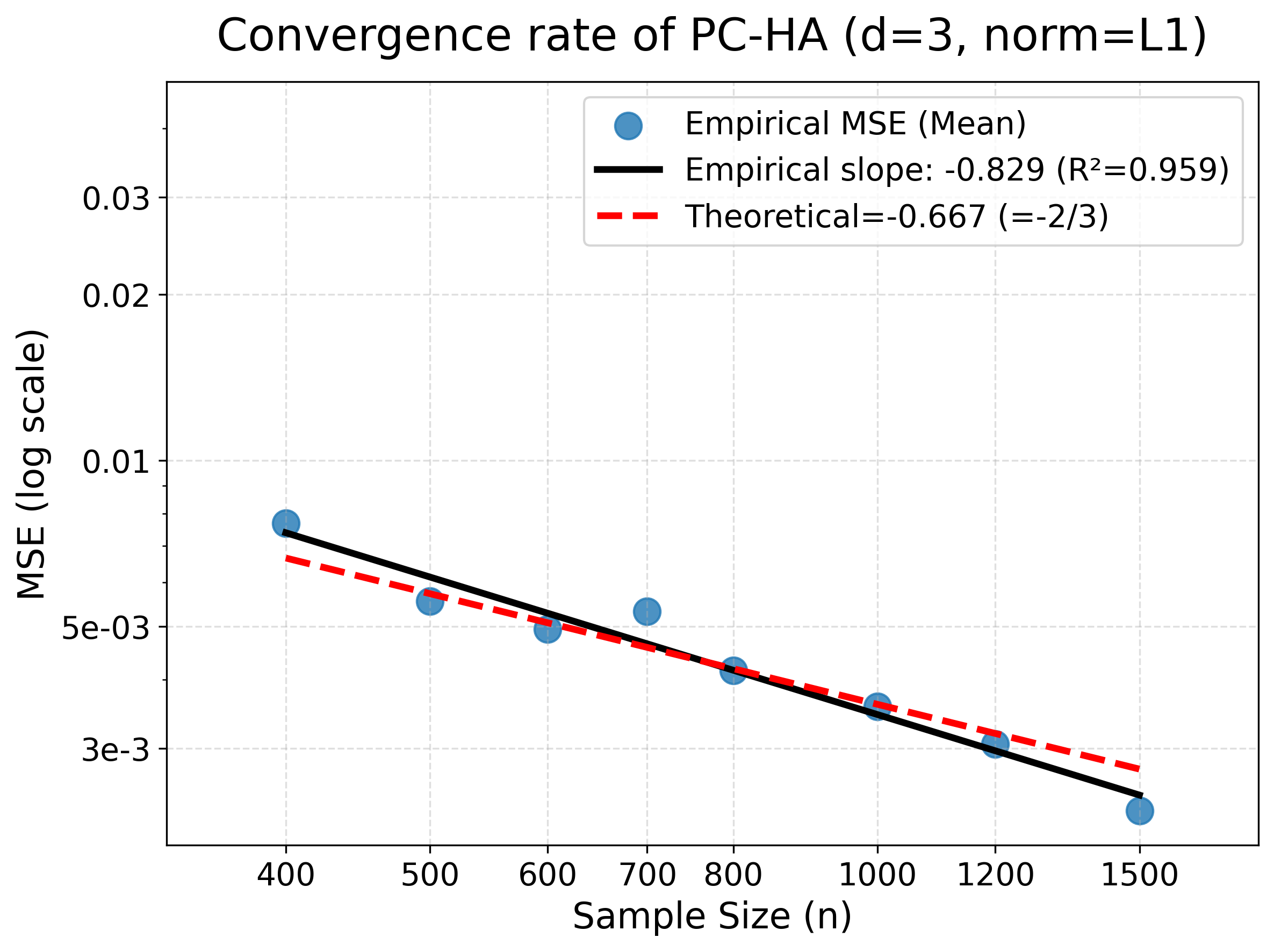}%
				\includegraphics[width=0.3\textwidth]{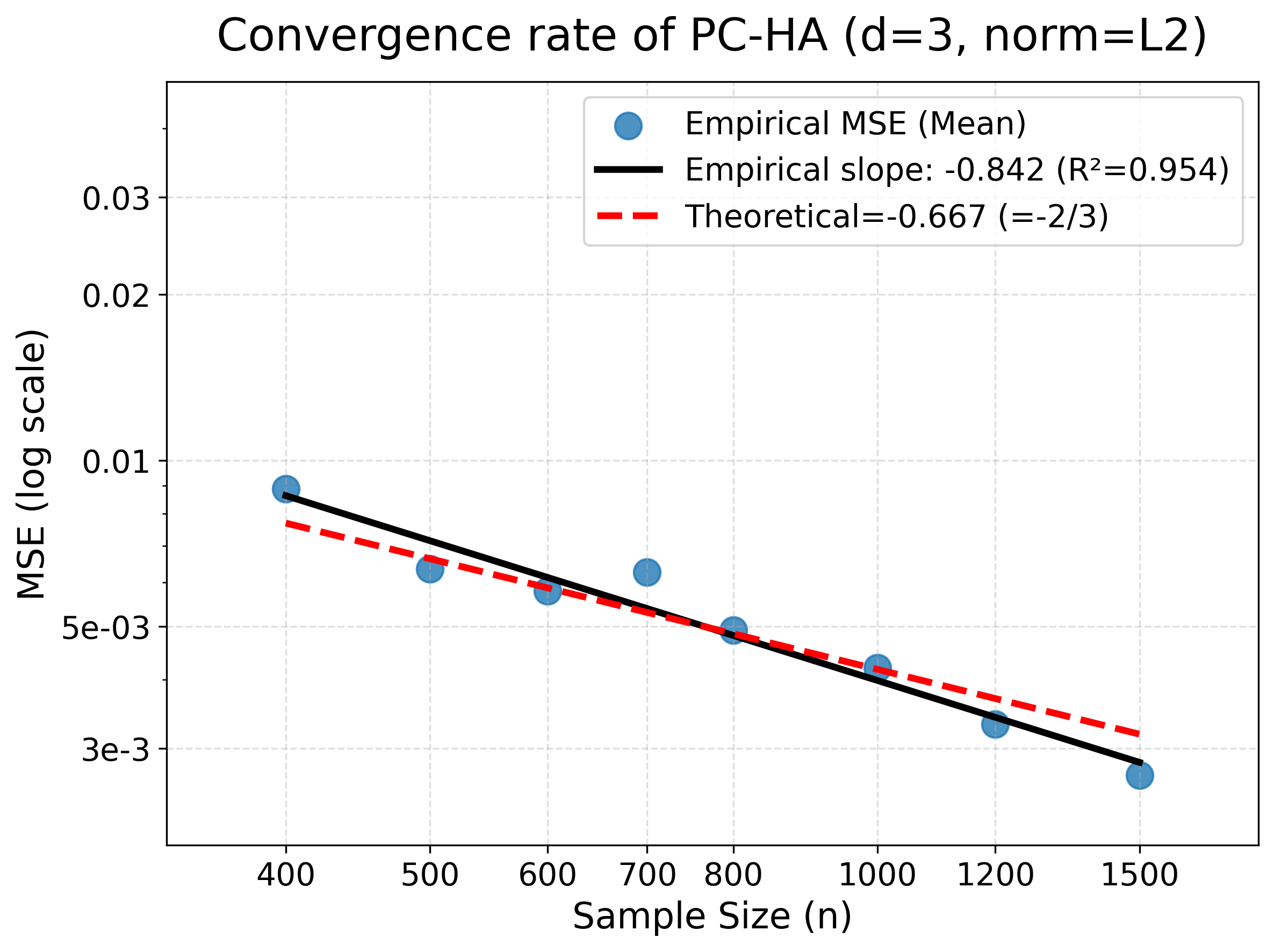}%
				\includegraphics[width=0.3\textwidth]{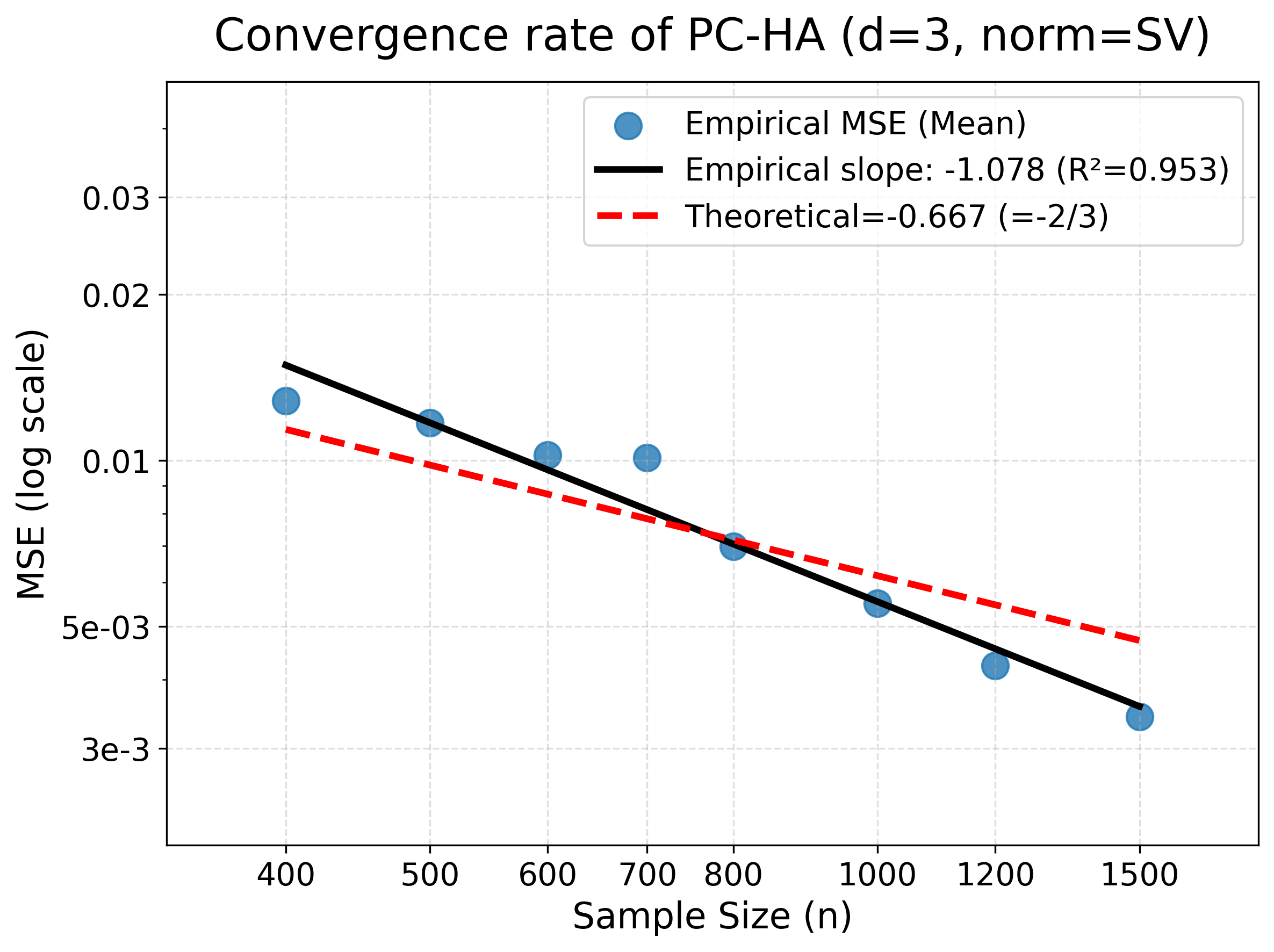}
				\vspace{0.2em}
				\includegraphics[width=0.3\textwidth]{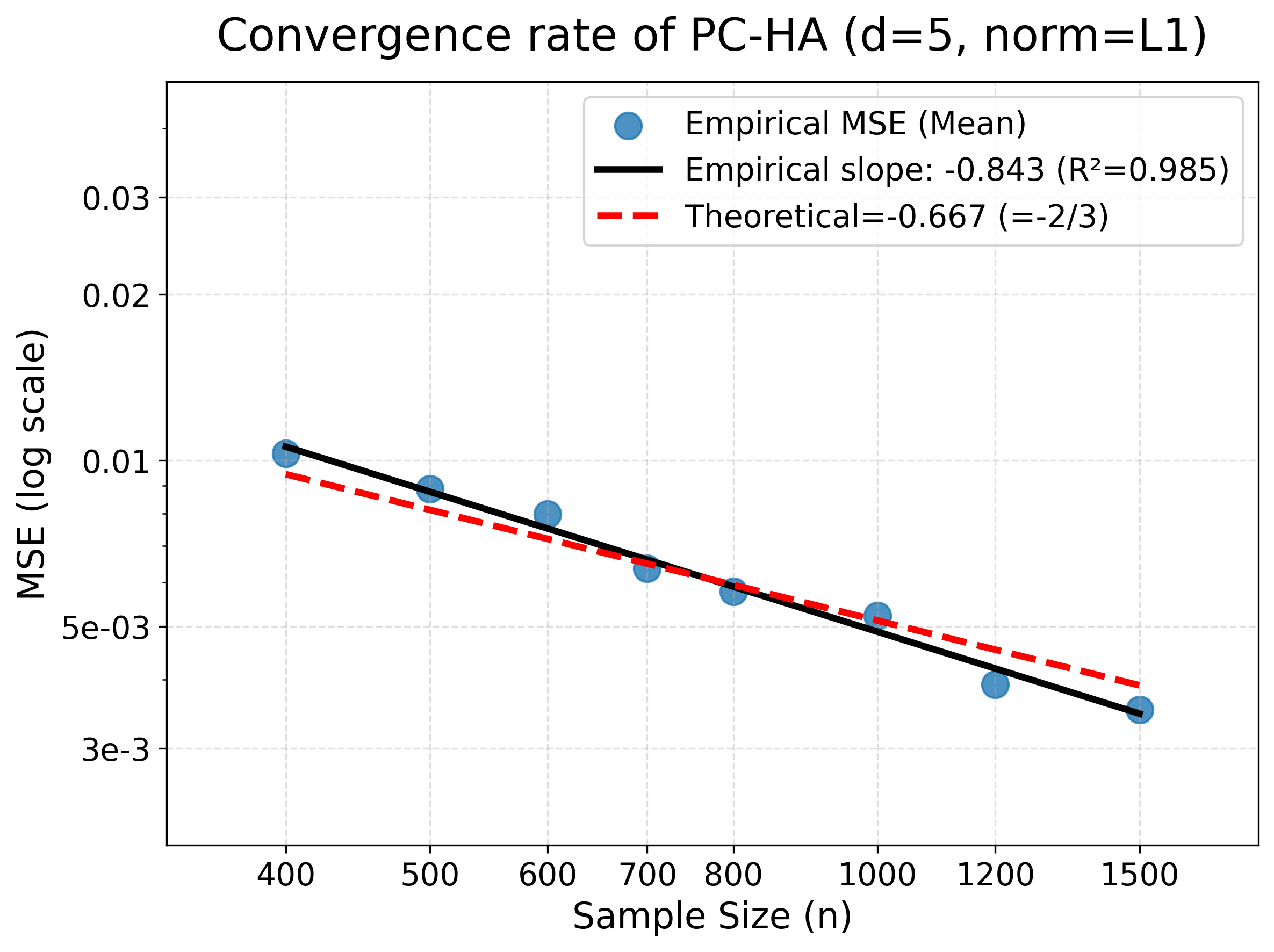}%
				\includegraphics[width=0.3\textwidth]{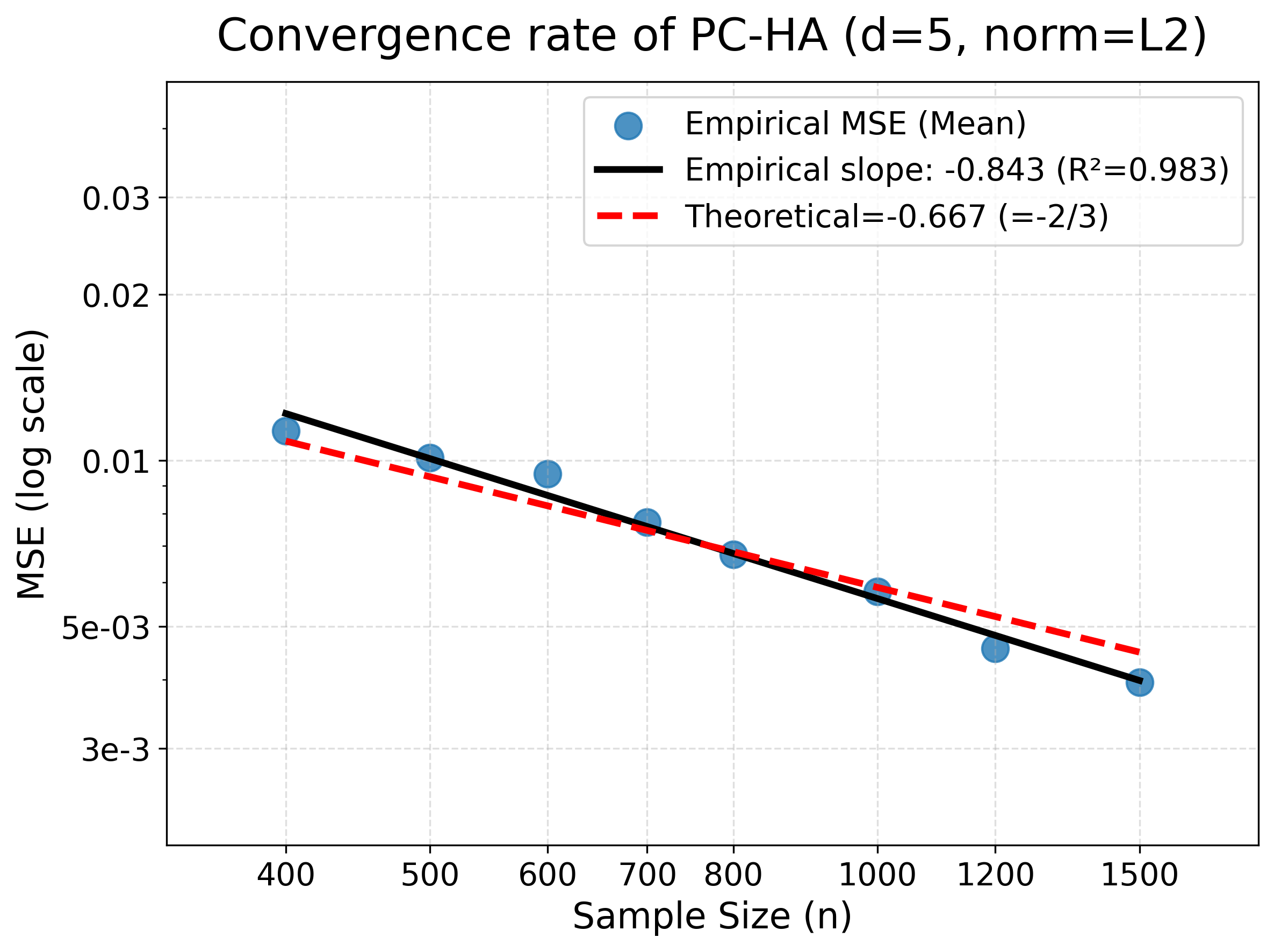}%
				\includegraphics[width=0.3\textwidth]{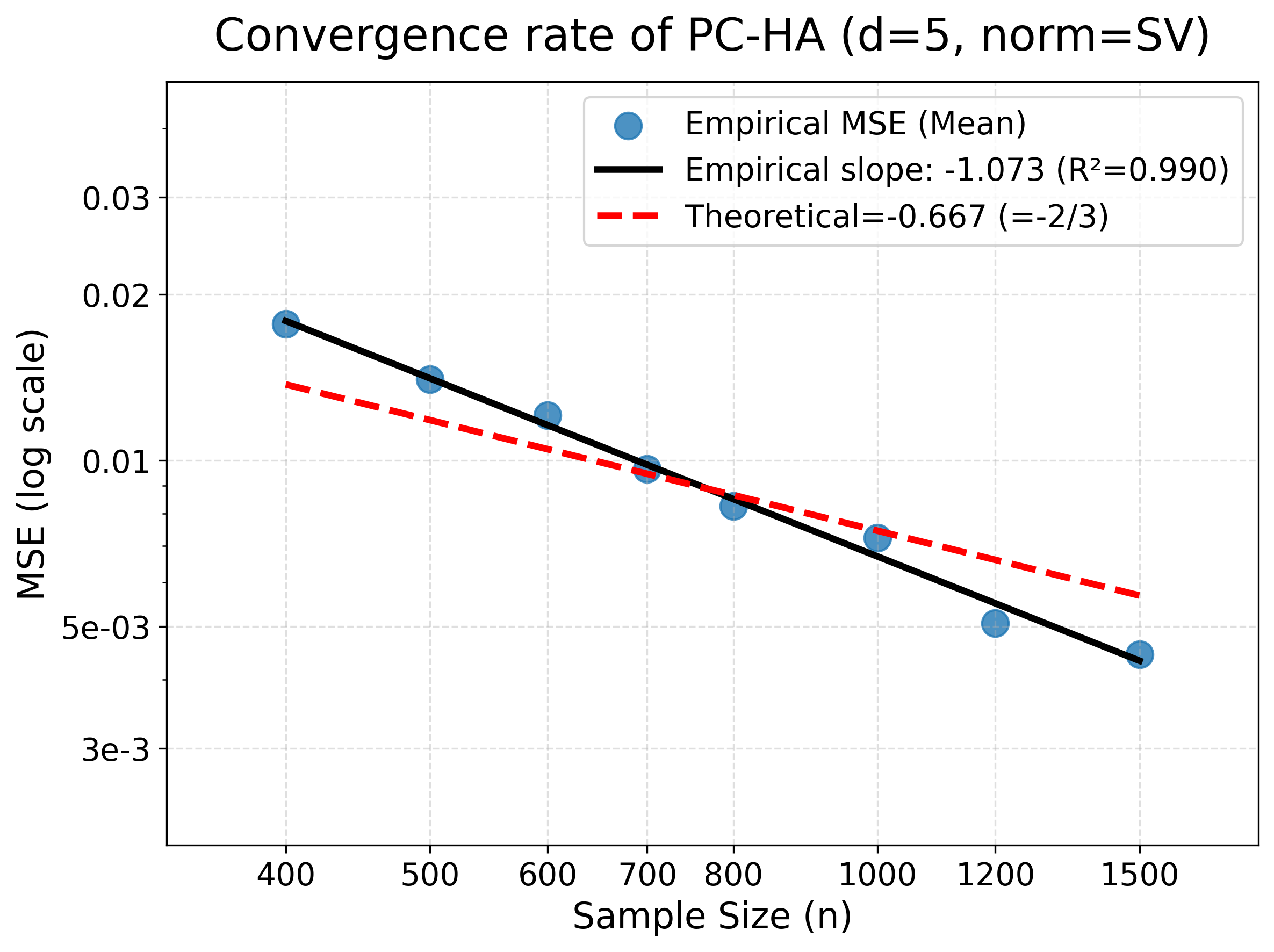}
				\vspace{0.2em}
				\includegraphics[width=0.3\textwidth]{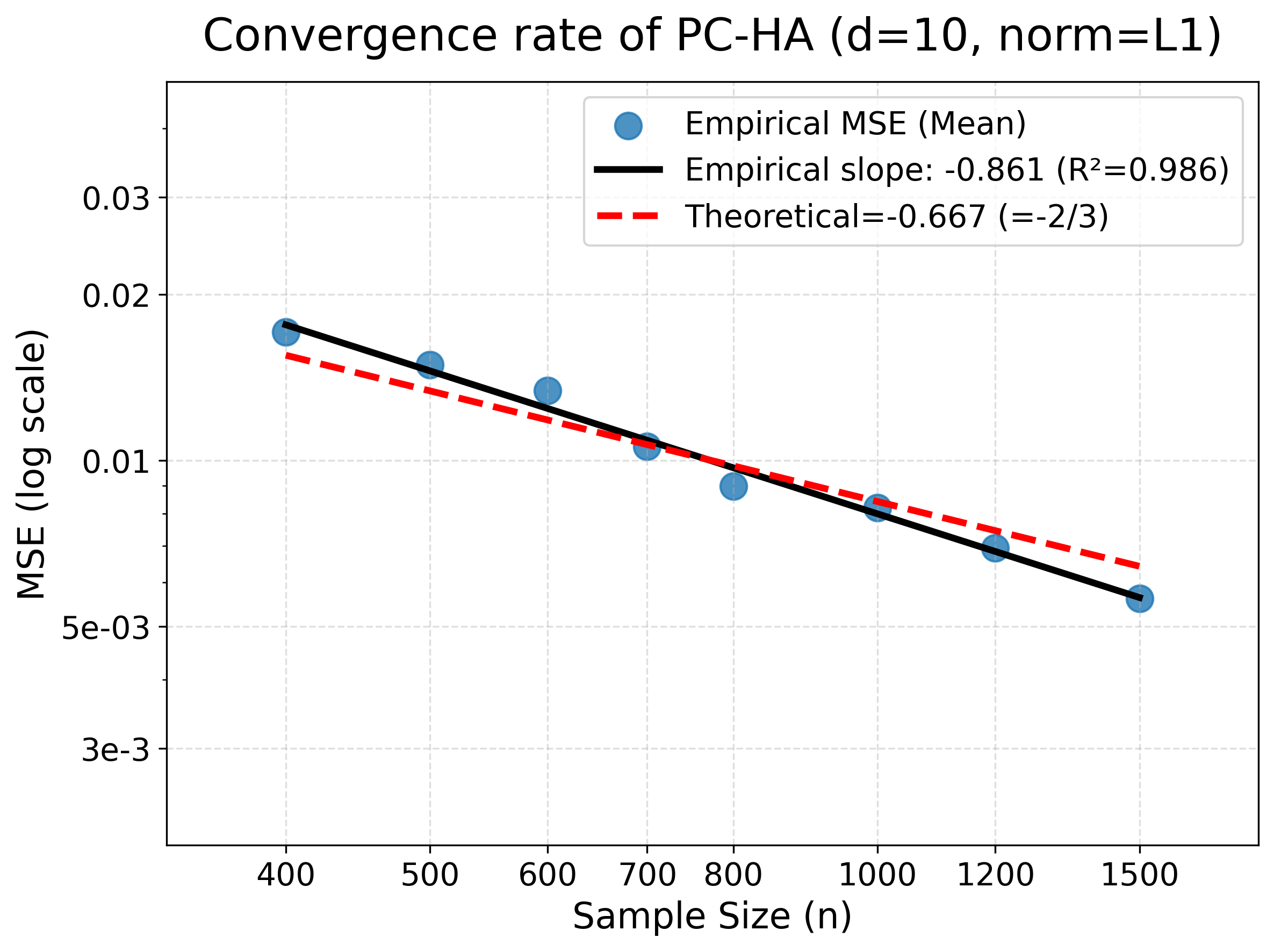}%
				\includegraphics[width=0.3\textwidth]{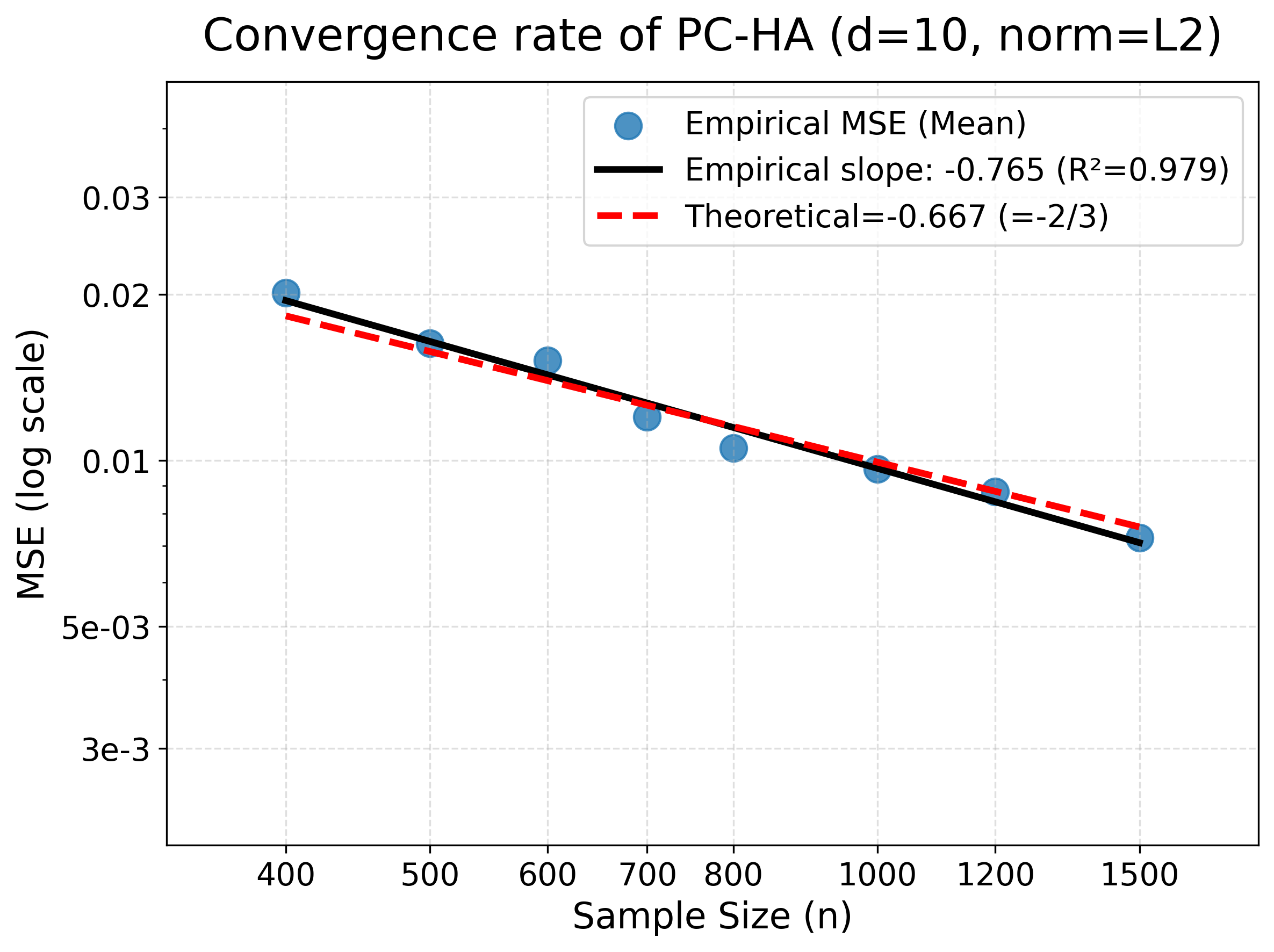}%
				\includegraphics[width=0.3\textwidth]{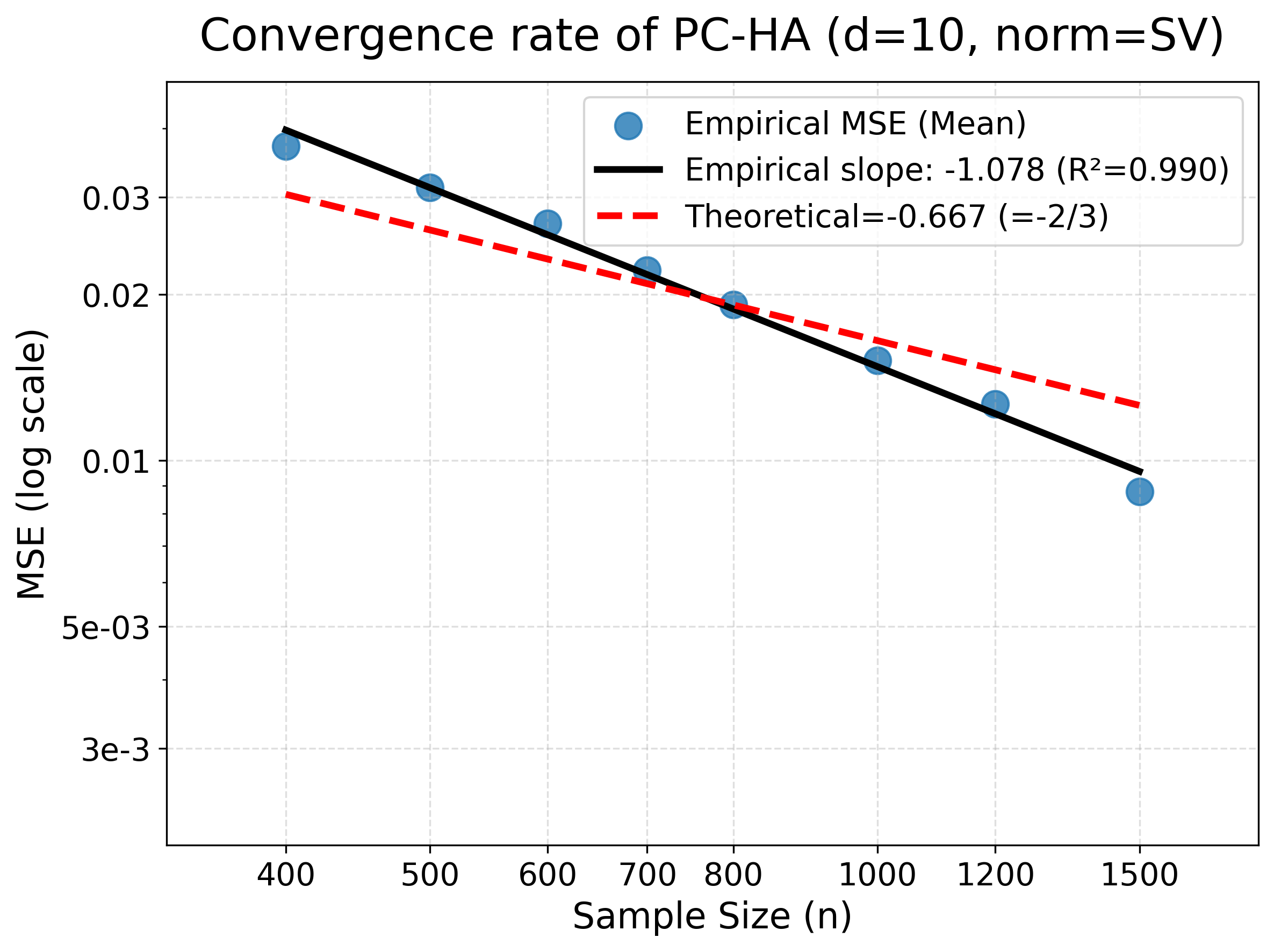}
			\end{minipage}
			\caption{Convergence rate of cross validated PC-HAs for the ``fast'' sinusoidal target function. Note that the slopes are much closer to $-1$.}
			\label{fig:convergence_mean_fast}
		\end{figure}
		
		\clearpage
		\subsection{Scaling of norms and complexity}
		
		\begin{figure}[b!]
			\centering
			\includegraphics[width=1.2\textwidth, height=0.15\textheight, keepaspectratio]{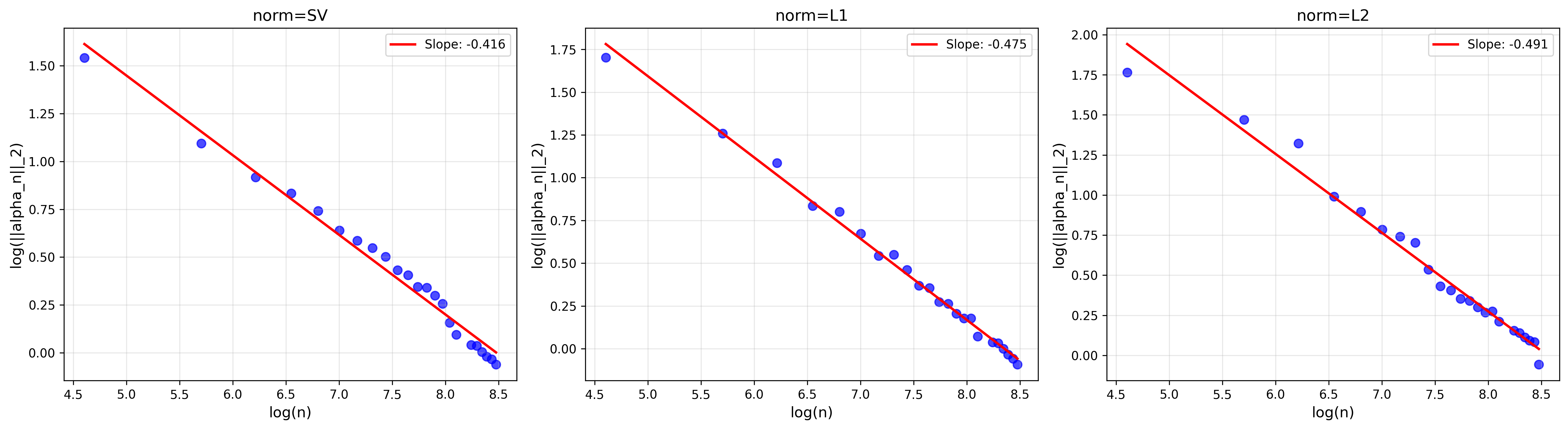}
			\vspace{0.3cm}
			\includegraphics[width=1.2\textwidth, height=0.15\textheight, keepaspectratio]{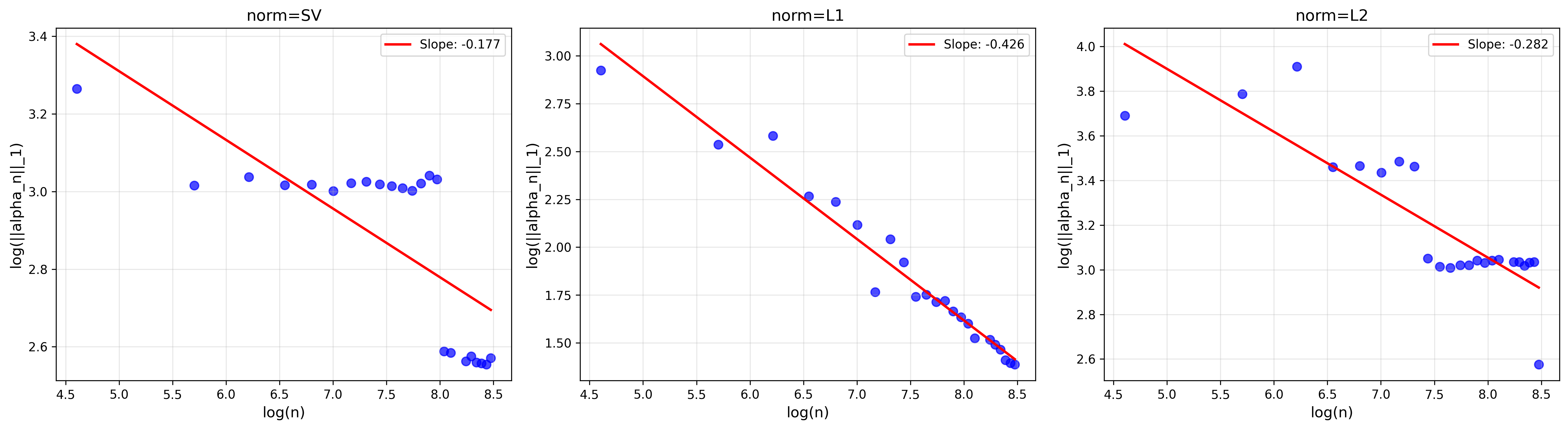}
			\vspace{0.3cm}
			\includegraphics[width=1.2\textwidth, height=0.15\textheight, keepaspectratio]{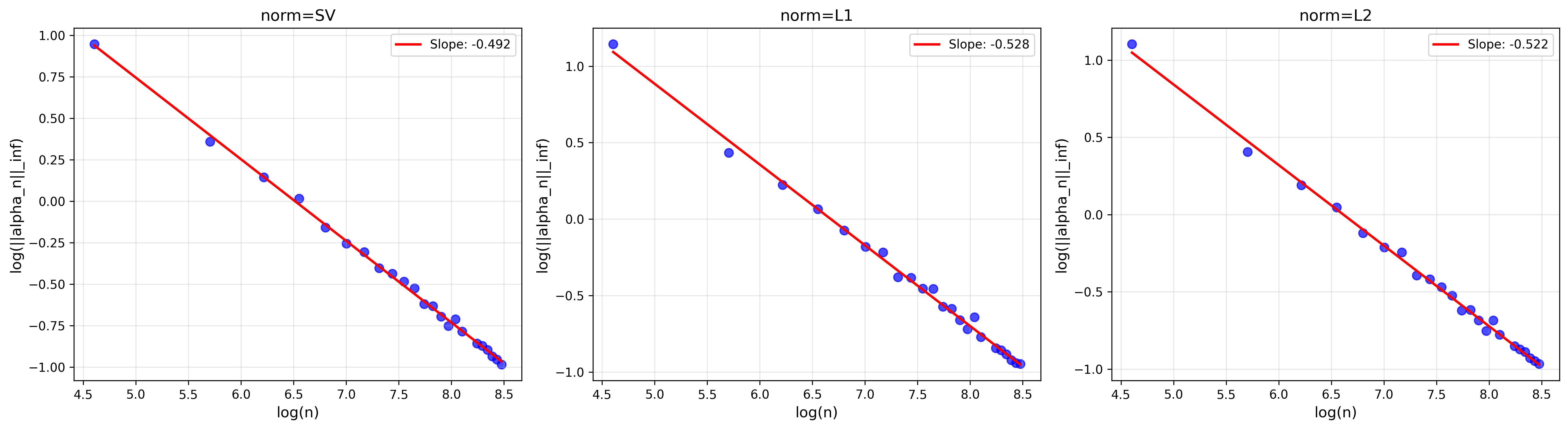}
			\vspace{0.3cm}
			\includegraphics[width=1.2\textwidth, height=0.15\textheight, keepaspectratio]{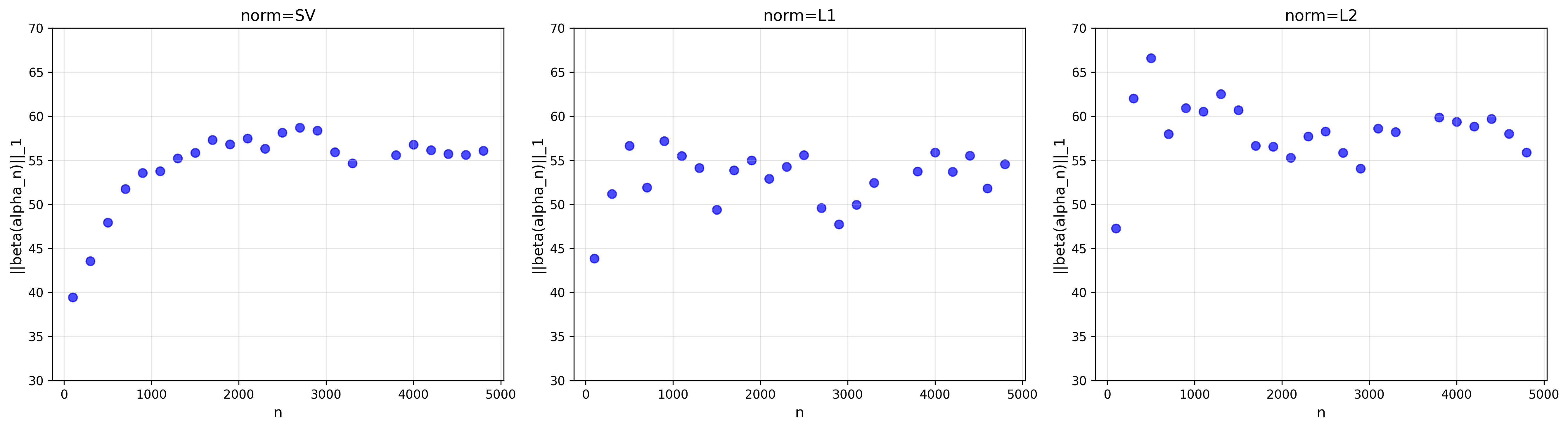}
			\vspace{0.3cm}
			\includegraphics[width=1.2\textwidth, height=0.15\textheight, keepaspectratio]{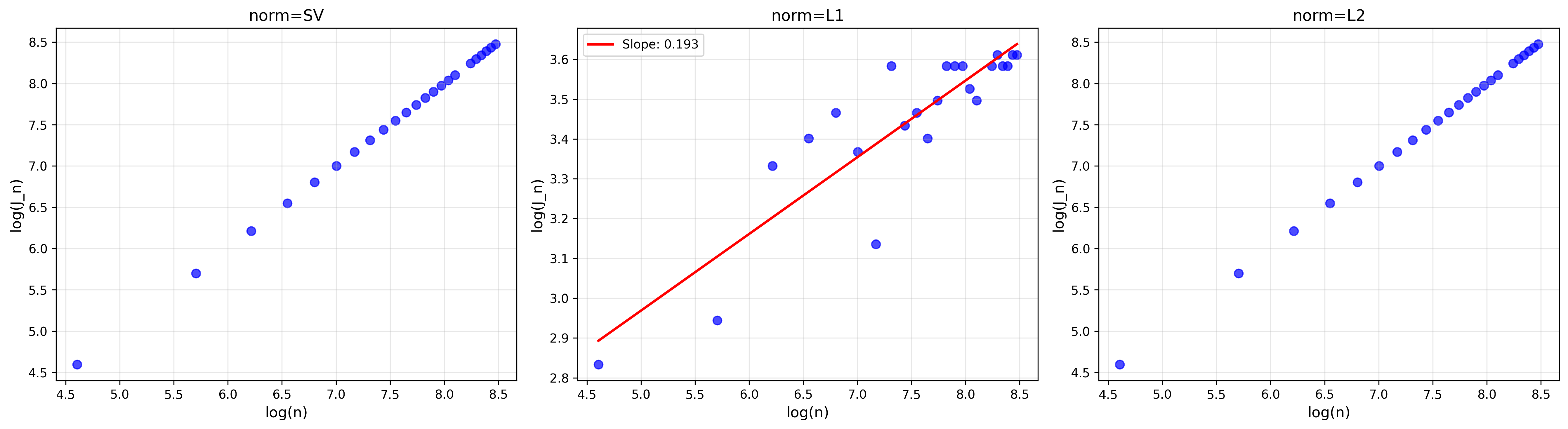}
			\caption{Scaling behavior of different norms and metrics as a function of sample size $n$. Each row corresponds to a different metric: (1) $\|\alpha_n\|_2$ vs $n$, (2) $\|\alpha_n\|_1$ vs $n$, (3) $\|\alpha_n\|_\infty$ vs $n$, (4) $\|\beta(\alpha_n)\|_1$ vs $n$, and (5) $J_n$ (number of selected coefficients) vs $n$. Each column corresponds to a different regularization method (PC-HAGL, PC-HAL, PC-HAR). All plots are log-log.}
			\label{fig:scaling_all}
		\end{figure}
		
		\clearpage
		\subsection{ATE undersmoothing: EIC vs.\ $\lambda$}
		
		\begin{figure}[b!]
			\centering
			\includegraphics[width=1.2\textwidth, height=0.15\textheight, keepaspectratio]{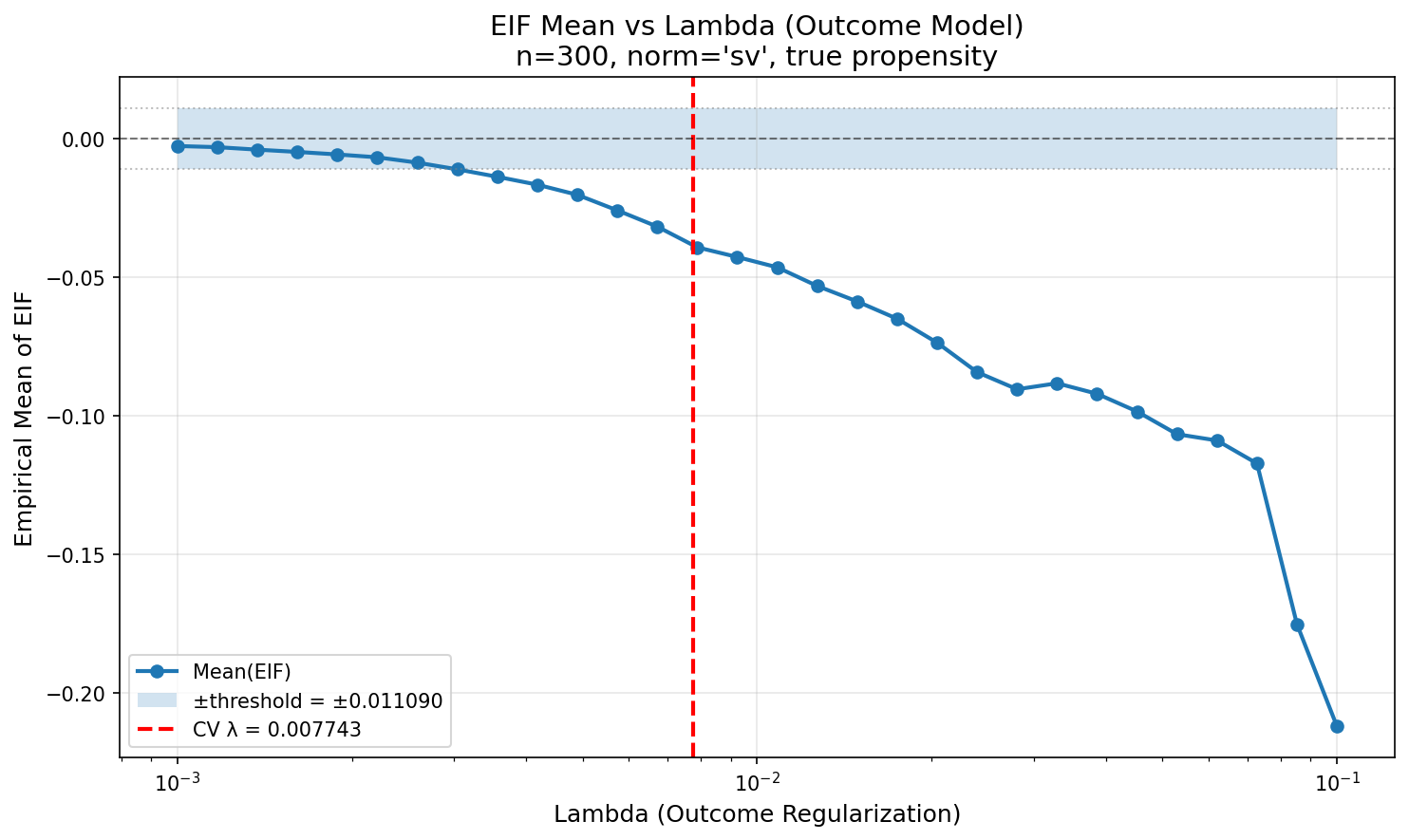}
			\vspace{0.3cm}
			\includegraphics[width=1.2\textwidth, height=0.15\textheight, keepaspectratio]{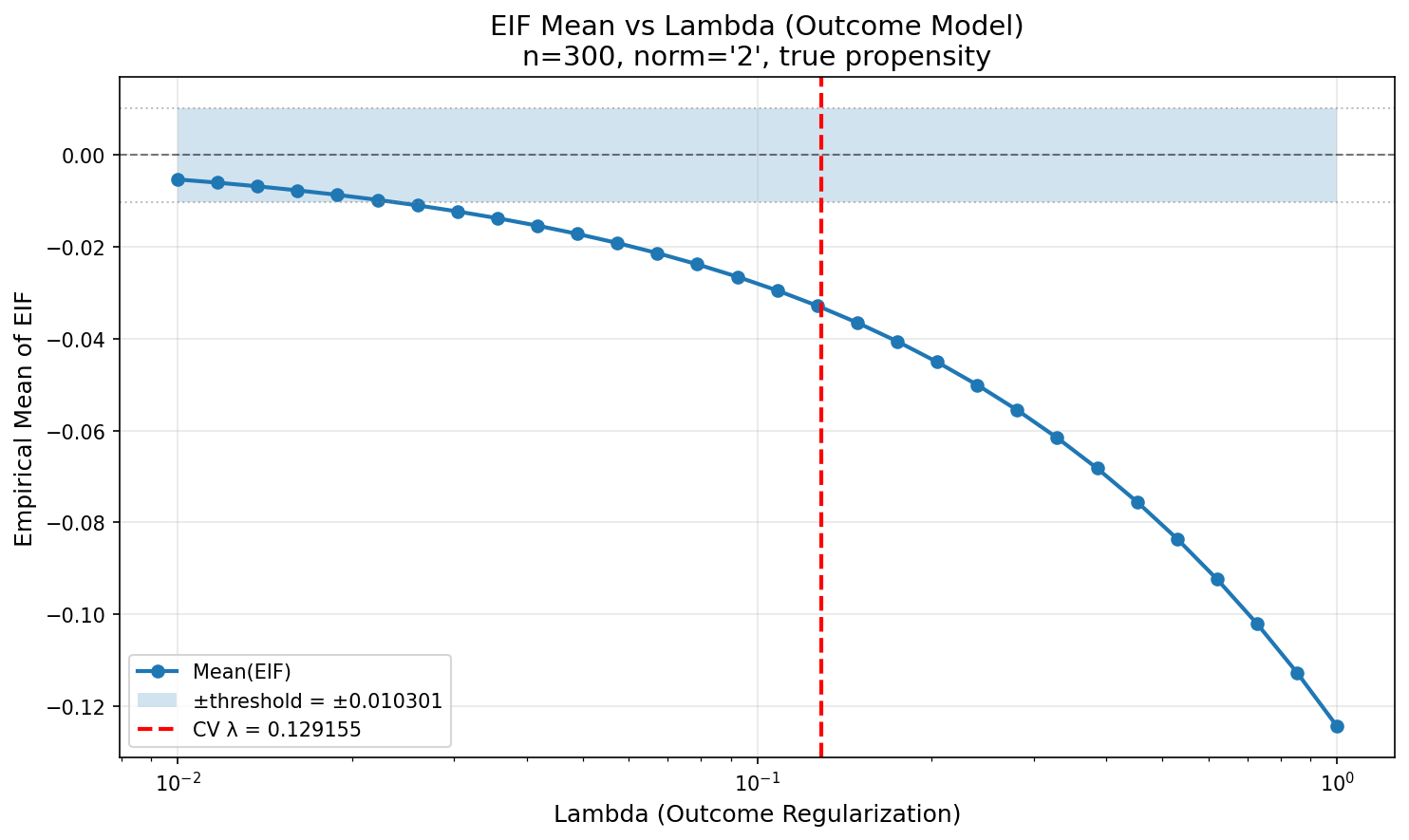}
			\vspace{0.3cm}
			\includegraphics[width=1.2\textwidth, height=0.15\textheight, keepaspectratio]{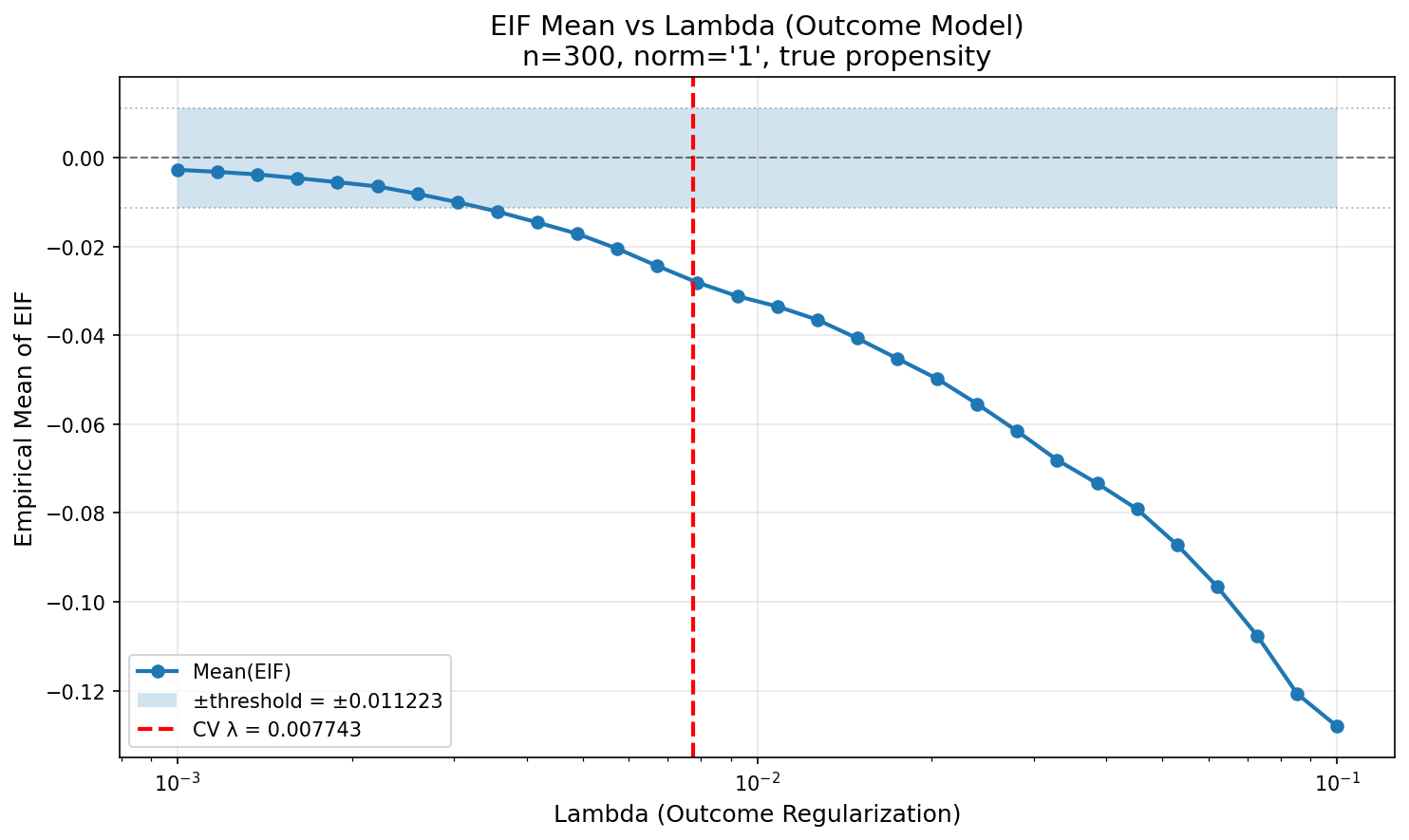}
			\caption{One panel per PC-HA estimator (sectional variation, $L_2$, $L_1$): empirical mean of the efficient influence curve (EIC) versus the outcome regularization parameter $\lambda$. The red vertical line marks the cross-validated choice of $\lambda$ for the outcome fit. The blue curve is the plug-in bias (empirical mean of the EIC) for each value of $\lambda$; the horizontal blue lines correspond to the threshold $\tau$ defined in the main text (Section~\ref{sec:ate}), so that undersmoothing selects the smallest $\lambda$ for which the blue curve lies within $\pm\tau$.}
			\label{fig:ate_eic_lambda}
		\end{figure}
		
		\newpage

		%Note that, given Step 1, Step 2 automatically holds for $\pl \alpha\pl_3=\pl \beta(\alpha)\pl_1$-norm. For the other norms, we already know from Step 1 that $\pl \beta(\alpha_{0,n})\pl_1=O(1)$, so that it is a matter of bounding $\pl \beta(\alpha)\pl_1$ in terms of $\pl \alpha\pl$ so that we know the required $C_{0,n}$, and then one needs to show $\pl \alpha_{0,n}\pl<C_{0,n}$. 
		%For example, if $\pl \beta(\alpha_{0,n})\pl_1>\delta C_{0,n}^{-1} \pl \alpha_{0,n}\pl_1$ for some $\delta>0$, then  $\pl \beta(\alpha_{0,n})\pl_1=O(1)$ implies $\pl \alpha_{0,n}\pl=O(C_{0,n})$, establishing Step 2. 
		%If one can establish a $C_{0,n}$ so that $\pl \beta(\alpha)\pl_1=O(1)$ if and only if $\pl \alpha\pl < C_{0,n}$, then Step 2 automatically holds. 
		
		%\begin{lemma}Given Step 1, a sufficient assumption for Step 2 is that  there exist a rate $r(n)$ so that $\sup_{\alpha} \pl \beta(\alpha)\pl_1/\pl \alpha\pl =O(r(n)^{-1})$ and $\sup_{\alpha} \pl \alpha\pl/\pl \beta(\alpha)\pl_1 =O(r(n))$. Of course, this trivially holds for $\pl \alpha\pl\equiv \pl \beta(\alpha)\pl_1$.Then, $\pl \beta(\alpha_{0,n})\pl_1=O(1)$ implies $\pl \alpha_{0,n}\pl =O(r(n))$ so that Step 2 holds with $C_{0,n}=r(n)$. \end{lemma}
		
		%\section*{???}%% if no title is needed, leave empty \section*{}.
	\end{appendix}
\end{document}